\documentclass{amsart}
\usepackage[T1]{fontenc}
\usepackage{latexsym,amssymb,amsmath}

\usepackage[all,ps]{xy}
\usepackage{amsthm}
\usepackage{longtable}
\usepackage{epsfig}
\usepackage{mathrsfs}
\usepackage{hhline}
\usepackage{epic}
\usepackage{enumerate}

 \newcommand{\res}{\operatorname{res}}

 \newcommand{\Alt}{\mathcal{A}}

 \newcommand{\sym}{\mathfrak{S}}
 \newcommand{\IBr}{\operatorname{IBr}}

 \newcommand{\Ind}{\operatorname{Ind}}
 \newcommand{\Res}{\operatorname{Res}}

 \newcommand{\Cen}{\operatorname{C}}

 \newcommand{\cyc}[1]{\langle\,#1\,\rangle}

 \newcommand{\C}{\mathbb{C}}

 \newcommand{\Z}{\mathbb{Z}}
 
 \newcommand{\Sym}{\mathfrak{S}}

 \newcommand{\pd}{\cdot}
 \newcommand{\indicatrice}{1}
 \newcommand{\Irr}{\operatorname{Irr}}

\newcommand{\cal}[1]{\mathcal{#1}}

\newcommand{\carwr}{\theta}
\newcommand{\carsym}{\chi}
\newcommand{\caralt}{\rho}
\newcommand{\mm}{\boldsymbol{\mu}}
\newcommand{\pp}{\boldsymbol{\pi}}

\newcommand{\fleche}{\rightsquigarrow}

 \newcommand{\dis}{\displaystyle}

 \newcommand{\la}{\lambda}

 \newcommand{\da}{\delta}
  
  \newcommand{\sa}{\sigma}

 \newcommand{\ga}{\gamma}
 \newcommand{\e}{\varepsilon}
 
  \newcommand{\p}{\bar{p}}
  \newcommand{\q}{\bar{q}}

\newcommand{\cla}{\xi_{\la}}

\newcommand{\pla}{\zeta_{\la}}
\newcommand{\Dla}{\Delta_{\la}}

\newcommand{\cmu}{\xi_{\mu}}
\newcommand{\pmu}{\zeta_{\mu}}
\newcommand{\Dmu}{\Delta_{\mu}}

\newcommand{\tSym}{\widetilde{\Sym}}
\newcommand{\tAlt}{\widetilde{\Alt}}

\newtheorem{theorem}{Theorem}[section] 
\newtheorem{lemma}[theorem]{Lemma}     
\newtheorem{corollary}[theorem]{Corollary}
\newtheorem{proposition}[theorem]{Proposition}

\newtheorem{convention}[theorem]{Convention}
\newtheorem{definition}[theorem]{Definition}
\newtheorem{example}[theorem]{Example}
\theoremstyle{definition}
\newtheorem{remark}[theorem]{Remark}

\title[]
{Perfect isometries and Murnaghan-Nakayama rules}
\author{Olivier Brunat}

\address{Universit\'e Paris-Diderot Paris 7\\ Institut de math\'ematiques de
         Jussieu -- Paris Rive Gauche\\ UFR de math\'e\-matiques\\ Case
7012\\ 75205 Paris Cedex 13\\
         France.}
\email{brunat@math.univ-paris-diderot.fr}

\author{Jean-Baptiste Gramain}

\address{Institute of Mathematics, 
University of Aberdeen, King's College \\
Fraser Noble Building, Aberdeen AB24 3UE, UK
}
\email{jbgramain@abdn.ac.uk}


\begin{document}
\begin{abstract}
This article is concerned with perfect isometries between blocks of
finite groups. Generalizing a method of Enguehard to show that any two $p$-blocks of (possibly different)
symmetric groups with the same weight are perfectly isometric, we prove analogues of this result for $p$-blocks of alternating groups 
(where the blocks must also have the same sign when $p$ is odd), of double covers of 
alternating and symmetric groups (for $p$ odd, and where we obtain crossover isometries when the blocks have opposite signs), 
of complex reflection groups $G(d,1,n)$ (for $d$ prime to $p$), of Weyl
groups of type $B$ and $D$ (for $p$ odd), and of certain wreath products.
In order to do this, we need to generalize the theory of blocks, in a way which should be of
independent interest.
\end{abstract}
\maketitle

\section{Introduction}

Perfect isometries, introduced by M. Brou\'e in \cite{Broue}, are the
shadow, at the level of characters, of very deep structural
correspondences between blocks of finite groups (such as derived
equivalences, or splendid equivalences).  The existence of such
equivalences is at the heart of Brou\'e's Abelian Defect Conjecture,
which predicts that any $p$-block of a finite group $G$ with abelian defect
group $P$ and its Brauer correspondent in $\operatorname{N}_G(P)$ are
derived equivalent.  

Recently, there has been considerable progress in the construction of
equivalences between blocks, especially using a method, introduced and
developed by J. Chuang and R. Rouquier in~\cite{Chuang-Rouquier}, and
based on $\mathfrak{sl}_2$-categorification. As a consequence of their
work, they show that two $p$-blocks of (possibly different) symmetric
groups with isomorphic defect groups are splendidly equivalent;
see~\cite[Theorem 7.2]{Chuang-Rouquier}. This explains the result~\cite[Theorem 11]{Enguehard} of M.
Enguehard, which is an analogue of
\cite[Theorem 7.2]{Chuang-Rouquier}, but at the level of characters,
that is, the existence of Brou\'e perfect isometries between such
blocks.

For $p$-blocks of (possibly different) double covers of the symmetric
and alternating groups, it has been conjectured  by M. Schaps and R.
Kessar that, with some additional assumptions, a similar result holds.
There are partial results in this direction, for example~\cite{Kessar1},
\cite{Kessar-Schaps} and~\cite{Leabovich-Schaps}. However, even at the
level of characters, the existence of perfect isometries between these
$p$-blocks was yet unproved.

This article dicusses perfect isometries. Besides suggesting the
existence of a derived equivalence between blocks, any perfect isometry
between two $p$-blocks of finite groups provides an isomorphism between
their centres, and an isomorphism between the Grothendieck groups of
their module categories. In particular, perfectly isometric $p$-blocks
have the same numbers of ordinary and of modular characters, and their
Cartan matrices and decomposition matrices have the same invariant
factors.

Furthermore, the weaker version of Brou\'e's Abelian Defect Conjecture
(that is, Brou\'e's conjecture at the level of characters) gives, in the
abelian defect case, deep insight into more numerical conjectures, such
as the Alperin, Kn\"{o}rr-Robinson, Alperin-McKay and Dade conjectures
(see for example \cite{Eaton}).  

In this paper, we generalize Enguehard's method (see \cite{Enguehard})
based on the Murnaghan-Nakayama rule in the symmetric group (which gives
a way to compute iteratively the values of irreducible complex
characters).  We will prove that similar results hold for many classes
of groups where some analogues of the Murnaghan-Nakayama rule are
available.

For this, we extract the properties of the Murnaghan-Nakayama rule
needed in Enguehard's method, which we axiomatize in the concept of
an {\emph{MN-structure}} for a finite group. In some cases, for example
when the analogue of the Murnaghan-Nakayama rule for the considered
groups do not give information on the whole group, but only on certain
conjugacy classes (this happens for the double covers of the symmetric
and alternating groups), we need to replace the set of $p$-regular
elements of the group by an arbitrary union of conjugacy classes.  We
then develop a generalized modular theory, and define generalized blocks
and generalized perfect isometries. Note that the notion of generalized
blocks and generalized perfect isometries introduced by B.
K\"{u}lshammer, J. B. Olsson and G. R. Robinson in \cite{KOR} is not
exactly the same as ours. In some way, our notion is more general,
because any K\"{u}lshammer-Olsson-Robinson isometry or Brou\'e isometry
is a generalized perfect isometry in our sense.

The article is organized as follows. In Section~\ref{section:generalite}
we generalize the theory of blocks of characters. Note that \S\ref{pargen} is
of independent interest, because it in particular gives
a natural framework to use the techniques of the usual
modular $p$-block theory for the theory
of K\"{u}shammer, Olsson and Robinson.
The main result of this section (Theorem
\ref{th:iso}) provides the bridge necessary to compare blocks and spaces
of class functions of (possibly distinct) groups which have similar
MN-structures. This {\emph{combinatorification}} of the ideas in
\cite{Enguehard} can in turn be used to exhibit perfect isometries
between blocks of these groups (see Corollary \ref{cor:isoparfaitegene}
and Theorem \ref{th:broue}).

The remaining sections are devoted to describing MN-structures in
several families of finite groups, and using our methods to build
explicitly perfect isometries between their blocks. 

More precisely, we
prove in Section~\ref{section:An} that two $p$-blocks of (possibly
different) alternating groups with  same weight, (and the same
{\emph{signature type}} when $p$ is odd) are perfectly isometric (see
Theorems \ref{theo:mainAn}, \ref{theo:mainAn2} and
\ref{theo:mainAnpegal2}).  

Then, in Section~\ref{section:SnTilde}, we study the case of spin blocks
of the double covers of the symmetric and alternating groups, and  we
prove the perfect isometry version of the Kessar-Schaps conjecture.  We
show that, when $p$ is odd, any two spin $p$-blocks with the same weight
and sign are perfectly isometric (see Theorem \ref{theo:mainTilde} and
Corollary~\ref{cor:brouetilde}).  As
is to be expected in these groups (see~\cite{Kessar-Schaps}), we also
obtain {\emph{crossover}} isometries, relating a $p$-block in ``the
symmetric case'' to a $p$-block in ``the alternating case''. Note that,
in the proof of these results, even though the isometries we obtain are
Brou\'e isometries, we crucially need the generalized theory introduced
in Section~\ref{section:generalite}.

In the last section, we examine the case of certain wreath products.
Applying our method, we give in~\S\ref{subsec:sym}
and~\S\ref{subsec:osima} a new and more uniform construction of the
isometries appearing in Brou\'e's Abelian Defect Conjecture for
symmetric groups, isometries introduced by M. Osima, and the generalized
perfect isometry considered in~\cite{BrGr} in order to show the
existence of $p$-basic sets for the alternating group (see
Theorem~\ref{theo:BrGr}, Theorem~\ref{theo:osima} and
Corollary~\ref{theo:BroueSn}).  Even though these results are not new,
they give explicit isometries, and considerably simplify the
calculations (for example, note that the initial proof of
Rouquier~\cite{Rouquier} of Brou\'e's perfect isometries Conjecture for
symmetric groups (see~\cite{Rouquier}) is not constructive, and is based
on a strong result of Fong and Harris in~\cite{Fong-Harris2} on perfect
isometries in wreath products).

In \S\ref{subsec:wreath}, we apply our method to $p$-blocks of complex
reflection groups $G(d,1,n)$ with $d$ prime to $p$, and obtain in
Theorem~\ref{theo:couronne} an analogue of Enguehard's result for these
groups. In particular, this gives the result for $p$-blocks (with $p$
odd) of (possibly different) Weyl groups of type $B$ (see Corollary
\ref{cor:broueBn}). In~\S\ref{subsec:typeD}, we also prove the result
for $p$-blocks (with $p$ odd) of (possibly different) Weyl groups of
type $D$ (Theorems \ref{theo:broueDnconj} and
\ref{theo:broueDnpasconj}). All of these are new results.

Finally, in \S\ref{subsec:fh}, we give an analogue of the generalized
perfect isometry of~\cite[Thoerem 3.6]{BrGr} for $p$-blocks (with $p$ odd) of
alternating groups (see Therorem~\ref{theo:fh}). In a certain sense (see
Example~\ref{ex:osima}), this is a natural analogue of Osima's
isometry for alternating groups. When the $p$-block of the
alternating group has abelian defect, our result gives an alternative
proof of Brou\'e's perfect isometries Conjecture first obtained by Fong
and Harris in~\cite{fongharris} (see Corollary~\ref{cor:fh}).

We hope that our results, and in particular the fact that the Brou\'e 
perfect isometries constructed here are explicit, 
will help to prove that the
corresponding $p$-blocks are in fact derived equivalent.

\section{Generalities}
\label{section:generalite}

In this section, $G$ denotes a finite group and $\mathcal C$ a set of
conjugacy classes of $G$. We set
\begin{equation}
\label{eq:ele}
C=\bigcup_{c\in\mathcal C}c.
\end{equation}
We write $\Irr(G)$ for the set of irreducible
characters of $G$ over the complex field $\C$, and $\cyc{\,,}_G$ for the
usual hermitian product on $\C\Irr(G)$. For $x\in G$, we
denote by $x^G$ the conjugacy class of $x$ in $G$. Define
$\res_C:\C\Irr(G)\rightarrow \C\Irr(G)$ by setting, for any class function
$\varphi\in\C\Irr(G)$,
$$\res_C(\varphi)(g)=
\left\{
\begin{array}{ccl}
\varphi(g) &  & \textrm{if }g\in C,\\
0 & & \textrm{otherwise}.
\end{array}\right.$$
For $B\subseteq\C\Irr(G)$, we set
$B^{\mathcal C}=\{\res_C(\chi)\ |\ \chi\in B\}$.

\subsection{Generalized modular theory}\label{pargen}
Let $b$ be a $\Z$-basis of the $\Z$-module $\Z\Irr(G)^{\mathcal C}$.
For every $\chi\in\Irr(G)$,  there are uniquely determined integers
$d_{\chi\varphi}$ such that
\begin{equation}\label{eq:decnumber}
\res_C(\chi) = \sum_{\varphi\in
b}d_{\chi\varphi}\varphi.
\end{equation}
We denote by $b^{\vee}$
the dual basis of $b$ with respect to $\cyc{\,,}_G$, i.e. the unique
$\C$-basis $b^{\vee}=\{\Phi_{\varphi}\, |\,\varphi\in b\}$ of
$\C\Irr(G)^{\mathcal C}$ such that
$\cyc{\Phi_{\varphi},\vartheta}=\delta_{\vartheta\varphi}$ for all
$\vartheta\in b$.


\begin{proposition}Let $\mathcal C$ be a set of conjugacy classes of
$G$. Suppose that $b$ is a $\Z$-basis of $\Z\Irr(G)^{\mathcal
C}$,
and denote by $b^{\vee}=\{\Phi_{\varphi}\, |\,\varphi\in b\}$ the dual basis of $b$ with respect to
$\cyc{\,,}_G$ (as above). Then:
\begin{enumerate}
\item[(i)] For every $\varphi\in b$, we have 
\begin{equation}
\label{eq:pim}
\Phi_{\varphi}=\sum_{\chi\in\Irr(G)}d_{\chi\varphi}\chi =\sum_{\chi\in\Irr(G)}d_{\chi\varphi}\res_C(\chi),
\end{equation}
where the $d_{\chi\varphi}$'s are the integers defined in
Equation~(\ref{eq:decnumber}).
\item[(ii)] We have $$\Z\Irr(G)\cap \Z\Irr(G)^{\mathcal
C}=\Z b^{\vee}.$$
\end{enumerate}
\label{prop:mod}
\end{proposition}

\begin{proof}Let $\varphi\in b$. We have
$\Phi_{\varphi}\in\C\Irr(G)^{\mathcal C}$. It follows that
$\cyc{\Phi_{\varphi},\chi}_G=\cyc{\Phi_{\varphi},\res_C(\chi)}_G$
for all $\chi\in\Irr(G)$. Using Equation~(\ref{eq:decnumber}), we deduce
that 
\begin{eqnarray*}\cyc{\Phi_{\varphi},\chi}_G&=&\sum\limits_{\vartheta\in
b}d_{\chi\vartheta}\cyc{\Phi_{\varphi},\vartheta}_G\\
&=&d_{\chi\varphi}.
\end{eqnarray*}
This proves (i).

By (i), we clearly have $\Z b^{\vee}\subseteq \Z\Irr(G)\cap\Z
\Irr(G)^{\mathcal C}$. Conversely, suppose that $\psi$ is a generalized
character vanishing on the elements $x$ such that $x^G\notin\mathcal C$.
Then
$$\psi=\sum_{\varphi\in b}\cyc{\psi,\varphi}_G\,\Phi_{\varphi}.$$
Since $b\subset\Z\Irr(G)^{\mathcal C}$, for every $\varphi\in b$, there
are integers $a_{\varphi\chi}$ (not necessarily unique) such that 
$$\varphi=\sum_{\chi\in\Irr(G)}a_{\varphi\chi}\res_C(\chi).$$
Define 
$$\psi_{\varphi}=\sum_{\chi\in\Irr(G)}a_{\varphi\chi}\chi\in\Z\Irr(G).$$
Then $\res_C(\psi_{\varphi})=\varphi$. Moreover,
$\psi\in\Z\Irr(G)^{\mathcal C}$. It follows that
$$\cyc{\psi,\varphi}_G=\cyc{\psi,\psi_{\varphi}}_G,$$
which is an integer 
because $\psi\in\Z\Irr(G)$ and (ii) holds.
\end{proof}

Now, we introduce a graph as follows. The vertex set is $\Irr(G)$ and
two vertices $\chi$ and $\chi'$ are linked by an edge, if there is
$\varphi\in b$ such that
$d_{\chi\varphi}\neq 0$ and $d_{\chi'\varphi}\neq 0$. The connected components of this graph are called the $\mathcal C$-blocks of
$G$. 

\begin{remark}
\label{rk:uniciteblock}
Note that the $\mathcal C$-blocks of $G$ depend on the choice of the
$\Z$-basis $b$ of $\Z\Irr(G)^{\mathcal C}$.
\end{remark}

If $B$ is a union of $\mathcal C$-blocks of $G$, we write $\Irr(B)$
for the subset of $\Irr(G)$ corresponding to the vertices
of $B$, and $b_B$ for the set of elements of $b$ which give edges
in $B$. We set $b_B^{\vee}=\{\Phi_{\varphi}\,|\,\varphi\in b_B\}$. Note that
$b_B^{\vee}$ is the dual basis of $b_B$ (when $b_B$ is viewed as a basis
of the $\C$-vector space $\C b_B$) with respect to $\cyc{\,,}_G$.

We may (and do) order the elements of $\Irr(G)$ and $b$ in such a way that, if the rows and columns of $D=(d_{\chi\varphi})_{\chi\in\Irr(G),\varphi\in b}$ are ordered correspondingly, then $D$ is a block-diagonal matrix, and
each (diagonal) block $D_B$ of $D$ corresponds to a $\mathcal C$-block $B$ of $G$.

\begin{corollary}\label{blocs}
With the notation as above, for every $\mathcal C$-block $B$ of $G$, we
have $\Phi_{\varphi}=\sum\limits_{\chi\in \Irr(B)}d_{\chi\varphi}\chi$
for all $\varphi\in b_B$, and
$$\Z\Irr(B)\cap\Z\Irr(G)^{\mathcal C}=\Z b_B^{\vee}.$$
\end{corollary}

\begin{corollary}\label{critere}
With the above notation, let $\chi,\,\psi\in \Irr(G)$ and
$\varphi,\,\vartheta \in b$ be such that $\cyc{\varphi,\vartheta}_G\neq 0$ and
$d_{\chi\varphi}\neq 0\neq d_{\psi\vartheta}$. Then $\chi$ and $\psi$
lie in the same $\mathcal{C}$-block.
\end{corollary}  

\begin{proof}
Let $\varphi,\,\theta\in b$. By Proposition~\ref{prop:mod}(i), we have
\begin{eqnarray*}\delta_{\varphi\,\theta}=\cyc{\Phi_{\varphi},\theta}_G & =& 
\sum_{\chi\in\Irr(G)}d_{\chi\varphi}\cyc{\res_C(\chi), \theta }_G \\  & = &
\sum_{\chi\in\Irr(G)}\left (\sum_{\eta\in
b}d_{\chi\varphi}d_{\chi\eta}\right)\cyc{\eta,\theta}_G \\ & = &
\sum_{\eta\in
b}\left (\sum_{\chi\in\Irr(G)}d_{\chi\varphi}d_{\chi\eta}\right)\cyc{\eta,
\theta}_G.\end{eqnarray*}
Now, if we write $K=(\cyc{\varphi,\theta}_G)_{\varphi,\theta\in b}$, then the
preceding equation gives $I=^t\!\!DDK$. Thus, $K$ is invertible and
$K^{-1}=^t\!\!DD$.  Furthermore, $D$ is a block-diagonal matrix. Hence,
$K^{-1}$ also has a block-diagonal structure. More precisely, the blocks
of $K^{-1}$ are the $^t\!D_BD_B$'s for all $\mathcal C$-blocks $B$ of $G$. It
follows that $K$ has the same block-diagonal structure as $K^{-1}$. In
particular, if $\cyc{\varphi,\theta}_G\neq 0$, then $\varphi$ and $\theta$ lie
in the same $\mathcal C$-block of $G$.

Our assumption that $\cyc{\varphi,\vartheta}_G\neq 0$ therefore implies that $\vartheta$ and
$\varphi$ lie in a common $\mathcal C$-block $B$ of $G$. By the definition of $\mathcal C$-blocks, this means that $\varphi$ and
$\vartheta$ correspond to some subsets $c_{\varphi}$ and $c_{\vartheta}$ of edges in a
connected component $B$ of the graph previously introduced. Moreover,
$\chi$ (respectively $\psi$) is a vertex of some edge in $c_{\varphi}$ (respectively in
$c_{\vartheta}$), because $d_{\chi\varphi}\neq 0$ (respectively
$d_{\psi\vartheta}\neq 0$). Therefore $\chi,\,\psi\in B$.
\end{proof}

\subsection{MN-Restriction}
We fix a set of $G$-conjugacy classes $\mathcal C$ and a union of
$\mathcal C$-blocks $B$ of $G$, and denote by $C$ the corresponding set
of elements as in Equation~(\ref{eq:ele}). 

\begin{definition}\label{defMN}
We say that $G$ has an MN-structure with respect to $\mathcal C$ and $B$, 
if the following properties hold.
\begin{enumerate}
\item[1.] There is a subset $S\subseteq G$ containing $1$ and stable under $G$-conjugation.
\item[2.] There is a bijection between a subset $A\subseteq S\times C$
and $G$ (the image of $(x_S,x_{C})\in A$ will be denoted by
$x_S\pd x_{C}$), such that for $(x_S,x_{C})\in A$ 
$$^g(x_S\pd x_{C})=({}^gx_S)\pd ({}^gx_{C})\quad\textrm{
and}\quad
x_{S}\pd x_C=x_Sx_{
C}=x_{C}x_S.$$
Moreover, for all $x_S\in S$ and $x_C\in C$, we have $(x_S,1)\in A$ and
$(1,x_C)\in A$.
\item[3.] 
For $x_S\in S$, there is a subgroup $G_{x_S}\leq \Cen_G(x_S)$ such that
$$G_{x_S}\cap C=\{x_C\in C\ |\ (x_S,x_C)\in A\}.$$
For $x_S\in S$, we denote by $\mathcal C_{x_S}$ the set of $G_{x_S}$-conjugacy classes of $G_{x_S}\cap C$. 
\item[4.]
For $x_S\in S$, there is a union of $\mathcal C_{x_S}$-blocks $B_{x_S}$ of
$G_{x_S}$ and a 
homomorphism
$r^{x_S}:\C\Irr(B)\rightarrow \C\Irr(B_{x_S})$ satisfying
$$r^{x_S}(\chi)(x_C)=\chi(x_S\pd x_C)\quad\textrm{for all }\chi\in\C
\Irr(B)\ \textrm{and }(x_S,x_C)\in A.$$
Moreover, we assume that $G_1=G$, $B_1=B$ and $r^1=\operatorname{id}$.
\end{enumerate}
\end{definition}

%

In the rest of this subsection, we suppose that $G$ has an MN-structure.
For $x_S\in S$, we define a homomorphism 
$d_{x_S}:\C \Irr(B)\rightarrow \C\Irr(B_{x_S})^{\mathcal C_{x_S}}$ 
by setting
\begin{equation}
d_{x_S}(\chi)=\res_C \circ r^{x_S}(\chi)\quad\textrm{for }
\chi\in\Irr(B).
\label{eq:dlambda}
\end{equation}

The $\C$-basis of $\C\Irr(B_{x_S})^{\mathcal C_{x_S}}$ used to define
the union of
$\mathcal C_÷{x_S}$-blocks $B_{x_S}$ of $G_{x_S}$ (see
Remark~\ref{rk:uniciteblock}) is denoted by $b_{x_S}$,
and we write
$$e_{x_S}:\C
b_{x_S}^{\vee}\rightarrow\C\Irr(B)$$ for the adjoint map of
$d_{x_S}$ with respect to $\cyc{\,,}_G$.

Take any $y=y_S\pd y_C\in G$ with $y_S\in x_S^G$. For any $t\in G$ such that
${}^ty_S=x_S$, one has ${}^ty_C\in G_{x_S}$ by
Definition~\ref{defMN}(3), and  
the set $X$ of elements
${}^ty_C$ with $t\in G$ such that ${}^ty_S=x_S$ is stable under
$G_{x_S}$-conjugation (because $G_{x_S}\subseteq \Cen_G(x_S)$). 
We denote by $\mathcal E_{y_S,y_C}^{x_S}$ a set of representatives of 
the $G_{x_S}$-classes of $X$.

\begin{lemma}
With the notation above, for any $\phi\in\C b_{x_S}^{\vee}$ and
$(y_S,y_C)\in A$, we have
$$e_{x_S}(\phi)(y_{S}\pd y_{C})=\left\{
\begin{array}{ccl}
|\Cen_G(y_S\pd y_C)|\displaystyle\sum\limits_{x_C\in \mathcal E_{y_S,y_C}^{x_S}
}\dfrac{\phi(x_C)}{|\Cen_{G_{x_S}}(x_C)|} & & \textrm{if } \; y_S^G=x_S^G, \\
0 & & \textrm{otherwise}.
\end{array}\right.$$
\label{valeurE}
\end{lemma}

\begin{proof}
We denote by $\indicatrice_{G,x}$ the indicator function of the conjugacy class of $x$ in $G$. We have
\begin{eqnarray*}
e_{x_S}(\phi)(y_S\pd y_C)&=&|\Cen_{G}(y_S\pd
y_C)|\cyc{e_{x_S}(\phi),\indicatrice_{G,y_S\pd y_C}}_G\\
&=&|\Cen_G(y_S\pd
y_C)|\cyc{\phi,d_{x_S}(1_{G,y_S\pd y_C})}_{G_{x_S}},
\end{eqnarray*}
because $d_{x_S}$ and $e_{x_S}$ are adjoint.
Moreover, by Definition~\ref{defMN}(3),
we deduce that
$$d_{x_S}(\indicatrice_{G,y_S\pd y_C})=\sum_{x_C\in\mathcal
E_{y_S,y_C}^{x_S}}\indicatrice_{G_{x_S},
x_C}$$ if $y_S^G=x_S^G$ and $0$ otherwise. This implies in particular
that $e_{x_S}(\phi)(y_S\pd y_C)=0$ whenever $y_S^G\neq x_S^G$.
Now, suppose that $y_S^G=x_S^G$. Then
\begin{eqnarray*}
e_{x_S}(\phi)(y_S\cdot y_C)&=&|\Cen_G(y_S\pd
y_C)|\sum_{x_C\in \mathcal E_{y_S,y_C}^{x_S}}\cyc{\phi,\indicatrice_{G_{x_S},x_C}}_{G_{x_S}}\\
&=&|\Cen_G(y_S\pd y_C)|\sum_{x_C\in \mathcal E_{y_S,y_C}^{x_S}}\frac{\phi(x_C)}{|\Cen_{G_{x_S}}(x_C)|},
\end{eqnarray*}
as required.
\end{proof}

By Definition~\ref{defMN}(1), $S= \bigcup_{ \lambda \in \Lambda}
\lambda$, where each $\lambda \in \Lambda$ is a conjugacy class of $G$.
For each $\lambda\in\Lambda$, we choose a representative
$x_{\lambda}\in\lambda$, and we let $G_{\lambda}=G_{x_{\lambda}}$,
$B_{\lambda}=B_{x_{\lambda}}$, $r^{\lambda}=r^{x_{\lambda}}$, $\mathcal
C_{\lambda}=\mathcal
C_{x_S}$ and $d_{\lambda}=d_{x_{\lambda}}$. 

Let $g=g_S\pd g_C\in G$. In the following, we say that $g$ is of type 
$\lambda$ if $g_S\in\lambda$. Furthermore, we set $\mathcal
E_{g_S,g_C}^{\lambda}=\mathcal E_{g_S,g_C}^{x_{\lambda}}$.

Now, we set $$d_G:\C \Irr(B)\rightarrow \bigoplus_{\lambda\in\Lambda}
\C\Irr(B_{\lambda})^{\mathcal C_{\lambda}},\quad\chi\mapsto
\sum_{\lambda\in\Lambda}d_{\lambda}(\chi).$$

For $\lambda\in\Lambda$,
we define
$l_{\lambda}:\C\Irr(G_{\lambda})^{\mathcal C_{\lambda}}\rightarrow \C\Irr(G)$ by
setting
\begin{equation}
l_{\lambda}(\psi)(g)=\left\{
\begin{array}{l}\dfrac{1}{|\mathcal E_{g_S,g_C}^{\lambda}|}\sum\limits_{x_C\in
\mathcal E_{g_S,g_C}^{\lambda}}\psi(x_C)\quad\textrm{if }g_S^G=\lambda\\
0\quad\textrm{otherwise}
\end{array}\right.,
\label{eq:llambda}
\end{equation}
and put 
$$l_G:\bigoplus_{\lambda\in\Lambda}\C\Irr(G_{\lambda})^{\mathcal
C_{\lambda}}\rightarrow\C\Irr(G),\quad
\sum_{\lambda\in\Lambda}\psi_{\lambda}\mapsto\sum_{\lambda\in\Lambda}l_{\lambda}(\psi_{\lambda}).$$

\begin{remark}
\label{rk:fonctionCstable}
Let $\psi\in\C\Irr(G_{\lambda})^{\mathcal
C_{\lambda}}$, and suppose that $\psi$ is constant on $\mathcal
E_{x_{\lambda},y}^{\lambda}$ for every
$y\in C\cap G_{\lambda}$. 
Then $l_{\lambda}(\psi)(g)=0$ except when
$g_S\in\lambda$. In this case, we have $l_{\lambda}(\psi)(g)=\psi(x_C)$,
where $x_C$ is any element of $\mathcal E_{g_S,g_C}^{\lambda}$.
\end{remark}

\begin{lemma}
The homomorphism $d_{G}$ is injective, and the map $l_G\circ
d_G$ is the identity on $\C\Irr(B)$.
\label{applicationD}
\end{lemma}
\begin{proof}
%
Let $x\in G$. Then by Definition~\ref{defMN}(2), $x$ is $G$-conjugate
to $x_{\lambda}\pd x_C$ for some $\lambda\in \Lambda$ and $x_C\in
C$, and for any $\chi\in\C\Irr(B)$, we then have
\begin{eqnarray*} \chi(x)=\chi(x_{\lambda}x_C) & = &
r^{\lambda}(\chi)(x_C)  \qquad \mbox{(by Definition~\ref{defMN}(4))} \\ & = &
\res_C(r^{\lambda}(\chi))(x_C) \\ & = & d_{\lambda}(\chi)(x_C) \qquad \mbox{(by definition of} \; d_{\lambda}) . \end{eqnarray*}
Now, fix
$\chi\in\C\Irr(B)$ such that $d_G(\chi)=0$. Then, for every
$\lambda'\in\Lambda$, we have $d_{\lambda'}(\chi)=0$. In particular,
$d_{\lambda}(\chi)(x_C)=0$, and it follows that $\chi(x)=0$. Thus $d_G$ is
injective.

Note that $$l_G\circ d_G=\sum_{\lambda'\in\Lambda}l_{\lambda'}\circ
d_{\lambda'}.$$
Hence, for every $\chi\in\C\Irr(B)$, we have
$$l_G\circ d_G(\chi)(x)=\sum_{\lambda'\in\Lambda}l_{\lambda'}\circ
d_{\lambda'}(\chi)(x).$$
By Equation~(\ref{eq:llambda}), if
$\lambda'\neq\lambda$, then $l_{\lambda'}(d_{\lambda'}(\chi))(x)=0$.
On the other hand, since $d_{\lambda}(\chi)$ is constant on 
$\mathcal E_{x_{\lambda},y}^{\lambda}$ for any
$y\in C\cap G_{\lambda}$, 
Remark~\ref{rk:fonctionCstable}
implies $$l_{\lambda}\circ
d_{\lambda}(\chi)(x)=l_{\lambda}\circ d_{\lambda}(\chi)(x_{\lambda}x_C)=d_{\lambda}(\chi)(x_C)=
\res_C(r^{\lambda}(\chi))(x_C)=\chi(x),$$
as required.
\end{proof}

For $\lambda\in\Lambda$, we set $b_{\lambda}=b_{x_{\lambda}}$ and
$e_{\lambda}=e_{x_{\lambda}}$.
The dual of $\bigoplus_{\lambda\in\Lambda}\C\Irr(B_{\lambda})^{\mathcal
C_{x_S}}$
is $\bigoplus_{\lambda\in\Lambda}\C b_{\lambda}^{\vee}$ and the
homomorphism
$$e_{G}:\bigoplus_{\lambda\in\Lambda}\C b_{\lambda}^{\vee}\rightarrow
\C\Irr(B),\quad \sum_{\lambda\in\Lambda}\phi_{\lambda}\mapsto
\sum_{\lambda\in\Lambda}e_{\lambda}(\phi_{\lambda})$$
is the adjoint of $d_G$. 

\begin{remark}
\label{rk:baseadaptee}
Write
$\Z b_{\lambda}(\mathcal C)$ for the submodule of $\Z b_{\lambda}$
consisting of class functions constant on $\mathcal E_{x_{\lambda},y}^{\lambda}$ for any
$y\in C\cap G_{\lambda}$. 
Let  $K=\operatorname{rk}_{\Z}(\Z
b_{\lambda}(\mathcal C))$. By the invariant factor decomposition theorem, there
are a $\Z$-basis $\mathfrak b_{\lambda}=\{b_1,\ldots,b_N\}$ of $\Z
b_{\lambda}$ and positive integers $m_1,\ldots,m_K$ 
such that
$m_1|m_2|\cdots|m_K$ and 
$\{m_1 b_1,\ldots,m_K b_K\}$ is a $\Z$-basis of
$\Z b_{\lambda}(\mathcal C)$. Let $1\leq i\leq K$ and $y\in C\cap
G_{\lambda}$. Then for any $t\in \mathcal E_{x_{\lambda},y}^{\lambda}$,
one has $m_ib_i(t)=m_ib_i(y)$ because $m_ib_i\in\Z
b_{\lambda}(\mathcal C)$. 
Since $m_i\neq 0$, we deduce that $b_i(t)=b_i(y)$. Thus, $b_i\in
\Z b_{\lambda}(\mathcal C)$ and $m_i=1$.
In the following, we will write
$\mathfrak b_{\lambda}^{\mathcal C}=\{b_1,\ldots,b_K\}$ for the
$\Z$-basis of $\Z b_{\lambda}(\mathcal C)$ coming from such a construction.
\end{remark}

\subsection{Isometries}\label{pariso}
Let $G$ and $G'$ be two finite groups. We fix $\mathcal C$ (respectively
$\mathcal C'$) a set of conjugacy classes of $G$ (respectively $G'$), and $B$ (respectively
$B'$) a union of $\mathcal C$-blocks of $G$ (respectively $\mathcal
C'$-blocks of $G'$). As above, we write
$$C=\bigcup_{c\in\mathcal C}c\quad\textrm{and}\quad C'=\bigcup_{c'\in \mathcal
C'}c'.$$
We consider the isomorphism
$$\Theta: \left\{ \begin{array}{ccc} \C\Irr(B)\otimes\C\Irr(B')  & \longrightarrow
& \operatorname{End}(\C\Irr(B),\C\Irr(B')) \\ \sum_{\chi, \chi'} \chi\otimes\chi' & \longmapsto
 & \left(\varphi\mapsto\sum_{\chi, \chi'}\cyc{\varphi,\overline{\chi}}_G 
\chi'\right) \end{array} \right. $$

Note that, if we write $\widehat{f}=\Theta^{-1}(f)$ for any 
$f\in\operatorname{End}(\C\Irr(B),\C\Irr(B'))$, then 
\begin{equation}
\label{eq:chapeau}
\widehat{f}=\sum_{i=1}^r\overline{e_i^{\vee}}\otimes f(e_i),
\end{equation} where
$e=(e_1,\ldots,e_r)$ is any  $\C$-basis of
$\C\Irr(B)$ with dual basis $e^{\vee}=(e_1^{\vee},\ldots,e_r^{\vee})$ 
with respect to
$\cyc{\,,}_G$. 

\begin{theorem}
Let $G$ and $G'$ be two finite groups. Suppose that
\begin{enumerate}
\item[(1)] The group $G$ (respectively $G'$)
has an MN-structure with respect to $\mathcal C$ and $B$ (respectively
$\mathcal C'$ and $B'$). We keep the same notation
as above, and the object relative to $G'$ are denoted with a `prime'.
\item[(2)] Assume there are subsets $\Lambda_0\subseteq \Lambda$ and 
$\Lambda_0'\subseteq\Lambda'$ such that :
\begin{enumerate}
\item For every $\lambda\in\Lambda$ with $\lambda\not\in\Lambda_0$
(respectively $\lambda'\in\Lambda'$ with $\lambda'\notin\Lambda'_0$), we have
$r^{\lambda}=r'^{\lambda'}=0$.
\item  There is a bijection
$\sigma:\Lambda_0\rightarrow\Lambda_0'$ with $\sigma(\{1\})=\{1\}$ and
for $\lambda\in\Lambda_0$, an
isometry 
$I_{\lambda}:\C\Irr(B_{\lambda})\rightarrow\C\Irr(B'_{\sigma(\lambda)})$
such that 
$$I_{\lambda}\circ r^{\lambda}=r'^{\sigma(\lambda)}\circ I_{\{1\}}.$$
\end{enumerate}
\item[(3)] For $\lambda\in\Lambda_0$, we have
$I_{\lambda}(\C b_{\lambda}^{\vee}) = \C
{b'}_{\sigma(\lambda)}^{\vee}$. We write
$J_{\lambda}=I_{\lambda}|_{\C b_{\lambda}^{\vee}}$.
\end{enumerate}
Then for all $x\in G,\,x'\in G'$, we have
\begin{equation}
\label{eq:th}
\widehat{I}_{\{1\}}(x,x')=\sum_{\lambda\in\Lambda_0}\sum_{\phi\in
\mathfrak b_{\lambda}^{\mathcal C}}
\overline{e_{\lambda}(\Phi_{\phi})(x)}l'_{\sigma(\lambda)}(J_{\lambda}^{*-1}(\phi))(x'),
\end{equation}
where $\mathfrak b_{\lambda}$ and $\mathfrak b_{\lambda}^{\mathcal C}$
are as in Remark~\ref{rk:baseadaptee}, and 
$\mathfrak b_{\lambda}^{\vee}=\{\Phi_{\phi}\, |\, \phi\in \mathfrak b_{\lambda}\}$
is the dual basis of $\mathfrak b_{\lambda}$ as in \S\ref{pargen}.
\label{th:iso}
\end{theorem}

\begin{proof}
First, we remark that, for $\lambda\in\Lambda_0$, the adjoint of the
inclusion $i:\C b_{\lambda}^{\vee}\rightarrow
\C\Irr(B_{\lambda})$ is $i^*=\res_C$. Moreover, Hypothesis (3) implies
that the following diagram is commutative:
$$
\xymatrix{
   \C b_{\lambda}^{\vee}\ar[rr]^{J_{\lambda}} \ar[d]^{i} &&
   \C{b'}_{\sigma(\lambda)}^{\vee}\ar[d]^{i}\\
   \C\Irr(B_{\lambda}) \ar[rr]^ {I_{\lambda}}&&
   \C\Irr(B_{\sigma(\lambda)})
   } 
$$
Dualizing, we obtain the following commutative diagram:
$$
\xymatrix{
   \C b_{\lambda} &&
   \C{b'}_{\sigma(\lambda)}\ar[ll]^{J_{\lambda}^*}\\
   \C\Irr(B_{\lambda}) \ar[u]^{\res_{C}}&&
   \C\Irr(B_{\sigma(\lambda)})\ar[u]^{\res_{C'}}  \ar[ll]^ {I_{\lambda}^{-1}}
   } 
$$
(The bottom arrow is indeed $I_{\lambda}^{-1}$ because we identified $ \C\Irr(B_{\lambda}) $ and $\C\Irr(B_{\sigma(\lambda)})$ with their duals.) Thus, we have $\res_C\circ
I_{\lambda}^{-1}=J_{\lambda}^*\circ\res_{C'}$, which implies that
$J_{\lambda}^{*-1}\circ\res_C=\res_{C'}\circ I_{\lambda}$, and
we obtain 
\begin{eqnarray}
J_{\lambda}^{*-1}\circ\res_C\circ r^{\lambda}&=&\res_{C'}\circ
I_{\lambda}\circ r^{\lambda}\nonumber\\
J_{\lambda}^{*-1}\circ d_{\lambda}&=&\res_{C'}\circ
r'^{\sigma(\lambda)}\circ I_{\{1\}}\nonumber\\
J_{\lambda}^{*-1}\circ d_{\lambda}&=&d'_{\sigma(\lambda)}\circ
I_{\{1\}},
\label{eq:relcom}
\end{eqnarray}
where the second equality comes from Hypothesis (2). 

Write $\Z b_{\lambda}(\mathcal C)$ as in Remark~\ref{rk:baseadaptee}.
We have $d_{\lambda}(\C\Irr(B))\subset \C b_{\lambda}(\mathcal C)$.
Define $\mathfrak b_{\lambda}$, $\mathfrak b_{\lambda}^{\mathcal
C}$ as in Remark~\ref{rk:baseadaptee}, and set
$V_{\lambda}=l_{\lambda}(\C \mathfrak b_{\lambda}^{\mathcal C})$ and
$V'_{\lambda'}=l'_{\lambda'}(\C \mathfrak {b'}_{\lambda'}^{\mathcal C'})$. 

Now, 
the assumption (2.a) implies that
$d_G=\sum_{\lambda\in\Lambda_0}d_{\lambda}$ and if we again write $l_G$
for the restriction of $l_G$ to $\oplus_{\lambda\in\Lambda_0}
\C\Irr(G_{\lambda})^{\mathcal C_{\lambda}}$ to simplify the notation, 
then Lemma~\ref{applicationD} gives that $l_G\circ
d_G$ is the identity on $\C\Irr(B)$ (the same is true for $l'_{G'}\circ
d'_{G'}$).
In particular, $l_{\lambda}$ is surjective. Furthermore, for
$\lambda'\in\Lambda_0$ such that $\lambda\neq\lambda'$, one has
$d_{\lambda'}\circ l_\lambda(\phi)=0$ for
all $\phi\in\C b_{\lambda}(\mathcal C)$. Indeed, for every $x\in
C\cap G_{\lambda'}$, one has
$$d_{\lambda'}\circ l_{\lambda}(\phi)(x)=\res_{C}\circ
r^{\lambda'}(l_{\lambda}(\phi))(x)=r^{\lambda'}(l_{\lambda}(\phi)(x)=l_{\lambda}(\phi)(x_{\lambda'}x)=0,$$
because $\lambda\neq\lambda'$.
It follows that $\C\Irr(B)=\oplus_{\lambda\in\Lambda_0}V_{\lambda}$ (the
same is true for $\C\Irr(B')$). Thus, 
by Equation~(\ref{eq:relcom}), 
the following diagram is commutative:
\begin{equation}
\label{eq:diag1}
\xymatrix{
   \C \Irr(B)\ar[rr]^{I_{\{1\}}} \ar[d]^{d_G} &&
   \C\Irr(B')\ar[d]^{d'_{G'}}\\
   \bigoplus_{\lambda\in\Lambda_0}\C b_{\lambda} \ar[d]^{l_G}
\ar[rr]^{\oplus J_{\lambda}^{*-1}}&&
   \bigoplus_{\lambda\in\Lambda_0}\C
b'_{\sigma(\lambda)}\ar[d]^{l'_{G'}}\\
\bigoplus_{\lambda\in\Lambda_0}V_{\lambda}\ar[rr]^{I_{\{1\}}}&&
\bigoplus_{\lambda\in\Lambda_0}V'_{\sigma(\lambda)}
}
\end{equation}

Let $\lambda\in\Lambda_0$ and $\phi\in \C b_{\lambda}(\mathcal C)$. Then
$d_{\lambda}\circ l_{\lambda}(\phi)\in\C b_{\lambda}(\mathcal C)$. Let $x\in
G_{\lambda}$. If $x\notin C$, then 
$\phi(x)=0=d_{\lambda}\circ l_{\lambda}(\phi)(x)$. 
Assume that $x\in C$. Then $x\in G_{\lambda}\cap C$. So, by
Definition~\ref{defMN}(3), $(x_{\lambda},x)\in A$, and
by Definition~\ref{defMN}(4) and Remark~\ref{rk:fonctionCstable} we have
$$d_{\lambda}\circ l_{\lambda}(\phi)(x)=r^{\lambda}\circ
l_{\lambda}(\phi)(x)=l_{\lambda}(\phi)(x_{\lambda}\pd x).$$ 
Therefore, Remark~\ref{rk:fonctionCstable} gives $d_{\lambda}\circ
l_{\lambda}(\phi)(x)=\phi(x)$. So, this proves that for every $\phi\in
\C b_{\lambda}(\mathcal C)$, we have
\begin{equation}
d_{\lambda}\circ l_{\lambda}(\phi)=\phi.
\label{eq:invlambda}
\end{equation}
Consider 
$$e=\bigcup_{\lambda\in\Lambda_0}\{l_{\lambda}(\phi)\,|\,\phi\in
\mathfrak b_{\lambda}^{\mathcal C}\}.$$
By Equation~(\ref{eq:invlambda}), for $\lambda\in\Lambda_0$, the
family $\{l_{\lambda}(\phi)\,|\,\phi\in\mathfrak
b_{\lambda}^{\mathcal C}\}$ is linearly independent, and since $l_{\lambda}$ is
surjective, it is a basis of $V_{\lambda}$. Hence, it follows that
$e$ is a basis of $\C \Irr(B)$.

Now, we claim that
$$e^\vee=\bigcup_{\lambda\in\Lambda_0}
\left\{e_{\lambda}(\Phi_{\phi})\,|\,\phi\in
\mathfrak b_{\lambda}^{\mathcal C}\right\}.$$
Indeed, if $\lambda,\,\mu\in\Lambda_0$ with $\lambda\neq\mu$, then for any
$\vartheta\in \mathfrak b_{\lambda}^{\mathcal C}$ and $\phi \in
\mathfrak b_{\mu}^{\mathcal C}$, we have
$$\cyc{e_{\lambda}(\Phi_{\vartheta}),l_{\mu}(\phi)}_G=
\frac{1}{|G|}\sum_{g\in
G}e_{\lambda}(\Phi_{\vartheta})(g)\,\overline{l_{\mu}(\phi)(g)}=0,$$
by Equation~(\ref{eq:llambda}) and Lemma~\ref{valeurE}.
Furthermore, if $\phi,\,\varphi\in
\mathfrak b_{\lambda}^{\mathcal C}$, then Equation~(\ref{eq:invlambda})
gives
$$\cyc{e_{\lambda}(\Phi_{\varphi}),l_{\lambda}(\phi)}_G=\cyc{
\Phi_{\varphi},d_{\lambda}\circ
l_{\lambda}(\phi)}_{G_{\lambda}}
=\cyc{\Phi_{\varphi},\phi}_{G_{\lambda}}=\delta_{\varphi\phi},
$$
and the result follows.
Thus, writing $\widehat{I}_{\{1\}}$ with respect to the basis $e$, we
obtain
$$\widehat{I}_{\{1\}}=\sum_{\lambda\in\Lambda_0}\sum_{\phi\in
b_{\lambda}}\overline{e_{\lambda}(\Phi_{\phi})}\otimes
l'_{\sigma(\lambda)}(J_{\lambda}^{*-1}(\phi)),$$
as required.
\end{proof}

\begin{remark}
\label{rk:point3}
Note that the assumption (2) of the theorem implies that the assumption
(3) of the theorem holds for $\lambda=\{1\}$. Indeed, for $\phi\in
\C\Irr(B)$,
we have $\phi\in \C b_{\{1\}}^{\vee}$ if and only
if 
$\res_{\overline{C}}(\phi)=0$, where
$\overline{C}=G\backslash C$. However, $x\in G$ lies in $\overline{C}$
if and only if its type $\lambda$ is non-trivial. Thus,
Definition~\ref{defMN}(4) implies that $\phi\in \C b_{\{1\}}^{\vee}$ if
and only if 
$r^{\lambda}(\phi)=0$ for all
$\lambda\neq \{1\}$. Let $\phi\in\C b_{\{1\}}^{\vee}$. 
Then for any $\{1\}\neq
\lambda\in\Lambda_0$,
$$r'^{\sigma(\lambda)}(I_{\{1\}}(\phi))=I_{\lambda}(r^{\lambda}(\phi))=0.$$
Since $\sigma$ is a bijection with $\sigma(\{1\})=\{1\}$, we deduce that
$I_{\{1\}}(\phi)\in\C b'^{\vee}_{\sigma(\lambda)}$.
To obtain the reverse inclusion, we apply this argument to $I_{\{1\}}^{-1}$.

In particular, if for any $\lambda\in\Lambda_0$, the group $G_{\lambda}$
has an MN-structure with respect to $C\cap G_{\lambda}$ and
$B_{\lambda}$, then the assumption (3) of Theorem~\ref{th:iso} is
automatically satisfied.
\end{remark}

\begin{remark}\label{rk:adjoint}
Suppose that
$(I_{\lambda}:\C\Irr(B_\lambda)\rightarrow\C\Irr(B'_{\sigma(\lambda)}))_{\lambda\in\Lambda_0}$ 
are isometries such that properties (1), (2) and (3) of
Theorem~\ref{th:iso} hold. Then
$I_{\lambda}^{-1}:\C\Irr(B'_{\sigma(\lambda)})\rightarrow
\C\Irr(B_{\lambda})$ also satisfies the hypotheses of the theorem (for
$\sigma^{-1}:\Lambda_0'\rightarrow\Lambda_0$).
Moreover, writing $\widehat{I}$ with respect to the self-dual
$\C$-basis $\Irr(B)$ of $\C\Irr(B)$, we have
\begin{equation}
\label{eq:Ichapeau}
\widehat{I}=\sum_{\chi\in \Irr(B)}\overline{\chi}\otimes I(\chi).
\end{equation}
It follows that
$$\widehat{I}=\sum_{\chi'\in\Irr(B')}\overline{I^{-1}(\chi')}\otimes \chi'=
\operatorname{conj}\left(\sum_{\chi'\in\Irr(B')}
I^{-1}(\chi')\otimes\overline{\chi'}\right)=\operatorname{conj}\left(\widehat{I^{-1}}\circ\tau\right),$$
where $\tau:G\times G'\rightarrow G'\times G,\ (x,x')\mapsto (x',x)$ and
$\operatorname{conj}$ denotes the complex conjugation.
\end{remark}

\subsection{Generalized perfect isometries}

An isometry $I:\C\Irr(B)\rightarrow\C\Irr(B')$ with respect to the
scalar products $\cyc{\,,}_G$ and $\cyc{\,,}_{G'}$ is said to be a
generalized perfect isometry if $I(\Z\Irr(B))=\Z(\Irr(B'))$ and
\begin{equation}
\label{eq:isoperf}
I\circ\res_C=\res_{C'}\circ I.
\end{equation}

\begin{remark}
Following K\"ulshammer, Olsson and Robinson (see~\cite{KOR}), we say
that an isometry $I:\C\Irr(B)\rightarrow\C\Irr(B')$ is a KOR-isometry
if $I(\Z(\Irr(B))=\Z\Irr(B')$ and for all $\chi,\,\psi\in\Irr(B)$, one
has $$\cyc{\res_{C}(\chi),\res_C(\psi)}_G=\cyc{\res_{C'}\left(
I(\chi)\right),\res_{C'}\left( I(\psi)\right)}_{G'}.$$ Note that the 
argument in the proof of~\cite[Proposition 2.2]{BrGr} shows that the
KOR-isometries are precisely the isometries that satisfy
Equation~(\ref{eq:isoperf}). For the convenience of the reader, we now
prove this fact. Before this, we recall that the notion of blocks
in~\cite{KOR} is not the same as ours. The KOR-blocks are the
equivalence classes for the equivalence relation on $\Irr(G)$ obtained
by extending by transitivity the relation defined by
$\cyc{\res_C(\chi),\res_C(\psi)}_G\neq 0$. First, we will show that
$\Irr(B)$ is a union of KOR-blocks. Since the KOR-blocks are a partition
of $\Irr(G)$, it is clear that $\Irr(B)$ is contained in a union of
KOR-blocks. It is sufficient to show that if $\chi\in \Irr(B)$ and
$\psi\in\Irr(G)$ are such that $\cyc{\res_C(\chi),\res_C(\psi)}_G\neq 0$,
then $\psi\in\Irr(B)$.  Let $\chi\in \Irr(B)$ and $\psi\in\Irr(G)$ be
such that $\cyc{\res_C(\chi),\res_C(\psi)}_G\neq 0$, that is
$$\sum_{\varphi,\,\vartheta\in
b}d_{\chi\varphi}d_{\psi\vartheta}\cyc{\varphi,\vartheta}_G\neq 0.$$
In particular, there exists some $\varphi,\,\vartheta\in b$ such that
$d_{\chi\varphi}d_{\psi\vartheta}\cyc{\varphi,\vartheta}_G\neq 0$.
Hence, $d_{\chi\varphi}\neq 0\neq d_{\psi\vartheta}$ and
$\cyc{\varphi,\vartheta}_G\neq 0$. Thanks to Corollary~\ref{critere}, we
conclude that $\psi$ lies in the $\mathcal C$-block of $\chi$.

Now suppose $I:\C\Irr(B)\rightarrow\C\Irr(B')$ is a generalized perfect
isometry. 
Let
$\chi,\,\psi\in\Irr(B)$. Then
$$\cyc{\res_{C'}\left( I(\chi)\right),\res_{C'}\left(
I(\psi)\right)}_G=\cyc{I\left(\res_C(\chi)\right),I\left(\res_C(\psi)\right)}_G=\cyc{\res_C(\chi),\res_C(\psi)}_G,$$
because $I$ is an isometry.

Conversely, assume that $I$ is a
KOR-isometry. Let $\chi\in\Irr(B)$. We have
\begin{eqnarray*}
I(\res_C(\chi))&=&I\left(\sum_{\psi\in\Irr(G)}\cyc{\res_C(\chi),\psi}_G\psi\right)\\
&=&\sum_{\psi\in\Irr(B)}\cyc{\res_C(\chi),\res_C(\psi)}_GI(\psi)\\
&=&\sum_{\psi\in\Irr(B)}\cyc{\res_{C'}(I(\chi)),\res_{C'}(I(\psi))}_{G'}I(\psi)\\
&=&\sum_{\psi\in\Irr(B)}\cyc{\res_{C'}(I(\chi)),I(\psi)}_{G'}I(\psi)\\
&=&\res_{C'}(I(\chi)),
\end{eqnarray*}
proving the claim.
\label{rk:kor}
\end{remark}
%

\begin{proposition}
Suppose that $\{B_i\,|1\leq i\leq r\}$ is the set of KOR-blocks of
$G$ with respect to a set of classes $\mathcal C$. Then there is a
$\Z$-basis $b$ of $\Z\Irr(G)^{\mathcal C}$ such that the $B_i$'s are the
$\mathcal C$-blocks of $G$ with respect to $b$.
\label{prop:choixbase}
\end{proposition}

\begin{proof}
By definition of the KOR-blocks, the sets $\Irr(B_i)^{\mathcal C}$ and
$\Irr(B_j)^{\mathcal C}$ for $i\neq j$ are orthogonal with respect to
$\cyc{\,,}_G$, implying that
$$\Z\Irr(G)^{\mathcal C}=\bigoplus_{i=1}^r\Z\Irr(B_i)^{\mathcal
C}.$$
Choose any $\Z$-basis $b_i$ of $\Z\Irr(B_i)^{\mathcal C}$ and write
$b_i^{\vee}$ for the dual basis of $b_i$ in the $\C$-space
$\C\Irr(B_i)^{\mathcal C}$ with respect to $\cyc{\,,}_G$. Define
$b=b_1\cup\ldots\cup b_r$. Then $b$ is a $\Z$-basis of
$\Z\Irr(G)^{\mathcal C}$. Moreover, since $b_i^{\vee}\subseteq
\C\Irr(B_i)^{\mathcal C}$, and since the KOR-blocks are
orthogonal, we deduce that $b^{\vee}=b_1^\vee\cup\ldots\cup b_r^\vee$ is
the dual basis of $b$. 
Now, for $\varphi\in b_i$, we have
$$\Phi_{\varphi}=\res_C(\Phi_{\varphi})=\sum_{j=1}^r\sum_{\chi\in\Irr(B_j)}d_{\chi\varphi}\res_C(\chi)=\sum_{\chi\in\Irr(B_i)}d_{\chi\varphi}\res_C(\chi),$$
because $\Phi_{\varphi}\in\C\Irr(B_i)^{\mathcal C}$. Hence, for
$\chi'\notin \Irr(B_i)$, we have
$$d_{\chi'\varphi}=\cyc{\Phi_{\varphi},\chi'}_G=\sum_{\chi\in\Irr(B_i)}d_{\chi\varphi}\cyc{\res_{C}(\chi),\res_C(\chi')}_G=0.$$
This proves that $B_i$ is a union of $\mathcal C$-blocks. Furthermore, we
have seen in Remark~\ref{rk:kor} that conversely, the $\mathcal
C$-blocks are unions of KOR-blocks. The result follows.
\end{proof}

\begin{proposition}
Let $I:\C\Irr(B)\rightarrow\C\Irr(B')$ be an isometry, and assume that
$I(\Z\Irr(B))=\Z\Irr(B')$. The following
assertions are equivalent
\begin{enumerate}[(i)]
\item $I$ is a generalized perfect isometry.
\item If $\widehat{I}(x,y)\neq 0$, then either $(x,y)\in C\times
C'$, or $(x,y)\in\overline{C}\times\overline{C}'$, where
$\overline{C}=G\backslash C$ and $\overline{C}'=G\backslash C'$.
\end{enumerate}
\label{prop:equiviso}
\end{proposition}

\begin{proof}
Suppose that $I$ is a generalized perfect isometry. Note that
$\C\Irr(B)^{C\perp}=\C\Irr(B)^{\overline{C}}$ and
\begin{equation}
\C\Irr(B)=\C\Irr(B)^{C}\oplus\C\Irr(B)^{\overline{C}}.
\label{eq:int3}
\end{equation}
Moreover, for any $\phi\in\C\Irr(B')$, there is $\chi\in\C\Irr(B)$ 
such that $I(\chi)=\phi$ (because $I$ is an
isometry). Thanks to Equation~(\ref{eq:isoperf}), we have
$\res_{C'}(\phi)=I(\res_C(\chi))$. Hence, the restriction
$I:\C\Irr(B)^{C}\rightarrow \C\Irr(B')^{C'}$ is surjective, and yet
bijective (because $I$ is injective). Since $I$ is an isometry, we have
$$I\left(
(\C\Irr(B)^{C})^{\perp}\right)=I\left(\C\Irr(B)^{C}\right)^{\perp}=
(\C\Irr(B')^{C'})^{\perp}.$$  It follows that 
\begin{equation}
I(\C\Irr(B)^{\overline{C}})=\C\Irr(B')^{\overline{C}'}.
\label{eq:int2}
\end{equation}
Now, we choose a $\C$-basis $b$ of $\C\Irr(B)^{C}$ with dual basis
$b^{\vee}$
and a $\C$-basis $\overline{b}$ 
of $\C\Irr(B)^{\overline{C}}$ with dual basis $\overline{b}^{\vee}$.
Therefore, thanks to Equation~(\ref{eq:int3}),
$b\cup \overline{b}$ is a $\C$-basis of $\C\Irr(B)$ with dual
basis $b^{\vee}\cup\overline{b}^{\vee}$. 
Writing $\widehat{I}$ with respect to
this basis, we obtain
\begin{equation}
\label{eq:int1}
\widehat{I}=\sum_{\alpha\in b}\overline{\alpha^{\vee}}\otimes
I(\alpha)+\sum_{\beta\in\overline{b}}\overline{\beta^{\vee}}\otimes I(\beta).
\end{equation}
Now, let $(x,y)\in C\times \overline{C}'$. Then Equation~(\ref{eq:int1})
gives $\widehat{I}(x,y)=\sum_{\alpha\in b}\overline{\alpha^{\vee}}(x)
I(\alpha)(y)$. But 
$I(\alpha)\in\C\Irr(B')^{C'}$, implying that $I(\alpha)(y)=0$. Hence,
$\widehat{I}(x,y)=0$.

For $(x,y)\in \overline{C}\times C'$, we similarly conclude that
$\widehat{I}(x,y)=0$ using Equations~(\ref{eq:int1})
and~(\ref{eq:int2}). This proves that (i) implies (ii).

Conversely, assume that (ii) holds. For $y\in G'$, we write
$\widehat{I}_y:G\rightarrow \C,\,x\mapsto \widehat{I}(x,y)$. This is a
class function on $G$. We now write $\widehat{I}$ with respect to the
$\C$-basis $\Irr(B)$. Thus, Equation~(\ref{eq:Ichapeau}) implies that
for $\chi\in \Irr(G)$ and $y\in G'$, we have
\begin{eqnarray*}
I(\chi)(y)&=&\sum_{\theta\in\Irr(B)}I(\theta)(y)\cyc{\overline{\theta},\overline{\chi}}_G
\\&=&\cyc{\widehat{I}_y,\overline{\chi}}_G\\
&=&\frac{1}{|G|}\sum_{x\in G}\widehat{I}(x,y)\chi(x).
\end{eqnarray*}
In particular, for any $\psi\in\C\Irr(G)$ and $y\in G'$, we have
\begin{equation}
I(\psi)(y)=\frac{1}{|G|}\sum_{x\in G}\widehat{I}(x,y)\psi(x).
\label{eq:int4}
\end{equation}
Let $\chi\in\C\Irr(B)$ and $y\in G'$. Applying Equation~(\ref{eq:int4}) to
$\res_C(\chi)$, we obtain 
$$ I(\res_C(\chi))(y)=\frac{1}{|G|}\sum_{x\in
G}\widehat{I}(x,y)\res_C(\chi)(x)
=\frac{1}{|G|}\sum_{x\in C}\widehat{I}(x,y)\chi(x).$$
Suppose that $y\in\overline{C}'$. Then $\widehat{I}(x,y)\neq 0$ only if
$x\in\overline{C}$ and the second equality gives $I(\res_C(\chi))(y)=0$.
Otherwise, if $y\in C'$, then $\widehat{I}(x,y)=0$ for $x\notin C$. 
In particular,
$\frac{1}{|G|}\sum_{x\in C}\widehat{I}(x,y)\chi(x)$ is equal to $I(\chi)(y)=\res_{C'}(I(\chi))(y)$.
This proves that $I$ satisfies Equation~(\ref{eq:isoperf}), whence is a generalized perfect isometry.
\end{proof}

\begin{remark}
Note that Equation~(\ref{eq:int3}) applied to $B'$ and
Equation~(\ref{eq:int2}) imply that
$$\res_{\overline{C}'}\circ I=I\circ\res_{\overline{C}}.$$
\end{remark}

\begin{corollary}
Let $G$ and $G'$ be two finite groups. We assume that Hypotheses (1),
(2) and (3) of Theorem~\ref{th:iso} are satisfied, and we keep the same
notation. If $I_{\{1\}}(\Z\Irr(B))=\Z\Irr(B')$, then $I_{\{1\}}$ is a 
generalized perfect isometry.
\label{cor:isoparfaitegene}
\end{corollary}

\begin{proof}
Let $(x,x')\in G\times G'$. Write $\mu$ and $\mu'$ for the type of $x$
and $x'$. Suppose that $(x,x')\notin C\times C'$ and
$(x,x')\notin\overline{C}\times\overline{C}'$. Then either $\mu=\{1\}$
and $\mu'\neq\{1\}$, or $\mu\neq\{1\}$ and $\mu'=\{1\}$. Since
$\sigma(\{1\})=\{1\}$, we deduce that $\mu'\neq\sigma(\mu)$. Thanks to
Lemma~\ref{valeurE} and Equation~(\ref{eq:llambda}), we have for every
$\lambda\in \Lambda_0$ and $\phi\in \mathfrak b_{\lambda}^{\mathcal C}$,
either $e_{\lambda}(\Phi_{\phi})(x)=0$, or
$l'_{\sigma(\lambda)}(J_{\lambda}^{*-1}(\phi))(x')=0$. In particular,
Equation~(\ref{eq:th}) gives $\widehat{I}_{\{1\}}(x,x')=0$, and the
result follows from Proposition~\ref{prop:equiviso}.
\end{proof}

\subsection{Brou\'e's isometries}\label{secbroue}
In this subsection, we fix a prime
number $p$ and assume that $C$ and $C'$ are the sets of $p$-regular elements
(that is, elements whose order is prime to $p$) of $G$ and $G'$,
respectively. Let $B$ and $B'$ be a union of $p$-blocks of $G$ and $G'$.
Denote by $(K,\mathcal R,k)$ a splitting $p$-modular system for $G$ and $G'$.
Let $I:\C\Irr(B)\rightarrow\C\Irr(B')$ be an isometry such that
$I(\Z\Irr(B))=\Z\Irr(B')$ and $\widehat{I}$ defined in
Equation~(\ref{eq:chapeau}) is perfect, that is
\begin{enumerate}[(i)]
\item For every $(x,x')\in G\times G'$, $\widehat{I}(x,x')$ lies in
$|\Cen_{G}(x)|\mathcal R\cap |\Cen_{G'}(x')|\mathcal R$.
\item $\widehat{I}$ satisfies property (ii) of
Proposition~\ref{prop:equiviso}.
\end{enumerate}
We call such an isometry a Brou\'e isometry.

\begin{remark}
In fact, the perfect character $\mu:G\times G'\rightarrow\C$ defined by 
Brou\'e in~\cite{Broue} 
is not exactly $\widehat{I}$, but
$\mu(x,x')=\widehat{I}(x^{-1},x')$. However, since the
sets of $p$-regular and $p$-singular elements are stable under $g\mapsto
g^{-1}$, it follows that $\mu$ is perfect if and only if $\widehat{I}$ is
perfect.
\end{remark}

\begin{remark}Since the set of irreducible Brauer characters $\IBr_p(G)$
is a $\Z$-basis of $\Z\Irr(G)^{C}$ which satisfies the
conclusion of Proposition~\ref{prop:choixbase}, Remark~\ref{rk:kor} and
Proposition~\ref{prop:equiviso} imply that a Brou\'e isometry is a
perfect generalized isometry (in our sense and in the sense of
K\"ulshammer, Olsson and Robinson).
\label{rk:broue}
\end{remark}

\begin{theorem}
Assume that $G$ and $G'$ are two finite groups and that $C$ and $C'$ are
the sets of $p$-regular elements of $G$ and $G'$, respectively. Suppose 
that the following three conditions are satisified:
\begin{enumerate}[(i)]
\item
The hypotheses of Theorem~\ref{th:iso} hold.
\item
For any $\lambda\in\Lambda_0$, we have $I_{\lambda}(\Z\Irr(B_{\lambda}))=
\Z\Irr(B'_{\sigma(\lambda)})$. 
\item For every
$g=g_S\pd g_C\in
G$ and $g'=g'_{S'}\pd g'_{C'}\in G'$, $p$ does not divide $|\mathcal
E_{g_S,g_C}^{\lambda}|$ for any $\lambda\in\Lambda$, and $p$ does not divide
$|\mathcal E_{g'_S,g'_{C'}}^{\lambda'}|$ for any $\lambda'\in\Lambda'$. 
\end{enumerate} 
Then $I_{\{1\}}$ is a Brou\'e isometry.
\label{th:broue}
\end{theorem}

\begin{proof}By Remark~\ref{rk:broue} and
Corollary~\ref{cor:isoparfaitegene}, $I_{\{1\}}$ satisfies Property
(ii). We thus only prove Property (i). 
For $\lambda\in\Lambda_0$, we take $b_{\lambda}=\IBr_p(B_{\lambda})$.  
In particular, $\Z b_{\lambda}^{\vee}$ is the set of projective characters of
$B_{\lambda}$. 
By Assumption~(ii), and Hypothesis~(iii) of Theorem~\ref{th:iso}, since
$I_{\lambda}$ is injective, one has 
$$I_{\lambda}\left(\Z\Irr(B_{\lambda})\cap\C
b_{\lambda}^{\vee}\right)=I_{\lambda}\left(\Z\Irr(B_{\lambda})\right)
\cap I_{\lambda}(\C
b_{\lambda}^{\vee})=\Z\Irr(B'_{\sigma(\lambda)})\cap\C
b'_{\sigma(\lambda)}.$$
Hence, Corollary~\ref{blocs} gives $J_{\lambda}(\Z
b_{\lambda}^{\vee})=\Z {b'}_{\sigma(\lambda)}^{\vee}$, and it follows that
$J_{\lambda}^{*-1}(\phi)\in\Z\IBr_p(B'_{\sigma(\lambda)})$ for all
$\phi\in\Z b_{\lambda}$.
Now, let
$\mathfrak b_{\lambda}$ and $\mathfrak b_{\lambda}^{\mathcal C}$ be as in
Remark~\ref{rk:baseadaptee}. Let 
$\phi\in \mathfrak b_{\lambda}^{\mathcal C}$. 
Let $g\in G$ and $g'\in G'$. Write
$g=g_S\pd g_C$ and $g'=g'_{S'}\pd g'_{C'}$, and assume that $g$ is of 
type $\mu$. Then by
Lemma~\ref{valeurE}, one has
$e_{\lambda}(\Phi_{\phi})(g)=0$ for $\lambda\neq\mu$ and 
$e_{\mu}(\Phi_{\phi})(g)=|\Cen_G(g)|\sum_{x_C\in
\mathcal E_{g_S,g_C}^{\mu}}\frac{\Phi_{\phi}(x_C)}{|\Cen_{G_{\mu}}(x_C)|}$. 
Furthermore, thanks to Equation~(\ref{eq:llambda}) and the fact that $p$
does not divide $|\mathcal E_{g'_{S'},g'_{C'}}^{\sigma(\mu)}|$, we deduce that
$l'_{\sigma(\mu)}(J_{\mu}^{*-1}(\phi))(g')\in\mathcal R$. Now, by
Equation~(\ref{eq:th}) and Theorem~\ref{th:iso}, we obtain
$$\frac{\widehat{I}_{\{1\}}(g,g')}{|\Cen_G(g)|}=\sum_{\phi\in
\mathfrak b_{\mu}^{\mathcal C}}\sum_{x_C\in\mathcal E_{g_S,g_C}^{\mu}}\frac{\overline{\Phi_{\phi}(x_C)}}{|\Cen_{G_{\mu}}(x_C)|_p}\cdot\frac{
l'_{\sigma(\mu)}(J_{\mu}^{*-1}(\phi))(g')}{|\Cen_{G_{\mu}}(x_C)|_{p'}}\in
\mathcal R,$$ 
%
%
because $1/|\Cen_{G_{\mu}}(x_C)|_{p'}\in\mathcal R$, and
$\Phi_{\phi}(x_C)/|\Cen_{G_\mu}(x_C)|_p\in\mathcal R$ by~\cite[2.21]{Navarro}.
Similarly, using Remark~\ref{rk:adjoint}, we deduce that
$\widehat{I}_{\{1\}}(g,g')/|\Cen_{G'}(g')|\in\mathcal R$, as required.
\end{proof}

\begin{remark}
The proof of Theorem~\ref{th:broue} shows that the condition (iii) can
be replaced by
$I_{\lambda}(\C\Irr(B_{\lambda})(\mathcal
C))=\C\Irr(B_{\sigma(\lambda)})(\mathcal C')$ for any $\lambda\in\Lambda$, where similarly to
Remark~\ref{rk:baseadaptee}, $\C\Irr(B_{\lambda})(\mathcal C)$ denotes
the set of class functions of $\C\Irr(B_{\lambda})$ constant on
$\mathcal E_{x_{\lambda},y}^{\lambda}$ for any $y\in C\cap
G_{\lambda}$. Indeed, with this assumption, we have
$J_{\lambda}^{*-1}(\Z b_{\lambda}({\mathcal C}))\subseteq
\Z b_{\sigma(\lambda)}({\mathcal C'})$, and
Remark~\ref{rk:fonctionCstable} gives that
$l'_{\sigma(\lambda)}(J_{\lambda}^{*-1}(\phi))(g')\in\mathcal R$ for any
$\phi\in\mathfrak b_{\lambda}^{\mathcal C}$ and $g'\in G'$.
\label{rk:prop3bis}
\end{remark}

\section{Alternating groups}
\label{section:An}

Let $n$ be a positive integer and $p$ be a prime. 
We denote by $\mathcal{P}_n$ (or $\mathcal P$)
the set of partitions of $n$, by $\mathcal
O_n$ (or $\mathcal O$) the set of partitions of $n$ whose parts are odd, 
and by $\mathcal D_n$ (or $\mathcal D$) the set of partitions of
$n$ whose parts are distinct. We also write $\mathcal{OD}_n$ (or
$\mathcal{OD}$) for $\mathcal O_n\cap\mathcal D_n$ (respectively $\mathcal
O\cap\mathcal D$). Moreover, for $\lambda=(\lambda_1,\ldots,\lambda_r)
\in\mathcal P$, we write
$|\lambda|=\sum \lambda_i $ and $\ell(\lambda)=r$.

\subsection{Notation}\label{subsec:notAn}
For any $\lambda\in\mathcal{P}_n$, we write $\chi_{\lambda}$ for the
corresponding irreducible character of $\sym_n$, and $\lambda^*$ for the
conjugate partition of $\lambda$. It is well-known that
$\chi_{\lambda^*}=\chi_{\lambda}\otimes\varepsilon$, where $\varepsilon$
denotes the sign character of $\sym_n$. 
The character
$\chi_{\lambda}$ is called self-conjugate when
$\lambda=\lambda^*$. 
We denote by
$\mathcal{SC}_n$ (or $\mathcal{SC}$) the set of self-conjugate
partitions of $n$.  For any $\lambda\in\mathcal{SC}_n$, write
$\overline{\lambda}\in\mathcal{OD}_n$ for the partition whose parts are
the diagonal hook lengths of $\lambda$ (see~\cite[p.\,4]{olsson} for the
definition of a hook and its hook length. 
Recall that with the notation of~\cite{olsson}, 
a diagonal hook is an $(i,i)$-hook for some $i$), and define the map
\begin{equation}
\label{eq:defa}
a:\mathcal{SC}_n\longrightarrow\mathcal{OD}_n,\quad\lambda\mapsto
\overline{\lambda}.
\end{equation}
We remark that $a$ is bijective, and that $a^{-1}(\lambda)$ is the
self-conjugate partition whose diagonal hooks have lengths the parts of
$\lambda$.

Now, recall that
$\Res_{\Alt_n}^{\sym_n}(\chi_{\lambda})$ is irreducible if and only if
$\lambda$ is a non self-conjugate partition (i.e. $\lambda\neq\lambda^*$).
In this case,
$\chi_{\lambda}$ and $\chi_{\lambda^*}$ restrict to the same
irreducible character, which we denote by $\caralt_{\lambda}$. Otherwise,
when $\lambda=\lambda^*$, the restriction of $\carsym_{\lambda}$ to
$\Alt_n$ is the sum of two irreducible characters $\caralt_{\lambda}^-$
and $\caralt_{\lambda}^+$. 
Moreover, the conjugacy class of $\sym_n$ labeled by
$a(\lambda)$ splits into two classes $a(\lambda)^{\pm}$
of $\Alt_n$, and following~\cite[Theorem 2.5.13]{James-Kerber}, the
notation can be chosen such that 
$\caralt_{\lambda}^{\pm}(a(\lambda)^+)=x_{\lambda}\pm y_{\lambda}$
and $\caralt_{\lambda}^{\pm}(a(\lambda)^-)=x_{\lambda}\mp
y_{\lambda}$ with
\begin{equation}
\label{eq:Anchoix}
x_{\lambda}=\frac{1}{2}(-1)^{\frac{n-k}{2}}\quad\textrm{and}
\quad y_{\lambda}=\frac{1}{2}\sqrt{(-1)^{\frac{n-k}{2}}h_1\cdots h_k},
\end{equation}
where $a(\lambda)=(h_1>h_2>\cdots>h_k)$. Note that
$x_{\lambda}=\carsym_{\lambda}(a(\lambda))/2$, and 
if $x\in\Alt_n$ does
not belong to the class of $\sym_n$ parametrized by
$a(\lambda)$, then
$\caralt_{\lambda}^+(x)=\caralt_{\lambda}^-(x)=\carsym_{\lambda}(x)/2$.

Let $q$ be a positive integer.
To any $\lambda\in\mathcal P_n$, we associate its $q$-core
$\lambda_{(q)}$ and its $q$-quotient
$\lambda^{(q)}=(\lambda^1,\ldots,\lambda^q)$; see for
example~\cite[p.\,17]{olsson}. Recall that the map
\begin{equation}
\label{eq:partquotientcoeur}
\lambda\mapsto(\lambda_{(q)},\lambda^{(q)})
\end{equation}
is bijective. Define
\begin{equation}
\lambda^{(q)*}=\left( (\lambda^{q})^*,\ldots,(\lambda^{1})^*\right).
\label{eq:conjquotient}
\end{equation}
Then by~\cite[Proposition 3.5]{olsson}, the $q$-core and $q$-quotient 
of $\lambda^*$ are
$\lambda_{(q)}^*$ and $\lambda^{(q)*}$ respectively. In particular,
\begin{equation}
\lambda=\lambda^* \Longleftrightarrow
\lambda_{(q)}^*=\lambda_{(q)}\quad\textrm{and}\quad\lambda^{(q)*}=\lambda^{(q)}.
\label{souv}
\end{equation}

\subsection{$p$-blocks of $\Alt_n$}
\label{subsec:pblockAn}
The ``Nakayama Conjecture'' asserts that two irreducible characters lie
in the same $p$-block of $\sym_n$ if and only if the partitions
labeling them have the same
$p$-core; see~\cite[Theorem 6.1.21]{James-Kerber}. Hence, the $p$-blocks
of $\sym_n$ are labeled by the $p$-cores of
partitions of $n$. Such $p$-cores are called $p$-cores of $n$ (or of
$\sym_n$).  For a $p$-core $\gamma$ of $n$, we denote by $B_{\gamma}$
the corresponding $p$-block of $\sym_n$. Moreover, we define the
$p$-weight of $\gamma$ (or of $B_{\gamma}$) by setting
$w=(n-|\gamma|)/p$.

Let $\gamma$ be a $p$-core of $n$. Then $\gamma^*$ is also a $p$-core of
$n$, and
$\Irr(B_{\gamma^*})=\{\chi_{\lambda^{*}}\in\Irr(\sym_n)\,|\,\lambda_{(p)}=\gamma\}=\Irr(B_{\gamma})^*$.

If $\gamma\neq \gamma^*$, then
$\Irr(B_{\gamma})\cap\Irr(B_{\gamma^*})=\emptyset$ and 
$\Irr(B_{\gamma})$ contains no self-conjugate character. In particular,
the $p$-blocks
$B_{\gamma}$ and $B_{\gamma^*}$
cover a unique $p$-block $b_{\gamma,\gamma^*}$ of $\Alt_n$, which is such that
$\Irr(b_{\gamma,\gamma^*})=\{\caralt_{\lambda}\in\Irr(\Alt_n)\,|\,\lambda_{(p)}=\gamma\}
=\{\caralt_{\lambda}\in\Irr(\Alt_n)\,|\,\lambda_{(p)}=\gamma^*\}$.

Assume instead that $\gamma=\gamma^*$. Suppose that $w>0$.  
By
Equation~(\ref{eq:partquotientcoeur}), there is a partition $\lambda$ of
$n$ with $p$-core $\gamma$ and $p$-quotient $(
(w),\emptyset,\ldots,\emptyset)$. Furthermore,
$\chi_{\lambda}\neq\chi_{\lambda^*}$ by  Equation~(\ref{souv}) and
$\chi_{\lambda}\in\Irr(B_{\gamma})$.  
Hence, $\chi_{\lambda}$ restricts irreducibly to $\Alt_n$, 
and~\cite[Theorem 9.2]{Navarro} implies that $B_{\gamma}$ covers a unique
$p$-block $b_{\gamma}$ of $\Alt_n$.

If, on the other hand, $w=0$, then $\Irr(B_{\gamma})=\{\chi_{\gamma}\}$
has defect zero. If $n\leq 1$, then $\Alt_n=\sym_n=\{1\}$, and
$\rho_{\gamma}=\chi_{\gamma}$ is the trivial character. The case $n=2$
does not occur, because there are no self-conjugate partitions of size
$2$.
If $n\geq 3$, then $\{\rho_{\gamma}^{+}\}$ and $\{\rho_{\gamma}^-\}$ are
$p$-blocks of defect zero of $\Alt_n$.

\subsection{Brou\'e perfect isometries}\label{subsec:BroueAn}

Let $q$ be a positive integer.
For $\lambda\in\mathcal P_n$, we denote by $M_q(\lambda)$ the set of
$\mu\in\mathcal P_{n-q}$ such that $\mu$ is obtained from $\lambda$ by
removing a $q$-hook. (The definition of $q$-hooks, and the process 
to remove a $q$-hook from a partition, is for example given
in~\cite[p.\,4,\,5,\,6]{olsson}). Note that, if $\mu\in M_{q}(\lambda)$, then $\mu^*\in M_{q}(\lambda^*)$.

For $\mu\in M_{q}(\lambda)$, we denote by $c_{\mu}^{\lambda}$ the $q$-hook of
$\lambda$ such that $\mu$ is obtained from $\lambda$ by removing
$c_{\mu}^{\lambda}$. Define
\begin{equation}
\label{eq:alpha}
\alpha_{\mu}^{\lambda}=(-1)^{L(c_{\mu}^{\lambda})},
\end{equation}
where $L(c_{\mu}^{\lambda})$ denotes the leglength of
$c_{\mu}^{\lambda}$ (see for
example~\cite[p.\,4]{olsson} for the definition of the leglength of a
hook). 

\begin{lemma}
If $q$ is an odd integer, then
$\alpha_{\mu}^{\lambda}=\alpha_{\mu^*}^{\lambda^*}$. 
\label{lem:jambeconj}
\end{lemma}

\begin{proof}
First, note that $c_{\mu^*}^{\lambda^*}=(c_{\mu}^{\lambda})^*$. 
In particular, the leg of
$c_{\mu^*}^{\lambda^*}$ is the arm of $c_{\mu}^{\lambda}$. Hence,
$L(c_{\mu}^{\lambda})+L(c_{\mu^*}^{\lambda^*})=q-1$. Since $q$ is odd, the
result follows.
\end{proof}

\begin{lemma}
Assume that $q$ is odd, and that $\lambda=\lambda^*$. 
The set $M_q(\lambda)$ contains a self-conjugate partition if and only if
$q\in\{\overline{\lambda}_1,\ldots,\overline{\lambda}_k\}$. In this
case, $M_q(\lambda)$ contains a unique self-conjugate partition $\mu$,
and $\mu$ is such that $\overline{\mu}=\overline{\lambda}\backslash \{q\}$.
\label{lem:uniqueselfcrochet}
\end{lemma}

\begin{proof}
Since $\lambda=\lambda^*$, it follows from Equation~(\ref{souv})
that $\lambda^i=(\lambda^{q-i+1})^*$, where
$\lambda^{(q)}=(\lambda^1,\ldots,\lambda^q)$ is the $q$-quotient of
$\lambda$. Moreover, by~\cite[Theorem 3.3]{olsson}
the multipartitions of $n-q$ obtained from $\lambda^{(q)}$ by
removing any $1$-hook are the $q$-quotients of partitions of
$M_{q}(\lambda)$. In particular, $\mu\in M_{q}(\lambda)$ is
self-conjugate if and only if $\mu^i=\lambda^i$ for $1\leq i\leq (q-1)/2$ and
$\mu^{(q+1)/2}$ is a self-conjugate 
partition obtained from $\lambda^{(q+1)/2}$ by
removing a $1$-hook. However, when we remove a $1$-hook
from a self-conjugate partition, the resulting partition is never a
self-conjugate partition, except if the removing box is a diagonal
$1$-hook. We now conclude with the argument of the proof
of~\cite[3.4]{BrGr}.
\end{proof}

Assume $\lambda=\lambda^*$. In the case that $M_q(\lambda)$ contains a
(unique) self-conjugate 
partition $\mu$, then we write
$\mu_{\lambda}=\mu$ (which is well-defined by
Lemma~\ref{lem:uniqueselfcrochet}).
Let $\mu,\,\mu'\in M_q(\lambda)$. We write $\mu\sim \mu'$ if and only if
$\mu'=\mu^*$, and we denote by $M'_q(\lambda)$ a set of representatives
modulo $\sim$.
Finally, for $\lambda\in\mathcal P_n$, we set $\alpha(\lambda)=1$ if
$\lambda\neq\lambda^*$ and $\alpha(\lambda)=\frac{1}{2}$ otherwise.

\begin{theorem}
Let $q$ an odd integer and $\lambda\in\mathcal P_n$. If
$\lambda\neq\lambda^*$, then we write
$\caralt_{\lambda}=\caralt_{\lambda}^+=\caralt_{\lambda}^-$.
Let $\sigma$ be a $q$-cycle with support $\{n-q+1,\ldots, n\}$. 
Then for $\epsilon\in\{\pm 1\}$
and $g\in\Alt_{n-q}$, we have
$$\caralt^{\epsilon}_{\lambda}(\sigma g)=\sum_{\mu\in M'_q(\lambda)\atop
\mu\neq\mu^*}a(\caralt_{\lambda}^{\epsilon},\caralt_{\mu})\,\caralt_{\mu}(g)
+\sum_{\mu\in M_q(\lambda)\atop \mu=\mu^*}\left(
a(\caralt_{\lambda}^{\epsilon},\caralt_{\mu}^+)\,\caralt_{\mu}^+(g)+
a(\caralt_{\lambda}^{\epsilon},\caralt_{\mu}^-)\,\caralt_{\mu}^-(g)\right),$$
where the complex numbers 
$a(\caralt_{\lambda}^{\epsilon},\caralt_{\mu}^{\eta})$ (
$\eta\in\{\pm 1\}$) are defined as follows.
\begin{enumerate}
\item[--] If $\mu^*\neq \mu$ and $\mu^*\in M_{q}(\lambda)$, then
$a(\caralt_{\lambda}^{\epsilon},\caralt_{\mu})=
\alpha(\lambda)(\alpha_{\mu}^{\lambda}+\alpha_{\mu^*}^{\lambda})$.
\item[--] If $\mu^*\neq \mu$ and $\mu^*\notin M_{q}(\lambda)$, then
$a(\caralt_{\lambda}^{\epsilon},\caralt_{\mu})=
\alpha(\lambda)\alpha_{\mu}^{\lambda}$.
\item[--] If $\mu^*=\mu$ and $\mu\neq\mu_{\lambda}$, then
$a(\caralt_{\lambda}^{\epsilon},\caralt_{\mu}^{\eta})=
\alpha(\lambda)\alpha_{\mu}^{\lambda}$.
\item[--] If $\mu^*=\mu$ and $\mu=\mu_{\lambda}$, 
then
$a(\caralt_{\lambda}^{\epsilon},\caralt_{\mu_{\lambda}}^{\eta})=
\frac{1}{2}\left(
\alpha_{\mu_{\lambda}}^{\lambda}+\epsilon\eta \sqrt{(-1)^{(q-1)/2}q}\right)$.
\end{enumerate}
\label{theo:MNAn}
\end{theorem}

\begin{proof}
This is a consequence of Clifford theory and the Murnaghan-Nakayama
formula in $\sym_n$. We only prove the last case for $\epsilon=+1$. 
Assume that
$\lambda=\lambda^*$
and that $\sigma g$ has cycle type $\overline{\lambda}$. By
Lemma~\ref{lem:uniqueselfcrochet}, $g\in\overline{\mu_{\lambda}}^{\pm}$.
Now, if $\overline{\lambda}=(h_1>h_2>\cdots>h_k)$ then
\begin{eqnarray*}
\caralt_{\lambda}^+(\sigma g)&=&\frac{1}{2}\left(\carsym_{\lambda}(\sigma
g)\pm\sqrt{(-1)^{\frac{n-k}{2}}h_1\cdots h_k}\ \right),\\
&=&\sum_{\{\mu,\mu^*\}\subset
M_q(\lambda)\atop\mu\neq\mu^*}\frac{1}{2}(\alpha_{\mu}^{\lambda}+\alpha_{\mu^*}^{\lambda})\caralt_{\mu}(g)+
\sum_{\mu\in
M_q(\lambda)\atop\mu^*\notin M_q(\lambda)\
\mu\neq\mu^*}\frac{1}{2}\alpha_{\mu}^{\lambda}\caralt_{\mu}(g)\\
&&+\sum_{\mu\in
M_q(\lambda)\atop
\mu=\mu^*}\frac{1}{2}\alpha_{\mu}^{\lambda}(\caralt_{\mu}^+(g)
+\caralt_{\mu}^-(g))
\pm\frac{1}{2}\sqrt{(-1)^{\frac{n-k}{2}}
h_1\cdots h_k},\\
\end{eqnarray*}
If we write $\overline{\lambda}\backslash\{q\}=(h'_1>\cdots>h'_{k-1})$,
then
\begin{eqnarray*}
\sqrt{(-1)^{\frac{n-k}{2}}h_1\cdots
h_k}&=&\sqrt{(-1)^{\frac{q-1}{2}}q}\cdot\sqrt{(-1)^{\frac{(n-q)-(k-1)}{2}}h'_1\cdots
h'_{k-1}},\\
&=&\sqrt{(-1)^{\frac{q-1}{2}}q}\cdot\left(\caralt_{\mu_{\lambda}}^+(g)-
\caralt_{\mu_{\lambda}}^-(g)\right).
\end{eqnarray*}
The result follows.
\end{proof}

\begin{remark}
In the last proof, when $\lambda=\lambda^*$ and $\mu\in\mathcal P_{n-q}$
is not self-conjugate and satisfies $\{\mu,\mu^*\}\subset M_q(\lambda)$,
then
$a(\caralt_{\lambda}^{\epsilon},\caralt_{\mu})=\alpha_{\mu}^{\lambda}$,
because, by Lemma~\ref{lem:jambeconj}, $\alpha_{\mu}^{\lambda}
=\alpha_{\mu^*}^{\lambda}$.
\end{remark}

For $q_1,\,q_2$ multiples of $p$, we define 
\begin{equation}
\label{eq:Mlambda2}
M_{q_1,q_2}(\lambda)=\{\mu\in M_{q_2}(\nu)\,|\,\nu\in
M_{q_1}(\lambda)\},
\end{equation} 
and $M'_{q_1,q_2}(\lambda)$ denotes a set of representatives modulo
$\sim$, where $\sim$ is defined in a similar way as before
Theorem~\ref{theo:MNAn}.
Moreover, for $\mu\in M_{q_1,q_2}(\lambda)$, we
denote by $\mathcal{P}_{\lambda\fleche\mu}^{q_1,q_2}$ the set all of
pairs 
$(c_{\nu}^{\lambda},c_{\mu}^{\nu})$, where $\nu\in M_{q_1}(\lambda)$ 
and $\mu\in M_{q_2}(\nu)$.

\begin{theorem}Assume that $q_1$ and $q_2$ are even multiples of $p$.
Let $\sigma=\sigma_1\sigma_2$ be such that $\sigma_i$ is a $q_i$-cycle
(for $1\leq i\leq 2$),
and the supports of $\sigma_1$ and
$\sigma_2$ are $\{n-q_1-q_2+1,\ldots, n-q_2\}$ and $\{n-q_2+1,\ldots, n\}$,
respectively.
Then for $\epsilon\in\{\pm 1\}$ and $g\in \Alt_{n-q_1-q_2}$,
$$\caralt^{\epsilon}_{\lambda}(\sigma g)=\sum_{\mu\in M'_{q_1,q_2}(\lambda)\atop
\mu\neq\mu^*}a(\caralt_{\lambda}^{\epsilon},\caralt_{\mu})\,\caralt_{\mu}(g)
+\sum_{\mu\in M_{q_1,q_2}(\lambda)\atop \mu=\mu^*}\left(
a(\caralt_{\lambda}^{\epsilon},\caralt_{\mu}^+)\,\caralt_{\mu}^+(g)+
a(\caralt_{\lambda}^{\epsilon},\caralt_{\mu}^-)\,\caralt_{\mu}^-(g)\right),$$
where the complex numbers 
$a(\caralt_{\lambda}^{\epsilon},\caralt_{\mu}^{\eta})$ ($\eta\in\{\pm
1\}$) are defined as follows:
if $\mu^*\neq \mu$ and $\mu^*\in M_{q_1,q_2}(\lambda)$, then
$$a(\caralt_{\lambda}^{\epsilon},\caralt_{\mu})=\alpha(\lambda)\left(
\sum_{(c_{\nu}^{\lambda},c_{\mu}^{\nu})\in\mathcal
P_{\lambda\fleche\mu}^{q_1,q_2}}(-1)^{L(c_{\nu}^{\lambda})+L(c_{\mu}^{\nu})}
+
\sum_{(c_{\nu}^{\lambda},c_{\mu^*}^{\nu})\in\mathcal
P_{\lambda\fleche\mu^*}^{q_1,q_2}}(-1)^{L(c_{\nu}^{\lambda})+L(c_{\mu^*}^{\nu})}\right).$$
In all other cases, one has
$$a(\caralt_{\lambda}^{\epsilon},\caralt_{\mu}^{\eta})=\alpha(\lambda)
\sum_{(c_{\nu}^{\lambda},c_{\mu}^{\nu})\in\mathcal
P_{\lambda\fleche\mu}^{q_1,q_2}}(-1)^{L(c_{\nu}^{\lambda})+L(c_{\mu}^{\nu})}.$$
\label{theo:MNAn2}
\end{theorem}

\begin{proof}
Apply twice the Murnaghan-Nakayama formula in $\sym_n$
and conclude with Clifford theory.
\end{proof}

Let $q$ be an integer.
For $\lambda\in\mathcal P_n$ and $\mu\in M_{q}(\lambda)$, we introduce the
relative $q$-sign $\delta_q(\lambda,\mu)=\delta_q(\lambda)\delta_q(\mu)$ as
in~\cite[p.\,62]{morrisolsson}, where $\delta_q(\lambda)$ is the
$q$-sign of $\lambda$ (see~\cite[\S2]{morrisolsson}).

\begin{lemma}Assume that $q$ is odd.
For any $\lambda\in\mathcal P_n$, one has
$\delta_q(\lambda)=\delta_q(\lambda^*)$.
\label{lem:impair}
\end{lemma}

\begin{proof}
Let $k$ be the $q$-weight of $\lambda$. 
We construct a sequence of $q$-hooks $c_1,\ldots, c_k$ by
choosing $c_1$ to be a $q$-hook of $\lambda$ and
$c_i$ to be a $q$-hook of
$\lambda\backslash\{c_1,\ldots,c_{i-1}\}$ for $2\leq i\leq k$, such that
$\lambda\backslash\{c_1,\ldots,c_k\}=\lambda_{(q)}$. Note that 
$c_1^*,\ldots,c_k^*$ is a sequence of $q$-hooks from $\lambda^*$ to
$\lambda^*_{(q)}$. So, by~\cite[Corollary 2.3]{morrisolsson},
$\delta_q(\lambda)=\delta_q(\lambda,\lambda^{(q)})$ and~\cite[Proposition 2.2]{morrisolsson} yields 
\begin{equation}
\label{eq:sommecrochet}
\delta_q(\lambda)=(-1)^{L(c_1)+\cdots+L(c_k)}\quad\textrm{and}\quad
\delta_q(\lambda^*)=(-1)^{L(c_1^*)+\cdots+L(c_k^*)}.
\end{equation}
Now, by the argument of Lemma~\ref{lem:jambeconj}, we deduce that 
$L(c_i)\equiv L(c_i^*)\mod 2$ for $1\leq i\leq k$, because $q$ is odd.
The result follows.
\end{proof}

Let $\gamma$ and $\gamma'$ be two self-conjugate $p$-cores of 
$\sym_n$ and $\sym_m$ of the same $p$-weight $w>0$. We denote by $b_{\gamma}$ and 
$b_{\gamma'}$ the corresponding $p$-blocks of $\Alt_n$ and $\Alt_m$,
respectively.
Let $\lambda\in\mathcal P_n$ be such that $\lambda_{(p)}=\gamma$. By
Equation~(\ref{eq:partquotientcoeur}),
there is a unique partition $\Psi(\lambda)\in\mathcal P_m$
\label{defpsi} such that 
$\Psi(\lambda)_{(p)}=\gamma'$ and $\Psi(\lambda)^{(p)}=\lambda^{(p)}$.
In particular, if we denote by $f$ the canonical bijection between the
set of hooks of length divisible by $p$ in $\lambda$ and the set of
hooks in $\lambda^{(p)}$ as in~\cite[Proposition 3.1]{morrisolsson},
then for any integer $q$ divisible by $p$ and $\mu\in M_q(\lambda)$,
we have
\begin{equation}
\label{eq:fcrochet}
f\left(c_{\mu}^{\lambda}\right)=f\left(c_{\Psi(\mu)}^{\Psi(\lambda)}\right),
\end{equation}
where $\Psi:\mathcal P_{n-q}\rightarrow\mathcal P_{m-q}$ is defined as
above.
Moreover,~\cite[Corollary 3.4]{morrisolsson} gives
\begin{equation}
(-1)^{L(c_{\mu}^{\lambda})}=(-1)^{L\left(f(c_{\mu}^{\lambda})\right)}
\delta_p(\lambda,\mu).
\label{eq:lienjambef}
\end{equation}

\begin{lemma}\label{Psi}
\label{lem:coherent} Let $\lambda$ and $\Psi(\lambda)$ be as above.
For any multiple $q$ of $p$ and $\mu\in
M_q(\lambda)$ such that $\mu\neq\mu^*$, we have
$\Psi(\mu)\neq\Psi(\mu^*)$. Moreover,
$\mu^*\in M_q(\lambda)$ if and only if $\Psi(\mu^*)\in
M_q(\Psi(\lambda))$. In this case, $\Psi(\mu^*)=\Psi(\mu)^*$.
\end{lemma}

\begin{proof}
This is a consequence of
\cite[Proposition 3.1]{morrisolsson} and of
\cite[Proposition 3.5]{olsson}.
\end{proof}

\begin{proposition}
Assume $p$ is odd and keep the notation as above. We have
$$\delta_p(\lambda)\delta_p(\mu)a\left
(\caralt_{\lambda}^{\epsilon\delta_p(\lambda)},
\caralt_{\mu}^{\eta\delta_p(\mu)}\right)=
\delta_p(\Psi(\lambda))\delta_p(\Psi(\mu))a\left
(\caralt_{\Psi(\lambda)}^{\epsilon
\delta_p(\Psi(\lambda))},
\caralt_{\Psi(\mu)}^{\eta\delta_p(\Psi(\mu))}\right).$$
\label{prop:coeffAn}
\end{proposition}

\begin{proof}
Let $\mathcal E_{\gamma}$ be the set of partitions of $n$ with
$p$-core $\gamma$.
Since $\gamma$ is self-conjugate, by
Equation~(\ref{souv}), $\lambda\in\mathcal E_{\gamma}$ is self-conjugate
if and only if its $p$-quotient is self-conjugate. The same holds for
$\gamma'$ and $\Psi(\lambda)\in\mathcal E_{\gamma'}$. In particular, 
for $\lambda\in \mathcal E_{\gamma}$, we have
$$\alpha(\lambda)=\alpha(\Psi(\lambda)).$$

Let $q_1$ and $q_2$ be even multiples of $p$. With the notation of
Theorem~\ref{theo:MNAn2}, for any partition
$\mu\in M_{q_1,q_2}(\lambda)$, if
$(c_{\nu}^{\lambda},c_{\mu}^{\nu})\in\mathcal
P_{\lambda\fleche\mu}^{q_1,q_2}$, then Equations~(\ref{eq:fcrochet})
and~(\ref{eq:lienjambef}) give
\begin{eqnarray*}
\delta_p(\lambda,\mu)(-1)^{L(c_{\nu}^{\lambda})+L(c_{\mu}^{\nu})}&=&\delta_p(\lambda,\mu)(-1)^{
L(f(c_{\nu}^{\lambda}))+L(f(c_{\mu}^{\nu}))}\delta_p(\lambda,\nu)\delta_p(\nu,\mu)
\\
&=&\delta_p(\lambda)\delta_p(\mu)(-1)^{
L(f(c_{\nu}^{\lambda}))+L(f(c_{\mu}^{\nu}))}\delta_p(\lambda)\delta_p(\nu)
\delta_p(\nu)\delta_p(\mu)\\
&=& (-1)^{
L\left(f(c_{\nu}^{\lambda})\right)+L\left(f(c_{\mu}^{\nu})\right)}\\
&=& (-1)^{
L\left(f(c_{\Psi(\nu)}^{\Psi(\lambda)})\right)+
L\left(f(c_{\Psi(\mu)}^{\Psi(\nu)})\right)}\\
&=&\delta_p(\Psi(\lambda),\Psi(\mu))(-1)^{L\left(c_{\Psi(\nu)}^{\Psi(\lambda)}\right)
+L\left(c_{\Psi(\mu)}^{\Psi(\nu)}\right)}.
\end{eqnarray*}
Now, using Lemmas~\ref{lem:impair} and~\ref{lem:coherent}, 
we deduce that, if $\mu\neq\mu^*$ and 
$\mu^*\in M_{q_1,q_2}(\lambda)$, then 
$$\delta_p(\lambda,\mu)
a(\caralt_{\lambda}^{\pm\delta_p(\lambda)},\caralt_{\mu})=
\delta_p(\Psi(\lambda),\Psi(\mu))
a(\caralt_{\Psi(\lambda)}^{\pm\delta_p(\Psi(\lambda))},\caralt_{\Psi(\mu)}).$$
\textbf{Note :} Lemma~\ref{lem:impair} is used only when we apply the above
computations to evaluate
$\delta_p(\lambda,\mu)(-1)^{L(c_{\nu}^{\lambda})+L(c_{\mu^{*}}^{\nu})}$;
at the second line, we get a term $\delta_p(\mu)\delta_p(\mu^*)$, which
is thus $1$. In the same way, a term
$\delta_p(\Psi(\mu))\delta_p(\Psi(\mu)^*)$ disappears at the end.

We conclude similarly for the other cases appearing in
Theorem~\ref{theo:MNAn2} and for the coefficients appearing 
in Theorem~\ref{theo:MNAn}, except for $\lambda=\lambda^*$ and
$\mu=\mu_{\lambda}$. In this last case, first note that
$\Psi(\mu)=\mu_{\Psi(\lambda)}$. Moreover, 
\begin{eqnarray*}
a(\caralt_{\lambda}^{\epsilon},\caralt_{\mu_{\lambda}}^{\eta})&=&
\frac{1}{2}\left(\alpha_{\mu_{\lambda}}^{\lambda}+\epsilon\eta\sqrt{(-1)^{(q-1)/2}q}\right)\\
&=&\delta_p(\lambda)\delta_p(\mu_{\lambda})\delta_p(\Psi(\lambda))\delta_p(\Psi(\mu_{\lambda}))\frac{1}{2}\left(\alpha_{\Psi(\mu_{\lambda})}^{\Psi(\lambda)}\right.\\
&&\left.+
\delta_p(\lambda)\delta_p(\mu_{\lambda})\delta_p(\Psi(\lambda))\delta_p(\Psi(\mu_{\lambda}))\epsilon\eta\sqrt{(-1)^{(q-1)/2}q}\right)\\
&=&\delta_p(\lambda)\delta_p(\Psi(\lambda))\delta_p(\mu_{\lambda})\delta_p(\Psi(\mu_{\lambda}) )a\left(\caralt_{\Psi(\lambda)}^{\epsilon\delta_p(\lambda)\delta_p(\Psi(\lambda))},
\caralt_{\Psi(\mu_{\lambda})}^{\eta\delta_p(\mu_{\lambda})\delta_p(\Psi(\mu_{\lambda}))}\right).
\end{eqnarray*}
as required.
\end{proof}

\begin{theorem}Let $p$ be an odd prime.
Assume that $\gamma$ and $\gamma'$ are self-conjugate $p$-cores of
$\sym_n$ and $\sym_m$ respectively, and of same $p$-weight $w>0$.  Let
$b_{\gamma}$ and $b_{\gamma'}$ be the corresponding $p$-blocks of
$\Alt_n$ and $\Alt_m$. Define, for all $\lambda\in \mathcal E_{\gamma}$
and $\epsilon\in\{\pm 1\}$,
\begin{equation}
\label{eq:defisoAn}
I:\C\Irr(b_{\gamma})\rightarrow\C\Irr(b_{\gamma'}),\
\caralt_{\lambda}^{\epsilon}\mapsto
\delta_p(\lambda)\delta_p(\Psi(\lambda))\caralt_{\Psi(\lambda)}^{\epsilon\delta_p(\lambda)\delta_p(\Psi(\lambda))}.
\end{equation}
Then $I$ is a Brou\'e perfect isometry.
\label{theo:mainAn}
\end{theorem}

\begin{proof}
First, we prove that $\Alt_n$ has an MN-structure. Let $S$ be the set
of elements of $\Alt_n$ with cycle decomposition
$\sigma_1\cdots\sigma_r$ (where we only indicate non-trivial cycles) such that each $\sigma_i$ is a
cycle of length divisible by $p$. We remark that when $\sigma_i$ has even
length, there is $j\neq i$ such that $\sigma_j$ has even length (because
$\sigma_1\cdots\sigma_r\in\Alt_n$). Moreover, $S$ contains $1$ and is stable
by $\Alt_n$-conjugation. Let $C$ be the set of $p$-regular elements
of $\Alt_n$.
Now take any $\sigma\in\Alt_n$. Using the cycle decomposition of $\sigma$, there are
unique elements $\sigma_S\in S$ and $\sigma_C\in C$ with disjoint
support such that
$\sigma=\sigma_S\sigma_C=\sigma_C\sigma_S$. In
particular, Definition~\ref{defMN}(1) and (2) hold. Denote
by $J$ the support of $\sigma_S$, $\overline{J}=\{1,\ldots,n\}\backslash
J$ and define $G_{\sigma_S}=\Alt_{\overline{J}}$. Then $G_{\sigma_S}$
satisfies Definition~\ref{defMN}(3). Write $\Omega$ for the set of
partitions of the form $p\cdot\beta$ such that 
\begin{itemize}
\item[--] There is some $i\leq n$ such that $p\cdot \beta$ is a
partition of $i$. 
\item[--] The number of even parts of $\beta$ is even. In particular, we
choose the notation such that
$\beta=(\beta_1,\ldots,\beta_k)$ with $|\beta|=\beta_1+\cdots+\beta_k$
and there is $1\leq r\leq k$ with $\beta_i$ even for $1\leq i\leq r$ and
$\beta_i$ odd for $i>r$.
\end{itemize}

\noindent
Note that each partition of $\Omega$ labels either one $\Alt_n$-conjugacy class of
$S$ or two classes. 
Denote by $\Lambda$ the set of parameters for the
$\Alt_n$-classes of $S$ obtained by this process.
The elements of $\Lambda$ will be denoted $p\cdot\widehat{\beta}$, with
$\widehat{\beta}=\beta$ when $p\cdot\beta\in\Omega$ labels a unique class of $S$,
and $\widehat{\beta}\in\{\beta^+,\beta^-\}$ when $p\cdot\beta$ labels two classes of
$S$. The notation is chosen as in Equation~(\ref{eq:Anchoix}).
For $p\cdot \widehat{\beta}\in\Lambda$, 
we assume that the representative 
$\sigma_{\widehat{\beta}}=\sigma_{\widehat{\beta}_1}\cdots\sigma_{\widehat{\beta}_k}$
of the $\Alt_n$-class labeled by $p\cdot\widehat{\beta}$ in $\Alt_n$ 
satisfies that the cycle
$\sigma_{\widehat{\beta}_i}$ has support
$\{1+\sum_{j<i} p\beta_j,\ldots,\sum_{j\leq i}p\beta_j\}$.
Moreover, when $p\cdot\beta$ labels two classes of $\Alt_n$, 
we assume that
$\sigma_{\beta^+_i}=\sigma_{\beta^-_i}$ for every
$2\leq i\leq k$, and $\sigma_{\beta^+_1}$ and $\sigma_{\beta^-_1}$ are
representatives of the $\Alt_{p\cdot\beta_1}$-classes labeled by $p\cdot
\beta_1^+$
and $p\cdot\beta_1^-$, respectively. In particular,
$\sigma_{\widehat{\beta}_i}$ has length $p\cdot\beta_i$ and
the support of $\sigma_{\widehat{\beta}}$ is $\{1,\ldots,p|\beta|\}$. 
Hence, $G_{\sigma_{\widehat{\beta}}}=\Alt_{n-p|\beta|}$.

Now, we denote by
$\Omega_0$ the subset of partitions $p\cdot\beta\in\Omega$ such that
$|\beta|\leq w$, and by $\Lambda_0$ the corresponding subset of
$\Lambda$. 
%
%
%
For $p\cdot\beta\in\Omega_0$, define
$r^{\widehat{\beta}}:\C\Irr(b_{\gamma})\rightarrow\C\Irr(b_{\gamma}(\Alt_{n-p|\beta|}))$
by applying iteratively Theorem~\ref{theo:MNAn} with
$\sigma=\sigma_{\widehat{\beta}_i}$ when $\beta_i$ is odd and
Theorem~\ref{theo:MNAn2} with
$\sigma=\sigma_{\widehat{\beta}_i}\sigma_{\widehat{\beta}_{i+1}}$ when
$\beta_{i}$ and $\beta_{i+1}$ are even. By Theorems~\ref{theo:MNAn}
and~\ref{theo:MNAn2}, Definition~\ref{defMN}(4) holds. This proves that
$\Alt_n$ has an MN-structure with respect to $b_{\gamma}$ and the set of
$p$-regular elements of $\Alt_n$.
Let $\lambda\in\mathcal E_{\gamma}$. Then
$r^{\widehat{\beta}}(\caralt_{\lambda}^{\pm})(g)=
\caralt_{\lambda}^{\pm}(\sigma_{\widehat{\beta}}g)$, and
for $p\cdot\widehat{\beta}\in\Lambda\backslash\Lambda_0$,
the Murnaghan-Nakayama rule in~$\sym_n$ and Clifford theory imply that 
$\caralt_{\lambda}^{\pm}(\sigma_{\widehat{\beta}}g)=0$ except, maybe, when
$\lambda=\lambda^*$ and
$\sigma_{\widehat{\beta}}\,g$ is in the class of $\sym_n$ labeled by
$a(\lambda)$. In this
last case, $\lambda$ has more than $w$ diagonal hooks with length
divisible by $p$, contradicting the fact that $\lambda$ 
has $p$-weight $w$.
This proves that, if $p\cdot\widehat{\beta}\notin \Lambda_0$, then
$r^{\widehat{\beta}}=0$.

We define similarly an MN-structure for $\Alt_m$ with respect to
$b_{\gamma'}$ and the set of $p$-regular elements of $\Alt_m$. We denote by
$\Omega'$, $\Omega'_0$, $\Lambda'$ and $\Lambda'_0$ the corresponding
sets. Note that $\Omega_0=\Omega'_0$.

There
are two cases to consider.
First, assume that $|\Lambda_0|=|\Lambda'_0|$. In fact, this case occurs
if and only if $\Lambda_0=\Lambda'_0$, because $\Omega_0=\Omega'_0$. 

Now, we will prove that
Theorem~\ref{th:iso}(2) holds.
Let $p\cdot \widehat{\beta}\in\Lambda_0$. Write $\beta=(\beta_1,\ldots,\beta_k)$
and $r$ as above. Set $q_i=p|\beta_i|$ for $1\leq i\leq k$.
For $i > r$, write $x_i=q_i$, and for
$1\leq i\leq r/2$, write $x_{i}=\{q_{2i-1},q_{2i}\}$. We 
also set $s=n-r/2$.
For $1\leq i\leq s$, define
$M_{x_1,\ldots,x_i}(\lambda)=\{\mu\in M_{x_i}(\nu)\,|\,
\nu\in M_{x_1,\ldots,x_{i-1}}(\lambda)\}$ 
(recall that $M_{x_i}(\nu)$ is defined as in
Equation~(\ref{eq:Mlambda2}) when $x_i$ has two elements). 

Let $\theta\in \Irr(b_{\gamma})$. There are $\lambda\in\mathcal E_{\gamma}$ and
$\epsilon\in\{\pm 1\}$ such that
$\theta=\caralt_{\lambda}^{\epsilon\delta_p(\lambda)}$ (with the
convention, as above, that if $\lambda\neq\lambda^*$, then
$\caralt_{\lambda}^+=\caralt_{\lambda}=\caralt_{\lambda}^-$). Then we
set $\delta_{p}(\theta)=\delta_p(\lambda)$ and
$\Psi(\theta)=\rho_{\Psi(\lambda)}^{\epsilon\delta_p(\Psi(\lambda))}\in\Irr(
b_{\gamma'})$.
We have
\begin{equation}
\label{eq:dec}
r^{\widehat{\beta}}(\theta)=\sum_{\vartheta\in
\Irr(b_{\gamma}(n-p|\beta|))}a(\theta,\vartheta)\,\vartheta,
\end{equation}
where $b_{\gamma}(n-p|\beta|)$ denotes the union of $p$-blocks of
$\Alt_{n-p|\beta|}$ covered by the $p$-block $B_\gamma$ of
$\sym_{n-p|\beta|}$ labeled by $\gamma$. By~\S\ref{subsec:pblockAn}, 
$b_{\gamma}(n-p|\beta|)$ is a $p$-block of $\Alt_{n-p|\beta|}$, except 
when $|\gamma|>2$ and $|\beta|=w$. In this last case, it is a union of
two $p$-blocks $\{\rho_{\gamma}^+\}$ and $\{\rho_{\gamma}^-\}$ of defect
zero.
Similarly, we denote by $b_{\gamma'}(m-p|\beta|)$ the union of
$p$-blocks of $\mathcal A_{m-p|\beta|}$ covered by the $p$-block of
$B_{\gamma'}$ of $\sym_{m-p|\beta|}$ labeled by $\gamma'$. Define
$I_{\beta}:\C\Irr(b_{\gamma}(n-p|\beta|))\rightarrow
\C\Irr(b_{\gamma'}(m-p|\beta|))$ as in
Equation~(\ref{eq:defisoAn}).
Note that
\begin{equation}
\label{eq:coeffAitere}
a(\theta,\vartheta)=\sum_{\vartheta_1,\ldots,\vartheta_{s-1}}a(\vartheta_0,\vartheta_1)a(\vartheta_1,\vartheta_2)\cdots
a(\vartheta_{s-1},\vartheta_{s}),
\end{equation}
where $\vartheta_0=\theta$, $\vartheta_{s}=\vartheta$, and the sum runs
over the set of $\vartheta_1,\ldots,\vartheta_{s-1}$ such that 
for each $1\leq i\leq s$, there are $\mu_i\in
M_{x_1,\ldots,x_i}(\lambda)$ and $\epsilon_i\in\{\pm 1\}$ such that
$\vartheta_i=\rho_{\mu^i}^{\epsilon_i\delta_p(\mu_i)}$.
Since
$$\delta_p(\theta)\delta_p(\vartheta)=(\delta_p(\theta)\delta_p(\vartheta_1))\cdot(\delta_p(\vartheta_1)\delta_p(\vartheta_2))\cdots(\delta_p(\vartheta_{s-1})\delta_p(\vartheta)),$$
and thanks to Proposition~\ref{prop:coeffAn}, we deduce that
\begin{equation}
\label{eq:acommute}
\begin{split}
\delta_p(\theta)\delta_p(\vartheta)a(\theta,\vartheta)
&=\sum\delta_{p}(\vartheta_0)\delta_p(\vartheta_1)a(\vartheta_0,\vartheta_1)\cdots\\
&\qquad\cdots\delta_p(\vartheta_{s-1})\delta_p(\vartheta_{s})
a(\vartheta_{s-1},\vartheta_{s})\\
&=\sum\delta_{p}(\Psi(\vartheta_0))\delta_p(\Psi(\vartheta_1))a(\Psi(\vartheta_0),\Psi(\vartheta_1))\cdots\\
&\qquad\cdots\delta_p(\Psi(\vartheta_{s-1}))\delta_p(\Psi(\vartheta_{s}))a(\Psi(\vartheta_{s-1}),\Psi(\vartheta_{s}))\\
&=
\delta_p(\Psi(\theta))\delta_p(\Psi(\vartheta))a(\Psi(\theta),\Psi(\vartheta)).
\end{split}
\end{equation}
In particular, one has
$$\delta_p(\theta)\delta_p(\Psi(\theta))a(\Psi(\theta),\Psi(\vartheta))=
\delta_p(\vartheta)\delta_p(\Psi(\vartheta))\,a(\theta,\vartheta),$$
and it follows that
\begin{equation}
\label{eq:commuteAn}
\begin{split}
r^{\widehat{\beta}}(I(\theta))\quad&=\quad\delta_p(\theta)\delta_p(\Psi(\theta))r^{\widehat{\beta}}(\Psi(\theta)),\\
&=
\sum_{\vartheta\in
\Irr(b_{\gamma}(n-p|\beta|))}
\delta_p(\theta)\delta_p(\Psi(\theta))a(\Psi(\theta),\Psi(\vartheta))\,
\Psi(\vartheta),\\
&=\sum_{\vartheta\in
\Irr(b_{\gamma}(n-p|\beta|))}
\delta_p(\vartheta)\delta_p(\Psi(\vartheta))a(\theta,\vartheta)\,
\Psi(\vartheta),\\
&=\sum_{\vartheta\in
\Irr(b_{\gamma}(n-p|\beta|))}
a(\theta,\vartheta)\,
I_{\widehat{\beta}}(\vartheta),\\
&=\quad I_{\widehat{\beta}}(r^{\widehat{\beta}}(\theta)).
\end{split}
\end{equation}

\noindent \textbf{Note :} Assume $H$ is a normal subgroup of $G$ and 
the MN-structure of $H$ comes from Clifford theory. Then the $\mathcal
E_{x_{\lambda},y}^{\lambda}$ for $G$ have all size $1$ and the $\mathcal
E_{x_{\lambda},y}^{\lambda}$ for $H$ have size dividing this size for
$G$ multiplied by $[G:H]$. 
Since here this index is $2$, and since 
$p$ is odd, condition (iii) of Theorem~\ref{th:broue} holds. 

Note that the groups $G_{\sigma_{\widehat{\beta}}}$ and
$G'_{\sigma'_{\widehat{\beta}}}$ have an MN-structure with respect to
$$(b_{\gamma}({n-p|\beta|}),C\cap
G_{\sigma_{\widehat{\beta}}})\quad\textrm{ and }\quad 
(b'_{\gamma'}({m-p|\beta|}),C\cap
G'_{\sigma'_{\widehat{\beta}}}),$$ respectively. Applying the previous
computations to $G_{\sigma_{\widehat{\beta}}}=\mathcal A_{n-p|\beta|}$ and 
$G'_{\sigma'_{\widehat{\beta}}}=\mathcal A_{m-p|\beta|}$, we conclude that 
the condition (2) of Theorem~\ref{th:iso} holds for $I_{\widehat{\beta}}$. 
Now, Remark~\ref{rk:point3} gives the condition (3) of Theorem~\ref{th:iso}
for $I$. On
the other hand, by construction, assumption (ii) of Theorem~\ref{th:broue}
holds. The result then follows from Theorems~\ref{th:iso} and~\ref{th:broue}.

Assume now that $|\Lambda_0|\neq|\Lambda_0'|$. Without loss of
generality, we can suppose that $|\Lambda_0|>|\Lambda'_0|$. 
This means that $\Lambda'_0=\Omega_0$. 
Since $\Lambda_0\leq \Omega_0$, there is $p\cdot\beta\in\Omega_0$ such
that $(p\cdot\beta,1^{|\gamma|})\in\mathcal D_n\cap\mathcal O_n$. This
$p\cdot \beta$ also belongs to $\Lambda_0'$ and since
$\Lambda_0=\Omega_0$, $p\cdot\beta$ labels a unique conjugacy class of
$\Alt_{m}$, \emph{i.e.} $(p\cdot \beta,1^{|\gamma'|})\notin\mathcal D_m\cap
\mathcal O_m$. This happens if and only if $|\gamma'|\geq 2$.
(In fact, $|\gamma'|\geq 3$ because
$\gamma'$ is self-conjugate.) Since $\gamma'$ is self-conjugate, it
labels two irreducible characters $\rho_{\gamma'}^+$ and
$\rho_{\gamma'}^-$ of $\Alt_{m-pw}$.
Similarly, $\Lambda_0\neq\Omega_0$ implies 
that $|\gamma|\leq 1$. Note that in this case, although $\gamma$ is
self-conjugate, the restriction of $\chi_{\gamma}$ from $\sym_1$ (or
$\sym_0$) to $\Alt_1$ (or $\Alt_0$) is irreducible (because it is the
trivial character of the trivial group).
Let $p\cdot\beta$ be in $\Omega_0$. 

Suppose that $|\beta|<w$ or that
$|\beta|=w$ and $(p\cdot\beta,1^{|\gamma|})\notin\mathcal D_n\cap\mathcal
O_n$. Then
$p\cdot\beta\in \Lambda_0$.
The same computation as in
Equation~(\ref{eq:commuteAn}) gives 
\begin{equation}
\label{eq:comm1}
r^{\widehat{\beta}}\circ
I=I_{\beta}\circ r^{\widehat{\beta}}.
\end{equation}

Suppose now that $|\beta|=w$ and $(p\cdot\beta,1^{|\gamma|}) \in\mathcal
D_n\cap\mathcal O_n$. Then $p\cdot\beta$ parametrizes two classes of $S$
and one class of $S'$. Moreover, $|\Irr(b_{\gamma}(n-pw))|=1$ and
$|\Irr(b_{\gamma'}(m-pw))|=2$. Denote by
$G_{\beta^+}$ and $G_{\beta^-}$ two copies of the trivial group,
and set $\Irr(G_{\beta^{\pm}})=\{1_{\beta^\pm}\}$. In particular, $r^{\beta^{\pm}}(\C b_{\gamma})=\C\Irr(G_{\beta^{\pm}})$.
Define $I_{\beta}:\C\Irr(G_{\beta^{+}})\oplus\C\Irr(G_{\beta^-})\longrightarrow \C
\Irr(b_{\gamma'}(m-wp))$ by setting
\begin{equation}
\label{eq:defIbetabizarre}
I_{\beta}(1_{\beta^+})=\rho_{\gamma'}^+\quad\textrm{and}\quad
I_{\beta}(1_{\beta^-})=\rho_{\gamma'}^-.
\end{equation}

Let $\kappa$ be the self-conjugate partition of $n$ 
whose diagonal hook lengths are
the parts of the partition $(p\cdot\beta,1^{|\gamma|})$. 
By~\cite[3.4]{BrGr}, the $p$-quotient of
$\kappa$ satisfies $\kappa^i=\emptyset$ if $i\neq
(p+1)/2$ and  $\kappa^{(p+1)/2}=\beta_0$, where $\beta_0$ is the
partition of $w$ such that $a(\beta_0)=\beta$.
By the definition of $\Psi$, the partition $\Psi(\kappa)$ of $m$ has
the same $p$-quotient as $\kappa$. Thus, the proof of~\cite[3.4]{BrGr} 
also implies
that $\Psi(\kappa)$ has the same diagonal hook lengths
divisible by $p$
as $\kappa$, and the other diagonal hooks of $\Psi(\kappa)$ 
have $p'$-length. In particular,
$a(\Psi(\kappa))$ has $p\cdot\beta$ as a
subpartition (corresponding exactly to those of the parts of
$a(\Psi(\kappa))
$ that are divisible by $p$).
On the other hand, the $\sym_m$-class labeled by $a(\Psi(\kappa))$
 splits into two $\Alt_m$-classes with
representatives $\sigma'_{\beta}\sigma^+$ and
$\sigma'_{\beta}\sigma^-$, where the cycle type of $\sigma'_{\beta}$ 
is $p\cdot\beta$, and the $p$-regular elements $\sigma^+$ and
$\sigma^-$ are representatives of the split classes of
$\Alt_{m-p|\beta|}$ labeled by $a(\gamma')^{+}$ and $a(\gamma')^{-}$,
respectively. 

Let $\mu_1=\kappa$, $\mu_{\ell(\beta)}=\gamma$ and the $\mu_i$'s be partitions 
such that $\mu_1\fleche\mu_2\fleche\cdots\fleche\mu_{\ell(\beta)}$, where
$\mu_i$ is obtained from $\mu_{i-1}$ by removing the diagonal hook of
length $p\beta_i$. Since
$L(c_{\mu_i}^{\mu_{i-1}})=L(c_{\Psi(\mu_i)}^{\Psi(\mu_{i-1})})$ for every
$1\leq i\leq \ell(\beta)-1$, Equations~(\ref{eq:fcrochet})
and~(\ref{eq:lienjambef}) give
$\delta_p(\mu_i,\mu_{i+1})=\delta_p(\Psi(\mu_i),\Psi(\mu_{i+1}))$ 
and it follows from~\cite[Corollary 2.3]{morrisolsson} that
\begin{equation}
\label{eq:signedeltap}
\delta_p(\kappa)=\prod_{i=1}^{\ell(\beta)-1}\delta_p(\mu_i,\mu_{i+1})=
\prod_{i=1}^{\ell(\beta)-1}\delta_p(\Psi(\mu_i),\Psi(\mu_{i+1}))=
\delta_p(\Psi(\kappa)).
\end{equation}
Now, by~\cite[Theorem 11]{Enguehard}, we have
\begin{equation*}
\begin{split}
r^{\beta}(\chi_{\Psi(\kappa)})&=\quad
r^{\beta}\left(\delta_p(\kappa)\delta_p(\Psi(\kappa))\chi_{\Psi(\kappa)}
\right)\\
&=\quad r^{\beta}\circ I(\chi_{\kappa})\\
&=\quad I\circ
r^{\beta}(\chi_{\kappa})\\
&=\quad \chi_{\kappa}(x_{\beta})I(1_{\{1\}})\\
&=\quad \chi_{\kappa}(x_{\beta})\chi_{\gamma'},
\end{split}
\end{equation*}
where $x_{\beta}$ denotes a representative of the $\sym_n$-class
labeled by $(p\cdot\beta,1^{|\gamma|})$.
Furthermore, Clifford theory gives 
\begin{equation}
\label{eq:rhosomme}
r^{\beta}\left(\caralt_{\Psi(\kappa)}^+\right)+
r^{\beta}\left(\caralt_{\Psi(\kappa)}^-\right)
=\carsym_{\kappa}(x_{\beta})\left(\caralt_{\gamma'}^+
+\caralt_{\gamma'}^-\right),
\end{equation}

For $1\leq i\leq \ell(\beta)$, write $q_i=p\beta_i$.
Then we
have   
$$(-1)^{\frac{1}{2}(n-\ell(
(p\cdot\beta,1^{|\gamma|}))}=(-1)^{\frac{1}{2}(pw-\ell(\beta))}$$
and the product of the parts of $(p\cdot\beta,1^{|\gamma|})$ is $q_1\ldots q_{\ell(\beta)}$. 
Thus, Theorem~\ref{theo:MNAn} gives
\begin{equation}
\label{eq:rhodiff}
\begin{split}
r^{\beta}(\caralt_{\Psi(\kappa)}^+)-r^{\beta}(\caralt_{\Psi(\kappa)}^-)
&=\quad\sqrt{(-1)^{\frac{1}{2}\sum(q_i-1)}q_1\cdots
q_{\ell(\beta)}}\left(\caralt_{\gamma'}^+-\caralt_{\gamma'}^-\right)\\
&=\quad 2 y_{\kappa}\left(\caralt_{\gamma'}^+-\caralt_{\gamma'}^-\right),
\end{split}
\end{equation}
because $\sum(q_i-1)=pw-\ell(\beta)$. 
So, we deduce from Equations~(\ref{eq:rhosomme}) and~(\ref{eq:rhodiff})
that
\begin{equation}
\label{eq:rbeta}
r^{\beta}(\rho_{\Psi(\kappa)}^{\epsilon})=(x_{\kappa}+\epsilon
y_{\kappa})\rho_{\gamma'}^+ + (x_{\kappa}-\epsilon
y_{\kappa})\rho_{\gamma'}^-.
\end{equation}
Furthermore, one has
\begin{equation}
\label{eq:rbetaplusmoins}
r^{\beta^+}(\rho_{\kappa}^{\epsilon})=(x_{\kappa}+\epsilon
y_{\kappa}) 1_{\beta^+}\quad\textrm{and}\quad
r^{\beta^-}(\rho_{\kappa}^{\epsilon})=(x_{\kappa}-\epsilon
y_{\kappa}) 1_{\beta^-}.
\end{equation}
Hence, Equations~(\ref{eq:defIbetabizarre}),
(\ref{eq:signedeltap}),~(\ref{eq:rbeta}) and~(\ref{eq:rbetaplusmoins})
give
$$I_{\beta}\left(r^{\beta^+}(\rho_{\kappa}^{\epsilon})+r^{\beta^-}(\rho_{\kappa}^{\epsilon})\right)=r^{\beta}\left(I(\rho_{\kappa}^{\epsilon})\right).$$
Let now $\lambda\neq \kappa$ be with $p$-core $\gamma$.
Since
$\caralt_{\lambda}^{\epsilon}(\sigma_{\beta^{\pm}})=\alpha(\lambda)\chi_{\lambda}(x_{\beta})$,
we derive from
~\cite[Theorem 11]{Enguehard} and Clifford theory that
$I_{\beta}(r^{\beta^+}(\rho_{\lambda}^{\epsilon})+
r^{\beta^-}(\rho_{\lambda}^{\epsilon}))=
r^{\beta}(I(\rho_{\lambda}^{\epsilon}))$.
Finally, we obtain
\begin{equation}
\label{eq:com2}
I_{\beta}\circ(r^{\beta^+}+r^{\beta^-})=r^{\beta}\circ I.
\end{equation}
Using Equations~(\ref{eq:comm1}) and~(\ref{eq:com2}),
Remark~\ref{rk:point3} holds. Hence, the condition (2.b) of
Theorem~\ref{th:iso} is automatic for $I_{\widehat \beta}$ with
$|\beta|<w$, and is true for $I_{\beta}$ with $|\beta|=w$ (because the
characters $1_{\beta^{\pm}}$ and $\rho_{\gamma'}^{\pm}$ have defect
zero). We remark that in the last case, with the notation of
Theorem~\ref{th:iso}, one has
${J_{\beta}^*}^{-1}=I_{\beta}$.

Write $A$ and $B$ for the sets of partitions $\beta\in\Omega_0$ such
that $\beta=\widehat{\beta}$ and  $\beta\neq\widehat{\beta}$,
respectively.
Now, following step by step 
the proof of Theorem~\ref{th:iso}, 
we obtain
\begin{equation}
\begin{split}
\widehat{I}(x,x')=&\sum_{\beta\in A}\sum_{\phi\in
\mathfrak
b_{\beta}}\overline{e_{{\beta}}(\Phi_{\phi})(x)}l'_{\beta}({J_{\beta}^*}{^{-1}}(\phi))(x')\\&+\sum_{\beta\in
B}\sum_{\delta\in\{+,-\}}\overline{e_{\beta^\delta}(1_{\beta^\delta})}(x)
l'_{\beta}({J_{\beta}^*}{^{-1}}(1_{\beta^\delta})(x'),
\end{split}
\label{eq:guguAn}
\end{equation}
where $\mathfrak b_{\beta}$ (for $\beta\in A)$)
is the basis constructed from 
the set of irreducible Brauer characters
in the $p$-block $b_{\gamma}(n-p|\beta|)$ as in
Remark~\ref{rk:baseadaptee}. 
Since an analogue of Remark~\ref{rk:adjoint} holds, we conclude as in
Theorem~\ref{th:broue}.
\end{proof}

\begin{theorem}Let $p$ be an odd prime.
Assume that $\gamma$ and $\gamma'$ are non self-conjugate $p$-cores of
$\sym_n$ and $\sym_m$ respectively, of same $p$-weight $w>0$.  Let
$b_{\gamma,\gamma^*}$ and $b_{\gamma',\gamma'^*}$ be the corresponding
$p$-blocks of $\Alt_n$ and $\Alt_m$. Let
$$I:\C\Irr(b_{\gamma,\gamma^*})\rightarrow\C\Irr(b_{\gamma',\gamma'^*}),\
\caralt_{\lambda}\mapsto
\delta_p(\lambda)\delta_p(\Psi(\lambda))\caralt_{\Psi(\lambda)}.$$
Then $I$ is a Brou\'e perfect isometry.
\label{theo:mainAn2}
\end{theorem}

\begin{proof}
The proof is similar to that of Theorem~\ref{theo:mainAn}. We use the
same MN-structure. In a sense, this case is easier, because every
irreducible character in $\Irr(b_{\gamma,\gamma^*})$ is the restriction
of a character of $\sym_n$. Hence, the Murnaghan-Nakayama rule for $\sym_n$  
directly gives the result. 
\end{proof}

\begin{theorem}
Let $\gamma$ and $\gamma'$ be $2$-blocks of $\Alt_n$ and $\Alt_m$ of the
same positive weight. Then $I$ defined as in
Equation~(\ref{eq:defisoAn}) is a Brou\'e perfect isometry.
\label{theo:mainAnpegal2}
\end{theorem}

\begin{proof}
The MN-structure is defined as in the case where $p$ is odd, and one
always has that $\Lambda_0=\Omega_0=\Lambda'_0$. 
Note that $I$ satisfies the assumption of Remark~\ref{rk:prop3bis}.
Only the situation of Theorem~\ref{theo:MNAn2} occurs. The result of
Proposition~\ref{prop:coeffAn} still holds, but the simplifications
explained in the note within the proof are different. For any $2$-hook
$c$, one has $L(c) + L(c^*)\equiv 1\mod 2$. In particular,
for any $\mu\in M_{q_1,q_2}(\lambda)$, we deduce from
Equation~(\ref{eq:sommecrochet}) that
$$\delta_2(\mu)\delta_2(\mu^*)=(-1)^r
=\delta_2(\Psi(\mu))\delta_2(\Psi(\mu)^*),$$
where $r$ is the number of $2$-hooks to remove from $\mu$ to get to
$\mu_{(2)}$ (this is also the number of $2$-hooks we have to remove from
$\Psi(\mu)$ to obtain $\Psi(\mu)_{(2)}$).
The rest of the proof is similar to that of Theorem~\ref{theo:mainAn}.
\end{proof}

\section{Double covering groups of the symmetric and alternating groups}
\label{section:SnTilde}

In this section, we will consider the double covering group
$\widetilde{\Sym}_n$ (for a positive integer $n$) of the symmetric group
$\sym_n$ defined by
$$\widetilde{\Sym}_n=\left\langle z,\, t_i,\,1\leq i\leq n-1\ |\
z^2=1,\,t_i^2=z,\,(t_it_{i+1})^3=z,\,(t_it_j)^2=z\ (|i-j|\geq 2)
\right\rangle.$$
The group $\tSym_n$ and its representation theory were first studied by I. Schur in \cite{schur}, and, unless otherwise specified, we always refer to \cite{schur} for details or proofs.

\noindent
We recall 
that we have the following exact sequence
$$1\rightarrow \cyc{z}\rightarrow \widetilde{\Sym}_n\rightarrow
\sym_n\rightarrow 1.$$ We denote by
$\theta:\widetilde{\Sym}_n\rightarrow\sym_n$ the natural projection. Note
that for every $\sigma\in\sym_n$, we have
$\theta^{-1}(\sigma)=\{\widetilde{\sigma},z\widetilde{\sigma}\}$, where
$\widetilde{\sigma}\in\widetilde{\Sym}_n$ is such that 
$\theta(\widetilde{\sigma})=\sigma$.

\noindent
If we set
$$\tAlt_n=\theta^{-1}(\Alt_n),$$
then $\tAlt_n$ is the double covering group of the alternating group $\Alt_n$.

\medskip
Throughout this section, we fix an odd prime $p$.

\subsection{Conjugacy classes and spin characters of $\tSym_n$}


If $x,\,y\in\widetilde{\Sym}_n$ are $\widetilde{\Sym}_n$-conjugate, then
$\theta(x)$ and $\theta(y)$ are $\sym_n$-conjugate. Let
$\sigma,\,\tau\in\sym_n$. Choose
$\widetilde{\sigma},\, \widetilde{\tau}\in\widetilde{\Sym}_n$ such that
$\theta(\widetilde{\sigma})=\sigma$ and $\theta(\widetilde{\tau})=\tau$.
Suppose that $\sigma$ and $\tau$ are $\sym_n$-conjugate. Then
$\widetilde{\tau}$ is $\tSym_n$-conjugate to $\widetilde{\sigma}$ or
to $z\widetilde{\sigma}$ (possibly both). Hence, each conjugacy class $C$
of $\sym_n$ gives rise to either one or two conjugacy classes of $\widetilde{\Sym}_n$,
according to whether $\widetilde{\sigma}$ and $z\widetilde{\sigma}$ are
conjugate or not (here, $\sigma$ lies in $C$ and $\widetilde{\sigma}$ is
as above). In the first case, we say that the class is {\emph{non-split}}, and, in the second case, that it is \emph{split}. The split classes of
$\widetilde{\Sym}_n$ are characterized as follows. Recall that the conjugacy classes of
$\sym_n$ are labeled by the set $\mathcal{P}_n$ of partitions of $n$.
Write $\mathcal{O}_n$ for the set of $\pi\in\mathcal{P}_n$ such that
all parts of $\pi$ have odd length, and $\mathcal{D}_n$ for the set of
$\pi\in\mathcal{P}_n$ with distinct parts. The partitions in
$\mathcal D_n$ are called {\emph{bar partitions}}. Denote by $\mathcal{D}^+_n$
(respectively $\mathcal{D}^-_n$) the subset of $\mathcal{D}_n$ consisting of all
partitions $\pi\in \mathcal{D}_n$ such that the number of parts of
$\pi$ with an even length is even (respectively odd).  Schur proved (see~\cite[\S7]{schur})

\begin{proposition}
The split conjugacy classes of $\widetilde{\Sym}_n$ are those classes $C$ such that
$\theta(C)$ is labeled by $\mathcal{O}_n\cup \mathcal{D}^-_n$.
\label{prop:classschur}
\end{proposition}

We set $s_i=(i,i+1)\in\sym_n$. Then for every $1\leq i\leq n-1$, we have
$\theta(t_i)=s_i$. For $\pi=(\pi_1,\ldots,\pi_k)\in\mathcal{P}_n$,
write $s_{\pi}$ for a representative of the class of $\sym_n$ labeled
by $\pi$. If $s_{\pi}=s_{\pi_1}\cdots s_{\pi_k}$ is the
cycle decomposition (with disjoint supports)
of $s_{\pi}$, then we assume that the support of $s_{\pi_i}$ is 
\begin{equation}
\label{eq:suppSn}
\left\{1+\sum_{j<i} \pi_j,\ldots,\sum_{j\leq i} \pi_j\right\}.
\end{equation}
%

Now, for any $\pi\in\mathcal P_n$, we make the same choice of Schur~\cite[\S11]{schur} for a representative $t_{\pi}\in\widetilde{\sym}_n$ 
such that $\theta(t_{\pi})=s_{\pi}$. 
So, when $\pi\in\mathcal
O_n\cup\mathcal D_n^-$, the elements $t_{\pi}$ and $zt_{\pi}$ are
representatives of the two split classes of $\widetilde{\sym}_n$
labeled by $\pi$. We denote by $C_{\pi}^+$ (respectively $C_{\pi}^-$) 
the conjugacy class of $t_{\pi}$ (respectively $z t_{\pi}$) in $\tSym_n$. 
It will also sometimes be convenient to write $t_{\pi}^+$ for $t_{\pi}$, 
and $t_{\pi}^-$ for $z t_{\pi}$.
When $\pi \in \cal{P}_n \setminus (  \mathcal{O}_n\cup
\mathcal{D}^-_n)$, the elements 
$t_{\pi}$ and $z t_{\pi}$ belong to the same conjugacy class $C_{\pi}$ of 
$\tSym_n$. In all cases, an element $g$ (or an $\tSym_n$-class $C$) is said
to be of \emph{type} $\pi$ if the $\sym_n$-class of $\theta(g)$ (respectively of
$\theta(g)$ for any $g\in C$)  is labeled by $\pi$. 

Note that if $\pi=(\pi_1,\ldots,\pi_k)\in \mathcal P_n$, then the
construction of~\cite[III p.\,172]{schur} implies that
\begin{equation}
t_{\pi}=t_{\pi_1}\cdots t_{\pi_k}.
\label{eq:schurcycle}
\end{equation}
 
\begin{convention}
\label{conv}
Let $\pi=(\pi_1,\ldots,\pi_k)\in\mathcal P_n$. 
In the following, we do not necessarily assume as usual that $\pi_1\geq \cdots\geq \pi_k$.
Instead, we assume that the parts of $\pi$ are ordered in such a way that $\pi_1\geq \cdots\geq \pi_u$ and
$\pi_{u+1}\geq \cdots\geq \pi_k$, where $u$ is such that $\pi_{u+1}, \, \ldots, \,  \pi_k$ are exactly the odd parts of $\pi$ 
which are divisible by $p$ (if there is no such part in $\pi$, then $u=k$).

\end{convention}

%
%
%
%

\medskip



We are now interested in the set of irreducible complex characters of
$\tSym_n$. Any irreducible (complex) character of $\tSym_n$ with $z$ in
its kernel is simply lifted from an irreducible character of the
quotient $\Sym_n$. Any other irreducible character $\xi$ of $\tSym_n$ is
called a {\emph{spin character}}, and it satisfies $\xi(z)=-\xi(1)$. In
particular, for any spin character $\xi$ and any $\pi\in\mathcal{P}_n$,
one has $\xi(zt_{\pi})=-\xi(t_{\pi})$, which implies that $\xi$ vanishes
on the non-split conjugacy classes of $\widetilde{\Sym}_n$.

Define $\varepsilon=\operatorname{sgn}\circ \, \theta$, where
$\operatorname{sgn}$ is the sign character of $\sym_n$. Note that
$\widetilde{\Alt}_n=\ker(\varepsilon)$. Then $\varepsilon$ is a linear
(irreducible) character of $\widetilde{\Sym}_n$, and for any spin
character $\xi$ of $\widetilde{\Sym}_n$, $\varepsilon\otimes\xi$ is a
spin character (because
$\varepsilon\otimes\chi(z)=-\varepsilon\otimes\chi(1)$).  A spin
character $\xi$ is said to be {\emph{self-associate}} if $\varepsilon
\xi=\xi$. Otherwise, $\xi$ and $\varepsilon \xi$ are called
{\emph{associate characters}}. It follows that, if $\xi$ is
self-associate and $\pi\in \mathcal{D}^-_n$, then
$\xi(t_{\pi})=0=\xi(zt_{\pi})$. 

In~\cite{schur}, Schur proved that the spin characters of $\tSym_n$ are,
up to association, labeled by $\mathcal{D}_n$. More precisely, he
showed that every $\lambda\in\mathcal{D}^+_n$ indexes a self-associate
spin character $\xi_{\lambda}$, and every $\lambda\in\mathcal{D}^-_n$ a
pair $(\xi_{\lambda}^{+} , \xi_{\lambda}^-)$ of associate spin
characters. In this case, we will sometimes write $\cla$ for $\cla^+$,
so that $\cla^-=\e \cla$. 

For any partition $\lambda=(\lambda_1,\ldots,\lambda_k)$ of $n$ (where we don't include 0 parts), we set $|\lambda|=\sum \lambda_i$
and we define the {\emph{length}} $\ell(\la)$ of $\lambda$ by
$\ell(\lambda)=k$.  If $\lambda$ is furthermore a bar partition (i.e.
if the parts of $\lambda$ are pairwise distinct), then we set
$\sigma(\lambda)=(-1)^{|\lambda|-\ell(\lambda)}$. With this notation, we
then have (see e.g.~\cite[p. 45]{olsson})
\begin{equation}
\lambda\in\mathcal D_n^{\sigma(\lambda)}.
\label{eq:self}
\end{equation}
If $\sa(\la)=1$, then $\la$ is said to be {\emph{even}}; otherwise, it is said to be {\emph{odd}}.

Schur proved in~\cite{schur} that, whenever $\la
=(\lambda_1,\ldots,\lambda_k) \in \cal{D}_n^-$, the labeling can be chosen in such a way that, for any
$\pi\in\mathcal{D}^-_n$, we have
\begin{equation}
\cla^+(t_{\pi})=\delta_{\pi\lambda}i^{\frac{n-k+1}{2}}\sqrt{\frac{\lambda_1
\ldots  \lambda_k}{2}}.
\label{eq:valdmoins}
\end{equation}
Writing $z_{\la}$ for the product $\la_1 \ldots  \la_k$, we therefore have, for any $\pi \in \cal{D}_n^-$,
\begin{equation}
\label{eq:valdmoinsSn}
\cla^+(t_{\pi})=\cla^-(zt_{\pi})=\delta_{\pi\lambda}i^{\frac{n-r+1}{2}}\sqrt{\frac{z_{\la}}{2}} \; \;  \mbox{and} \; \;  \cla^+(zt_{\pi})=\cla^-(t_{\pi})=-\delta_{\pi\lambda}i^{\frac{n-r+1}{2}}\sqrt{\frac{z_{\la}}{2}}.
\end{equation}
Finally, for any $\pi \in \cal{O}_n$, we have
$$\cla^+(t_{\pi})=\cla^-(t_{\pi}) \qquad \mbox{and} \qquad  \cla^+(zt_{\pi})=\cla^-(zt_{\pi}).$$

\subsection{Conjugacy classes and spin characters of $\tAlt_n$}
\label{section:classAnTilde}
We also write $\theta:\tAlt_n\rightarrow\Alt_n$ for the restriction of
$\theta$ to $\tAlt_n$. As above, the type of $g\in\tAlt_n$ is the
partition encoding the cycle structure of $\theta(g)$. As above, there
is a notion of split classes with respect to $\theta$. Such classes will
be called $\Alt_n$-split in the following. On the other hand, since
$\tAlt_n$ is a subgroup of $\tSym_n$ with index $2$, every
$\tSym_n$-class contained in $\tAlt_n$ is either a single
$\tAlt_n$-class or a union of two $\tAlt_n$-classes. In the second case,
the $\tAlt_n$-classes will be called $\tSym_n$-split classes.
By~\cite[p.\,176]{schur}, we have

\begin{proposition}
Assume $n\geq 2$. The $\Alt_n$-split classes of $\tAlt_n$ are the classes whose elements have type
$\pi\in \mathcal D_n^+\cup \mathcal O_n$.
\label{prop:splitAntilde}
\end{proposition}

\begin{remark}
Let $t\in\tAlt_n$ be with support contained in $X=\{k,k+1,\ldots,k+l\}$
for some $1\leq k<n$ and $1\leq l$ with $k+l\leq n$. 
Let $1\leq i\leq n-1$ be such that $\{i,i+1\}\cap X=\emptyset$. 
Then $t$ and $t_i$ commute. Indeed, since $\varepsilon(t)=1$, there
are integers $k\leq j_1,\ldots,j_{2r}\leq k+l-1$ such that
$t=t_{j_1}\cdots t_{j_{2r}}$. Furthermore, 
we have $|i-j_u|\geq 2$ for all $1\leq u\leq 2r$.
Hence, ${}^{t_i}t_u=zt_u$ and ${}^{t_i}t=z^{2r}t=t$,
as required.
\label{rk:commute}
\end{remark}

Assume $n\geq 2$. 
Let $\pi\in \mathcal D_n^+\cup \mathcal O_n$.  If $\pi\notin \mathcal
D_n^+\cap \mathcal O_n$, then $\pi$ labels two classes $D_{\pi}^+$
and $D_{\pi}^-$ of $\tAlt_n$. We assume that $t_{\pi}$ defined above lies in
$D_{\pi}^+$, so that representatives for $D_{\pi}^+$ and $D_{\pi}^-$ are
$\tau_{\pi}^+=t_{\pi}$ and $\tau_{\pi}^{-}=zt_{\pi}$.

Otherwise, if $\pi\in \mathcal D_n^+\cap \mathcal O_n$, then $\pi$
labels four $\Alt_n$-classes (that are $\Alt_n$-split and $\tSym_n$-split
simultaneously).
Write $\pi=(\pi_1,\ldots,\pi_k)$
with respect to Convention~\ref{conv}, and
$o_{\pi_1}=z^{\frac{\pi_1^2-1}{8}}t_{\pi_1}$.  According
to~\cite[footnote (*), p.  179]{schur}, $o_{\pi_1}$ has odd order.
Furthermore,
we assume that the support of $s_{\pi_j}$ is as in
Equation~(\ref{eq:suppSn}). So,
$t_1\in\tSym_{\pi_1}$. Since $\pi_1\in\mathcal
D^+_{\pi_1}\cap \mathcal O_{\pi_1}$, the elements $o_{\pi_1}$ and
${}^{t_1}o_{\pi_1}$ are not $\tAlt_{\pi_1}$-conjugate. 
Now, we get $\tau_{\pi}^{++}=t_{\pi_k}\cdots t_{\pi_2}o_{\pi_1}$ and 
\begin{equation}
\label{eq:choixrep}
\tau_{\pi}^{+-}={}^{t_1}\tau_{\pi}^{++},\quad
\tau_{\pi}^{-+}=z\tau_{\pi}^{++}\quad\textrm{and}\quad\tau_{\pi}^{--}=z\tau_{\pi}^{+-}.
\end{equation}
So, $\tau_{\pi}^{++}$, $\tau_{\pi}^{-+}$, $\tau_{\pi}^{+-}$ and
$\tau_{\pi}^{--}$ belong to $4$ distinct $\tAlt_n$-classes, labeled 
$D_{\pi}^{++}$, $D_{\pi}^{-+}$, $D_{\pi}^{+-}$ and $D_{\pi}^{--}$
respectively. 
Note that ${}^{t_1}t_{\pi_j}=t_{\pi_j}$ for $j>1$ by
Remark~\ref{rk:commute}. In particular, one has
$\tau_{\pi}^{+-}=t_{\pi_k}\cdots t_{\pi_2}{}^{t_1}o_{\pi_1}$.

We can now describe the irreducible complex characters of $\tAlt_n$.
These are given by using Clifford's Theory between $\tSym_n$ and the
subgroup $\tAlt_n$ of index 2.  All the irreducible components of the
restrictions to $\tAlt_n$ of non-spin characters of $\tSym_n$ have $z$
in their kernel, whence are non-spin characters of $\tAlt_n$. They are
exactly the irreducible characters of $\tAlt_n$ lifted from those of
$\Alt_n$.

We now turn to spin characters (which are the irreducible components of the restrictions to $\tAlt_n$ of the spin characters of $\tSym_n$).

First consider $\la \in \cal{D}_n^-$. Then $\la$ labels two associated
spin characters $\cla^+$ and $\cla^-$ of $\tSym_n$, which have the same
restriction to $\tAlt_n$. The restriction
\begin{equation}
\label{eq:caraltildenon}
\pla=\Res^{\tSym_n}_{\tAlt_n}(\cla^+)=\Res^{\tSym_n}_{\tAlt_n}(\cla^-)
\end{equation}
is an irreducible spin character of $\tAlt_n$, and its only non-zero
values are taken on elements of type belonging to $\cal{O}_n$.

Now consider $\la \in \cal{D}_n^+$. Then $\la$ labels a single spin
character $\cla$ of $\tSym_n$, and
$\Res^{\tSym_n}_{\tAlt_n}(\cla)=\pla^++\pla^-$, where $\pla^+$ and
$\pla^-$ are two conjugate irreducible spin characters of $\tAlt_n$.
Throughout, the characters $\pla^+$ and $\pla^-$ are also called 
\emph{associate characters}.
These only differ on elements of type $\la$. Following Schur \cite[p.
236]{schur}, we have, writing $\Delta_{\la}$ for the {\emph{difference
character of $\cla$}} (which is not well-defined, but just up to
a sign), that 
\begin{equation}
\label{eq:defdiff}
\Delta_{\la}(t)= \left\{ \begin{array}{cl} \pm
i^{\frac{n-\ell(\la)}{2}} \sqrt{z_{\la}} & \mbox{if $t$ has type $\la$,}
\\ 0 & \mbox{otherwise,} \end{array} \right.
\end{equation}
where $z_{\la}$ is defined after Equation~(\ref{eq:valdmoins}).

We will now make the notation precise.
We distinguish two cases. 
Suppose first that $\la \in \cal{D}_n^+ \setminus \cal{O}_n$. 
Then 
$\pla^+$ and $\pla^-$ are completely defined by setting
$\Delta_{\la}=\pla^+-\pla^-$ and $\Delta_{\la}(\tau_{\la}^{+})=
i^{\frac{n-\ell(\la)}{2}} \sqrt{z_{\la}}$, where
$\tau_{\lambda}^{+}$ is the representative
of $D_{\lambda}^{+}$ as above. Note that, using
Equation~(\ref{eq:defdiff}) and 
$\tau_{\lambda}^-=z\tau_{\lambda}^+$, we deduce that 
$\Delta_{\la}(\tau_{\la}^{-})=
-i^{\frac{n-\ell(\la)}{2}} \sqrt{z_{\la}}$. Since, for
$\epsilon\in\{-1,1\}$,
\begin{equation}
\pla^{\epsilon}=\frac{1}{2}\left(\Res^{\tSym_n}_{\tAlt_n}(\cla) +
\epsilon \Delta_{\la}\right),
\label{eq:defcharaltildeass}
\end{equation}
and $\cla(t_{\la})=0$ (because $C_{\la}$ is a non-split class of
$\tSym_n$), 
we obtain
\begin{equation}
\label{eq:valDnplus}
\pla^+(\tau_{\la}^{\pm})=\frac{1}{2}\Delta_{\la}(\tau_{\la}^{\pm}) \qquad
\mbox{and} \qquad
\pla^-(\tau_{\la}^{\pm})=\frac{1}{2}\Delta_{\la}(\tau_{\la}^{\mp})=-\frac{1}{2}\Delta_{\la}(\tau_{\la}^{\pm}).
\end{equation}
And, on any element $\sa$ of type $\pi \neq \la$ of $\tAlt_n$, we have
\begin{equation}
\label{eq:valailleur}
\pla^+(\sa)=\pla^-(\sa)=\frac{1}{2} \cla(\sa).
\end{equation}

Now suppose that $\la \in \cal{D}_n^+ \cap \cal{O}_n$. Again,
we completely define $\pla^+$ and $\pla^-$ by setting
$\Delta_{\la}=\pla^+-\pla^-$ and
$\Delta_{\la}(\tau_{\la}^{++})=i^{\frac{n-\ell(\la)}{2}}
\sqrt{z_{\la}}$.
Note that this does define $\Delta_{\la}$, and thus $\pla^+$ and 
$\pla^-$ 
by Equation~(\ref{eq:defcharaltildeass}),
since we have $\Delta_{\la}(\tau_{\lambda}^{-+})=- {\Dla(\tau_{\la}^{++})}$
because $\tau_{\lambda}^{-+}=z\tau_{\lambda}^{++}$, and
$\Delta_{\la}(\tau_{\lambda}^{+-})=-{\Dla(\tau_{\la}^{++})}$ by Clifford
theory, because the
elements $\tau_{\lambda}^{+-}$ and $\tau_{\lambda}^{++}$ are
$\widetilde{\sym}_n$-conjugate, and $\pla^+$ and
$\pla^-$ are $\widetilde{\sym}_n$-conjugate.
Finally, on any element $\sa$ of type $\pi \neq \la$ of $\tAlt_n$,
Equation~(\ref{eq:valailleur}) holds.
%
%
%
%

\subsection{Combinatorics of bar partitions}

We just saw that the spin characters of $\tSym_n$ and $\tAlt_n$ are labeled by the set $\cal{D}_n$ of {\emph{bar partitions of $n$}}. We now present some of the combinatorial notions and properties we will need to study the characters and blocks of these groups. For all of these, and unless otherwise specified, we refer to \cite{olsson}. Note that, in this subsection and the next, where we only describe the standard combinatorics associated to bar partition and spin blocks, the parts of partitions and bar partitions are again ordered in decreasing order.

Let $\lambda=(\lambda_1,\ldots,\lambda_r)\in \mathcal{D}_n$ with
$\lambda_1>\cdots>\lambda_r>0$. For $1\leq i\leq r$, consider the set
$$J_{i,\lambda}=\left(\{1,\ldots,\lambda_i\}\cup
\{\lambda_i+\lambda_j\,|\,j>i\}\right)\backslash \{\lambda_i-\lambda_j\,|\,
j>i\}.$$
The multiset $\cal{B}(\la)= \bigcup_{i=1}^r J_{i,\la}$ is the multiset of {\emph{bar lengths}} of $\la$, which will play a role analogous to that played by hook lengths for partitions.

The {\emph{shifted tableau}} $S(\lambda)$ of $\lambda$ is obtained from
the usual Young diagram of $\lambda$ by shifting the $i$th row $i-1$
positions to the right, and writing in the nodes of the $i$th row the
elements of $J_{i,\lambda}$ in decreasing order. The $j$th node in the
$i$th row of $S(\lambda)$ will be called the $(i,j)$-node of
$S(\lambda)$.  Write $a_{i,j}$ for the integer lying in the $(i,j)$-node
of $S(\lambda)$. As in the case of hooks, we can associate to this node
a {\emph{bar}} $b_{i,j}$ of $\la$ whose length is $a_{i,j}$. The
construction goes as follows. If $i+j\geq r+2$, then $b_{i,j}$ is the
usual $(i,j)$-hook of $S(\lambda)$. If
$i+j=r+1$, then $b_{i,j}$ is the $i$th row of $S(\lambda)$. Finally, if
$i+j\leq r$, then $b_{i,j}$ is the union of the $i$th row together with the
$j$th row of $S(\lambda)$. In all cases, one checks that $a_{i,j}$ is exactly the number of nodes in $b_{i,j}$, and is therefore called the {\emph{bar length}} of $b_{i,j}$. We can also define the {\emph{leg length}} $L(b_{i,j})$ of the bar $b_{i,j}$ by setting
$$L(b_{i,j})= \left\{ \begin{array}{ll} | \{ k \, | \, \la_i > \la_k > \la_i - a_{i,j} \} | & \mbox{if $i+j \geq r+1$}, \\ \la_{i+j} +  | \{ k \, | \, \la_i > \la_k > \la_{i +j} \} | & \mbox{if $i+j \leq r$}. \end{array} \right.$$

As for hooks, it is always possible to {\emph{remove}} any bar $b$ from $S(\la)$. If $b$ has bar length $a$, then this operation produces a new bar partition, written $\la \setminus b$, of size $n-a$. 

\smallskip
Let $q$ be an odd integer. We call $q$-bar (respectively $(q)$-bar) any bar $b$ of $\lambda$ whose
length is $q$ (respectively divisible by $q$). Note that, for any positive integer $k$, the removal of a $kq$-bar can be achieved by succesively removing $k$ bars of length $q$ (this fails when $q$ is even). By removing all the $(q)$-bars in $\la$, one obtains the {\emph{$\q$-core}} $\la_{(\q)}$ of $\la$. One can show that $\la_{(\q)}$ is independent on the order in which one removes $q$-bars from $\la$. In particular, the total number $w_{\q}(\la)$ of $q$-bars to remove from $\la$ to get to $\la_{(\q)}$ is uniquely defined by $\la$ and $q$, and called the {\emph{$\q$-weight of $\la$}}. Note that $w_{\q}(\la)$ is also equal to the number of $(q)$-bars in $\la$.

It is also possible to define the {\emph{$\q$-quotient}} $\la^{(\q)}$ of $\la$, which contains the information about all the $(q)$-bars in $\la$ (see~\cite[p.
28]{olsson}).
We have $\lambda^{(\overline{q})}=(\lambda^0, \, \lambda^1, \, \ldots, \, \lambda^e)$, where $e=(q-1)/2$, $\lambda^0$ is a bar partition, and
the $\lambda^i$'s are partitions for $1\leq i\leq e$. For any integer $k$, we define a $k$-bar $b'$ of
$\lambda^{(\overline{q})}=(\lambda^{0}, \, \lambda^{1}, \, \ldots, \, \lambda^{e})$ to be either a $k$-bar of $\lambda^0$, or a
$k$-hook of $\lambda^i$ for some $1\leq i\leq e$. The removal of $b'$ from $\la^{(\q)}$ is then defined accordingly, and the resulting $\q$-quotient is denoted by $\la^{(\q)} \setminus b'$. The {\emph{leg length}} $L(b')$ is also defined in a natural manner. We then have the
following fundamental result (see~\cite[Proposition 4.2, Theorem 4.3]{olsson})

\begin{theorem}Let $q$ be an odd integer.
Then a bar partition $\lambda$ determines and is uniquely determined by its
$\overline{q}$-core $\lambda_{(\overline{q})}$ 
and its $\overline{q}$-quotient $\lambda^{(\overline{q})}$. Moreover,
there is a canonical bijection $g$ between the set of $(q)$-bars of
$\lambda$ and the set of bars of $\lambda^{(\overline{q})}$, such that, for each integer $k$, the image of a $kq$-bar of $\la$ is a $k$-bar of $\la^{(\q)}$. Furthermore, for the removal of corresponding bars, we have
$$(\lambda\backslash
b)^{(\overline{q})}=\lambda^{(\overline{q})}\backslash g(b).$$
\label{olsson}
\end{theorem}

Note that the above theorem provides a (canonical) bijection between the set of parts of $\la$ with length divisible by $q$ and the set of parts of $\la^0$ (see \cite[Corollary (4.6)]{olsson}).

Theorem \ref{olsson} also implies that the $\q$-weight $w_{\q}(\la)$ of $\la$ satisfies $w_{\overline{q}}(\lambda)=\sum_{i=0}^e |\lambda^i|$ (we say that $\la^{(\q)}$ is a $\q$-quotient of size $| \la^{(\q)} |= w_{\q}(\la)$), and that $|\lambda|=|\lambda_{(\overline q)}|+qw_{\overline{q}}(\lambda)$ (see~\cite[Corollary 4.4]{olsson}). In addition, if we write, in analogy with bar partitions, $\sigma(\lambda^{(\overline{q})})=(-1)^{| \la^{(\q)} | -  \ell(\la^0)}=(-1)^{w_{\overline{q}}(\lambda)-\ell(\lambda^0)}$, then we obtain that
\begin{equation}
\sigma(\lambda)=\sigma(\lambda_{(\overline{q})})
\sigma(\lambda^{(\overline{q})}).
\label{eq:sign}
\end{equation}

\medskip
When we introduce analogues of the Murnaghan-Nakayama rule for spin characters later on, we will also need to use the {\emph{relative sign}} for bar partitions introduced by Morris and Olsson in \cite{morrisolsson}. Given an odd integer $q$, one can associate in a canonical way to each bar partition $\la$ a sign $\da_{\q}(\la)$. If $\mu$ is a bar partition obtained from $\la$ by removing a sequence of $q$-bars, then we define the {\emph{relative sign}} $\da_{\q}(\la, \, \mu)$ by
\begin{equation}
\label{eq:signTilde}
\da_{\q}(\la, \, \mu)= \da_{\q}(\la) \da_{\q}( \mu).
\end{equation}
It is then possible to prove the following results (see \cite[Proposition (2.5), Corollary (2.6), Corollary (3.8)]{morrisolsson}):

\begin{theorem}\label{relativesign}
Let $\la$ and $\mu$ be bar partitions, and $q$ be an odd integer.

\begin{enumerate}[(i)]
\item
If $\mu$ is obtained from $\la$ by removing a sequence of $q$-bars with leg lengths $L_1, \, \ldots , \, L_s$, then
$$\da_{\q}(\la, \, \mu)=(-1)^{\sum_{i=1}^s L_i}.$$
In particular, the parity of $\sum_{i=1}^s L_i$ does not depend on the choice of $q$-bars being removed in going from $\la$ to $\mu$.

\item
If $\ga$ is a $\q$-core, then $\da_{\q}(\ga)=1$, so that
$$\da_{\q}(\la)=\da_{\q}(\la, \, \la_{(\q)}).$$

\item
If $b$ is a $(q)$-bar in $\la$ and $\mu = \la \setminus b$, then
$$(-1)^{L(b)}=(-1)^{L(g(b))} \da_{\q}(\la, \, \mu),$$
where $g$ is the bijection introduced in Theorem \ref{olsson}.
\end{enumerate}
\end{theorem}

\subsection{Spin blocks of $\tSym_n$ and $\tAlt_n$; Bijections}

We now describe the blocks of irreducible characters of $\tSym_n$ and
$\tAlt_n$, as well as bijections between them. Throughout this section,
we assume that $q$ is an odd prime (even though all the combinatorial
arguments hold for any odd $q$).

If $B$ is any $q$-block of $\tSym_n$, then $B$ contains either no or
only spin characters. In the former case, $B$ coincides with a $q$-block
of $\Sym_n$; in the latter, we say that $B$ is a {\emph{spin block}}.
The distribution of spin characters into spin blocks was first
conjectured by Morris. It was first proved by J. F. Humphreys in
\cite{Humphreys-Blocks}, then differently by M. Cabanes, who also
determined the structure of the defect groups of spin blocks (see
\cite{Cabanes}).

Similarly, any $q$-block $B^*$ of $\tAlt_n$ contains either no spin
character, and coincides with a $q$-block of $\Alt_n$, or only spin
characters, and is then called a spin block.

The spin blocks of $\tSym_n$ and $\tAlt_n$ are described by the
following:

\begin{theorem}\label{blockstSn}
Let $\chi$ and $\psi$ be two spin characters of $\tSym_n$, or two spin
characters of $\tAlt_n$, labeled by bar partitions $\la$ and $\mu$
respectively, and let $q$ be an odd prime. Then $\chi$ is of $q$-defect
0 (and thus alone in its $q$-block) if and only if $\la$ is a $\q$-core.
If $\la$ is not a $\q$-core, then $\chi$ and $\psi$ belong to the same
$q$-block if and only if $\la_{(\q)}=\mu_{(\q)}$.
\end{theorem}
One can therefore define the {\emph{$\q$-core}} of a spin block $B$ and
its {\emph{$\q$-weight}} $w_{\q}(B)$, as well as its {\emph{sign}}
$\sa(B)=\sa( \la_{(\q)})$ (for any bar partition $\la$ labeling some
character $\chi \in B$).

\smallskip
One sees that the spin $q$-blocks of positive weight (or defect) of
$\tSym_n$ can be paired with those of $\tAlt_n$. The spin characters in
any such $q$-block of $\tAlt_n$ are exactly the irreducible components
of the spin characters of a $q$-block of $\tSym_n$.

\smallskip
We can now define bijections between different blocks of possibly
different groups. Let $w>0$ be any integer, and let $Q_w$ be the set of
$\q$-quotients of size $w$. For any $\overline{q}$-core $\ga$, we let
$E_{\gamma,w}$ be the set of bar partitions $\lambda$ of length
$|\gamma|+qw$ with $w_{\overline{q}}(\lambda)=w$ and
$\lambda_{(\overline{q})}=\gamma$, and we denote by $B_{\ga, w}$ and
$B_{\ga,w}^*$ the spin $\q$-blocks of $\tSym_{|\gamma|+qw}$ and
$\tAlt_{|\gamma|+qw}$ respectively labeled by $\ga$. Note that the
characters in $B_{\ga,w}$ and those in $B_{\ga,w}^*$ are labeled by the
partitions in $E_{\ga,w}$. Note also that $$\Psi_{\ga} \colon \left\{
\begin{array}{ccc} E_{\ga,w} & \longrightarrow & Q_w \\ \la &
\longmapsto & \la^{(\q)} \end{array} \right.$$ is a bijection. It
provides us with the following:

\begin{lemma}\label{bijections}
Let $q$ be an odd prime, $w>0$ be any integer, and $\ga$ and $\ga'$ be
any $\q$-cores. Define the bijection $$\Psi=\Psi_{\ga'}^{-1} \circ
\Psi_{\ga} \colon E_{\ga,w} \longrightarrow E_{\ga',w}.$$
\begin{enumerate}[(i)] 
\item If $\sa(\ga)=\sa(\ga')$, then $\Psi$ induces
bijections $\widetilde{\Psi}$ between $B_{\ga,w}$ and $B_{\ga',w}$, and
$\widetilde{\Psi}^*$ between $B_{\ga,w}^*$ and $B_{\ga',w}^*$.
\item If $\sa(\ga)=-\sa(\ga')$, then $\Psi$ induces bijections
$\widetilde{\Psi}$ between $B_{\ga,w}$ and $B_{\ga',w}^*$, and
$\widetilde{\Psi}^*$ between $B_{\ga,w}^*$ and $B_{\ga',w}$.
\end{enumerate}
\end{lemma}

\begin{proof}
This follows easily from the definition of $\Psi$ and formula
(\ref{eq:sign}), which gives that, for any $\la \in E_{\ga,w}$ and $\mu
\in E_{\ga',w}$, we have $\sa(\la)=\sa(\ga) \sa(\Psi_{\ga}(\la))$ and
$\sa(\mu)=\sa(\ga') \sa(\Psi_{\ga'}(\mu))$. We therefore have (taking
$\mu$ to be $\Psi(\la)$) $$\sa(\Psi(\la))=\sa(\ga')
\sa(\Psi_{\ga'}(\Psi(\la))=\sa(\ga')
\sa(\Psi_{\ga'}(\Psi_{\ga'}^{-1}(\Psi_{\ga}(\la))))=\sa(\ga')
\sa(\Psi_{\ga}(\la)),$$ so that
$\sa(\Psi_{\ga}(\la))=\sa(\ga')\sa(\Psi(\la))$ and, finally,
$$\sa(\la)=\sa(\ga) \sa(\ga') \sa(\Psi(\la)) \qquad \mbox{for any $\la
\in E_{\ga,w}$}.$$ This means that, if $\sa(\ga)=\sa(\ga')$, then any $\la \in E_{\ga,w}$
labels the same numbers of spin characters in $\tSym_{|\ga|+qw}$
and $\tAlt_{|\ga|+qw}$ as $\Psi(\la)$ does in $\tSym_{|\ga'|+qw}$ and
$\tAlt_{|\ga'|+qw}$ respectively. If, on the other hand,
$\sa(\ga)=-\sa(\ga')$, then $\la$ labels label the same numbers of spin
characters in $\tSym_{|\ga|+qw}$ and $\tAlt_{|\ga|+qw}$ as $\Psi(\la)$
does in $\tAlt_{|\ga'|+qw}$ and $\tSym_{|\ga'|+qw}$ respectively.

We obtain the following description for the bijections
$\widetilde{\Psi}$ and $\widetilde{\Psi}^*$:

\begin{itemize}
\item[(i)]
If $\sa(\ga)=\sa(\ga')$, then, for any $\la, \, \mu \in E_{\ga,w}$ with
$\sa(\la)=1$ and $\sa(\mu)=-1$, $$\widetilde{\Psi} : \left\{
\begin{array}{ccc} \cla & \longmapsto & \xi_{\Psi(\la)} \\ \{ \cmu^+, \,
\cmu^- \} & \longmapsto & \{ \xi_{\Psi(\mu)}^+, \, \xi_{\Psi(\mu)}^- \}
\end{array} \right. \;  \mbox{and}  \;  \widetilde{\Psi}^* : \left\{
\begin{array}{ccc} \{ \pla^+, \, \pla^-\} & \longmapsto & \{
\zeta_{\Psi(\la)}^+, \, \zeta_{\Psi(\la)}^- \}  \\ \pmu & \longmapsto &
\zeta_{\Psi(\mu)} \end{array} \right. $$
\item[(ii)] If $\sa(\ga)=-\sa(\ga')$, then, for any $\la, \, \mu \in
E_{\ga,w}$ with $\sa(\la)=1$ and $\sa(\mu)=-1$, $$\widetilde{\Psi} :
\left\{ \begin{array}{ccc} \cla & \longmapsto & \zeta_{\Psi(\la)} \\ \{
\cmu^+, \, \cmu^- \} & \longmapsto & \{ \zeta_{\Psi(\mu)}^+, \,
\zeta_{\Psi(\mu)}^- \} \end{array} \right. \;  \mbox{and}  \;
\widetilde{\Psi}^* : \left\{ \begin{array}{ccc} \{ \pla^+, \, \pla^-\} &
\longmapsto & \{ \xi_{\Psi(\la)}^+, \, \xi_{\Psi(\la)}^- \}  \\ \pmu &
\longmapsto &  \xi_{\Psi(\mu)} \end{array} \right. $$
\end{itemize}
\end{proof}

\subsection{Morris' Recursion Formula and MN-structures for
$\widetilde{\sym}_n$ and $\tAlt_n$}
First, for the convenience of the reader, we prove the following useful
lemma.

\begin{lemma}
Let $\rho\in\sym_n$ be an element of odd order.
Then the set
$\theta^{-1}(\rho)$ has an element of odd order.
\label{lem:ordre}
\end{lemma}

\begin{proof}
Let $g\in\theta^{-1}(\rho)$, so that $\theta^{-1}(\rho)=\{g, \, zg\}$, 
and let $d$ be the order of $\rho$.
Since $\theta(g^d)=\theta(g)^d=\rho^d=1$, we obtain $g^{d}\in\{1,z\}$. If
$g^d=1$, then the order of $g$ is odd. Otherwise, $g^d=z$, and
$(zg)^d=z^dg^d=z^2=1$ because $d$ is odd. Thus $zg$ has odd order, as
required.
\end{proof}

In the following, if $\rho\in\sym_n$ has odd order, then we denote 
by $o_{\rho}$ the element of $\theta^{-1}(\rho)$ with odd order.

A. O. Morris was the first to prove a recursion formula, similar to the Murnaghan-Nakayama Rule, for computing the values of spin characters of $\tSym_n$ (see \cite{Morris1} and \cite{Morris}). This formula was then made more general by M. Cabanes in \cite{Cabanes}. We have the following:

\begin{theorem}{\cite[Theorem 20]{Cabanes}}\label{MNCabanes}
Let $n \geq 2$ be an integer, $q \in \{2, \, \ldots , \, n \}$ be an odd
integer, and $\rho$ a $q$-cycle of $\sym_n$ with support $\{n-q+1, \,
\ldots , \, n \}$. Let $\la$ be a
bar partition of $n$. If $\sa(\la)=1$, then we write
$\cla=\cla^+=\cla^-$. Then $x=o_{\rho}$ satisfies $C_{\tSym_n}(x)=
\tSym_{n-q} \times \langle x \rangle$ and, for all $g \in \tSym_{n-q}$,
\begin{equation}
\label{eq:claplusmarc}
\cla^{+}(xg)= \sum_{\mu \in M_q(\la) \atop \sa(\mu)=1}
a(\cla^{+},\cmu) \cmu (g)  +  \sum_{\mu \in M_q(\la) \atop \sa(\mu)=-1}
(a(\cla^{+},\cmu^+) \cmu^+ (g)+a(\cla^{+},\cmu^-) \cmu^- (g))
,
\end{equation}
where $M_q(\la)$ is the set of bar partitions of $n-q$ which
can be obtained from $\la$ by removing a $q$-bar, and $a(\cla^{+},
\cmu^+) , \, a(\cla^{+},\cmu^-) \in \C^*$ are the following:

\begin{itemize}
\item[--] if $\sa (\mu) = 1$, then $a(\cla^+, \cmu)= (-1)^{\frac{q^2-1}{8}} \alpha^{\la}_{\mu}$,

\item[--] if $\sa (\mu) = -1$ and $\mu \neq \la \setminus \{q\}$, then $a(\cla^+, \cmu^+)=a(\cla^+, \cmu^-)= \frac{1}{2} (-1)^{\frac{q^2-1}{8}} \alpha^{\la}_{\mu}$,

\item [--] if $\sa (\mu) = -1$ and $\mu = \la \setminus \{q\}$, then
$a(\cla^+, \cmu^+)= \frac{1}{2} (-1)^{\frac{q^2-1}{8}} (
\alpha^{\la}_{\mu} + i^{\frac{q-1}{2}} \sqrt{q})$ and $a(\cla^+,
\cmu^-)=  \frac{1}{2} (-1)^{\frac{q^2-1}{8}} ( \alpha^{\la}_{\mu} -
i^{\frac{q-1}{2}} \sqrt{q})$.
\end{itemize}

The $\alpha^{\la}_{\mu}$'s are the coefficient found by Morris in his
recursion formula (see~\cite[Theorem 2]{Morris}). They are given by
$$\alpha^{\la}_{\mu}= (-1)^{L(b)} 2^{m(b)},$$ where $L(b)$ is the leg
length of the $q$-bar $b$ removed from $\la$ to get $\mu$, and 
$$m(b)= \left\{ \begin{array}{rl} 1 & \mbox{if} \; \sa(\la) =1 \; \mbox{and} \;  \sa(\mu )= -1 , \\ 0 & \mbox{otherwise}. \end{array} \right.$$

\end{theorem}
\begin{remark}\label{sign-coef}
Note that, with the notation of Theorem \ref{MNCabanes}, since
$\cla^-(xg)=\e \cla^+(xg)$, and since, with a slight abuse of notation,
$\e (xg)= \e(g)$ (as $q$ is odd and $x=o_{\rho}$), we can also write
\begin{equation}
\label{eq:clamoinsmarc}
\cla^{-}(xg)= \sum_{\mu \in M_q(\la) \atop \sa(\mu)=1} a(\cla^{-},\cmu) \cmu (g)  +  \sum_{\mu \in M_q(\la) \atop \sa(\mu)=-1} a(\cla^{-},\cmu^+) \cmu^+ (g)+a(\cla^{-},\cmu^-) \cmu^- (g),
\end{equation}
where, whenever $\sa(\mu)=1$, $a(\cla^{-},\cmu)=a(\cla^{+},\cmu)$, and, whenever $\sa(\mu)=-1$, $a(\cla^{-},\cmu^+)=a(\cla^{+},\cmu^-)$ and $a(\cla^{-},\cmu^-)=a(\cla^{+},\cmu^+)$.
\end{remark}
\medskip

\begin{lemma}
Let $q$ be an odd number, and $a\in\tAlt_n$ be of cycle type $(q)$ such
that $\theta(a)$ has support $I=\{n-q+1,\ldots,n\}$. Let $g$ and $g'$ be
in $\tAlt_{n-q}$ such that $ag$ and $ag'$ are $\tAlt_n$-conjugate and in
an $\tAlt_n$-class labeled by $\lambda\in D_n^+$.  Then $g$ and $g'$
have type $\mu=\lambda\backslash\{q\}\in D_{n-q}^+$ and are
$\tAlt_{n-q}$-conjugate.
\label{lem:conjugaisondistinct}
\end{lemma}

\begin{proof} With the assumption, it is clear that $g$ and $g'$ have
cycle type $\mu$. 
Let $t\in\tAlt_n$ be such that ${}^t(ag)=ag'$. Then
${}^{\theta(t)}(\theta(a)\theta(g))=\theta(a)\theta(g')$. Since
${}^{\theta(t)}\theta(a)$ is a $q$-cycle of $\theta(a)\theta(g')$ and
$\theta(a)$ is the unique $q$-cycle of $\theta(a)\theta(g')$ (because
all the cycles of this element are distinct), it follows that
${}^{\theta(t)}\theta(a)=\theta(a)$. Thus, $I$ is invariant by 
$\theta(t)$. Set $v:=\theta(t)|_I\in\Sym_I$. Then
$v\in\Cen_{\Sym_n}(\theta(a))$,
and since the cycle type of $\theta(a)$ is odd and distinct in $\Alt_I$, one has
$\Cen_{\Sym_I}(\theta(a))=\Cen_{\Alt_I}(\theta(a))$, and in particular, $v\in \Alt_I$. Now, 
let $\widetilde{v}\in\tAlt_I$ be such
that $\theta(\widetilde{v})=v$. By Remark~\ref{rk:commute},
we have ${}^{\widetilde{v}}g=g$. Write
$w=t\,\widetilde{v}^{-1}\in\tAlt_{n-q}$. Then ${}^wa=a$ or
${}^wa=za$, and since $a$ and $za$ have distinct order, we deduce that
${}^wa=a$. It follows that $a\,{}^wg={}^w(ag)={}^t(ag)=ag'$, and
${}^wg=g'$ with $w\in\tAlt_{n-q}$, as required.
\end{proof}

\begin{remark}
\label{rk:choix}
Let $q$ be an odd multiple of $p$ and $\lambda=(\lambda_1,\ldots,\lambda_k=q)\in D_n^+$ be as in
Convention~\ref{conv} (in particular, $q$ is the smallest odd part of
$\lambda$  divisible by $p$).
Assume that $n-q\geq 2$ and write $I=\{n-q+1,\ldots,n\}$. 
Let $\rho$ be the $q$-cycle of $\sym_I$
with respect to the choice of representatives given before
Equation~(\ref{eq:suppSn}).  
Denote by $t_{\rho}$ the
element of $\widetilde{\sym}_I$ such that $\theta(t_{\rho})=\rho$ with
respect to the choice of Schur~\cite[\S11]{schur}, and write
$\mu=(\lambda_1,\ldots,\lambda_{k-1})\in \mathcal D_{n-q}^+$.
With the choice of Equation~(\ref{eq:choixrep}), one has
$\tau_{\lambda}^{\pm}=t_{\rho}\tau_{\mu}^{\pm}$ if $\lambda\notin
\mathcal O_n$ (note that $\mu\notin \mathcal O_{n-q}$) and
$\tau_{\lambda}^{\pm\pm}=t_{\rho}\tau_{\mu}^{\pm\pm}$ if $\lambda\in
\mathcal O_n$ (in this case, $\mu$ is automatically in $\mathcal
O_{n-q}$).
\end{remark}

We now obtain an analogue of Theorem \ref{MNCabanes} for $\tAlt_n$. Let
$q$ be an odd number such that $n-q\geq 2$. Let $\rho$ and $t_{\rho}$ be
as in Remark~\ref{rk:choix}. According to~\cite[footnote (*), p.
179]{schur}, recall that $o_{\rho}=z^{\frac{q^2-1}{8}}t_{\rho}$. 

\begin{theorem}\label{MNcoefAntilde} Let $q$ be an odd integer such that
$n-q\geq 2$. 
We keep the notation of Remark~\ref{rk:choix} and we set
$x:=o_{\rho}=z^{\frac{q^2-1}{8}}t_{\rho}$. In particular,
$x \in \tAlt_n$, and $C_{\tAlt_n}(x) = \tAlt_{n-q} \times \langle x \rangle$.
Assume that the choice for the labels of the
classes (and thus for the labels of the characters by~\S\ref{section:classAnTilde}) are as in Remark~\ref{rk:choix}. 
Take any $\lambda\in\mathcal D_n$, $\epsilon\in\{-1,1\}$ and
$g \in \tAlt_{n-q}$. When $\lambda\in
\mathcal D_n^-$, we set $\pla^+=\pla^-=\pla$. Finally, let $\pi$ be the cycle type of $xg$.
Then, if $\lambda \neq \pi$, or if $\lambda=\pi$ and $q$ is the last part of $\lambda$,
we have $$\pla^{\epsilon}(xg)=  \displaystyle  \sum_{\mu \in M_q(\la) \atop
\sa(\mu)=-1} a(\pla^{\epsilon},\pmu) \pmu (g) \, + \sum_{\mu \in M_q(\la) \atop
\sa(\mu)=1} \left( a(\pla^{\epsilon},\pmu^+) \pmu^+ (g) +
a(\pla^{\epsilon},\pmu^-)
\pmu^-(g) \right),$$
where the coefficients are the following:
\begin{itemize}
\item[--] if $\lambda\in\mathcal D_n^-$, then
$a(\pla,\pmu^{\eta})=(-1)^{\frac{q^2-1}{8}}\alpha_{\mu}^{\lambda}$ for
all $\mu\in M_q(\lambda)$ and $\eta\in\{-1,1\}$, where
$\alpha_{\mu}^{\lambda}$ is as in Theorem~\ref{MNCabanes}.
\item[--] if $\lambda\in\mathcal D_n^+$, then
$a(\pla^{\epsilon},\pmu)=
\frac{1}{2}(-1)^{\frac{q^2-1}{8}}\alpha_{\mu}^{\lambda}$
whenever $\sigma(\mu)=-1$, and
$a(\pla^{\epsilon},\pmu^{\eta})=
\frac{1}{2}(-1)^{\frac{q^2-1}{8}}(\alpha_{\mu}^{\lambda}+\epsilon\eta
i^{\frac{q-1}{2}}\sqrt{q})$ for
$\eta\in\{-1,1\}$ whenever $\sigma(\mu)=1$.
\end{itemize}
%
%
\end{theorem}

\begin{proof}
First assume that $\lambda\in\mathcal D_n^-$. Then, by
Equation~(\ref{eq:caraltildenon}), and Clifford theory applied to
Equation~(\ref{eq:claplusmarc}), we obtain the following.
Whenever $\sa(\mu)=-1$, 
we have
$$a(\pla,\pmu)=a(\cla^{+},\cmu^+) +a(\cla^{+},\cmu^-) =  \dis
\frac{1}{2} (-1)^{\frac{q^2-1}{8}} \alpha_{\mu}^{\la} + \dis \frac{1}{2}
(-1)^{\frac{q^2-1}{8}} \alpha_{\mu}^{\la} =(-1)^{\frac{q^2-1}{8}}
\alpha_{\mu}^{\la},$$and, whenever $\sa(\mu)=1$, $$a(\pla,\pmu^+) =
a(\pla,\pmu^-)=a(\cla^{+},\cmu)=(-1)^{\frac{q^2-1}{8}}
\alpha_{\mu}^{\la},$$
as required.

We now consider the case where $\la \in {\cal D}_n^+$. By
Equation~(\ref{eq:defcharaltildeass}) and Clifford theory applied to
Equation~(\ref{eq:claplusmarc}), we obtain
\begin{equation}
\label{eq:int411}
\begin{split}
\pla^+(xg) & = \dis \frac{1}{2} \left( \cla(xg) + \Dla (xg) \right) \\
& =\sum_{\mu \in M_q(\la) \atop \sa(\mu)=1} \frac{a(\cla,\cmu)}{2}
(\pmu^+ (g) + \pmu^-(g)) + \frac{\Dla(xg)}{2} \\&\qquad + {\dis
\sum_{\mu \in M_q(\la) \atop \sa(\mu)=-1} }\frac{ a(\cla,\cmu^+)
+a(\cla,\cmu^-)}{2} \pmu (g).
\end{split}
\end{equation}


We need to deal with the term $\frac{\Dla(xg)}{2}$. Recall that this is
0 unless $xg$ has cycle type $\pi=\la$.  We start by noticing that, if $xg$
does not have cycle type $\la$, then $g$ does not have cycle type $\mu$
for any $\mu \in M_q(\la)$ with $\sa(\mu)=1$. Indeed, if $\mu$ is
obtained from $\la$ by removing a bar $b$ of odd length $q$, then,
depending on the type of $b$, we have $\ell(\mu)=\ell(\la)$,
$\ell(\mu)=\ell(\la)-2$ or $\ell(\mu)=\ell(\la)-1$. In the first two
cases, we obtain $\sa (\mu) = (-1)^{n-q-\ell(\la)}=-\sa(\la)$. The last
case can only happen if $b$ is a part of $\la$, in which case $\mu = \la
\setminus \{q\}$ and $\sa(\mu)=\sa(\la)$. This has several consequences.
The first is that $\{ \mu \in M_q(\la) \, | \, \sa(\mu)=1 \}$ is either
empty, or contains only the partition $\la \setminus \{q\}$. This in
turn implies that $xg$ has cycle type $\la$ if and only if $\{ \mu \in
M_q(\la) \, | \, \sa(\mu)=1 \} = \{ \la \setminus \{q\} \}$ and $g$ has
cycle type $\la \setminus \{q\}$. Finally, $\{ \mu \in M_q(\la) \, | \,
\sa(\mu)=1 \}$ is empty if and only if $\la$ does not have a part of
length $q$, and, if this is the case, then $\frac{\Dla(xg)}{2}=0$ for
all $g \in \tAlt_{n-q}$.

\smallskip
We therefore suppose that $\la$ does have a part of length $q$, so that $$\{ \mu \in M_q(\la) \, | \, \sa(\mu)=1 \} = \{ \la \setminus \{q\} \}.$$
We will show that, if $\mu = \la \setminus \{ q \}$, then, for all $g \in \tAlt_{n-q}$, we have
\begin{equation}
\label{Deltas}
\Dla (xg) = (-1)^{\frac{q^2-1}{8}} i^{\frac{q-1}{2}} \sqrt{q} \Dmu (g).
\end{equation}
If $g$ does not have cycle type $\mu = \la \setminus \{ q \}$, then
$\Dmu(g)=0$, and $xg$ does not have cycle type $\la$, so that
$\Dla(xg)=0$ and Equation (\ref{Deltas}) holds. 

Now, assume that $g$ has cycle type $\mu$. Then $xg$ has cycle type
$\pi=\la$, so we assume furthermore that $q$ is the last part of $\la$. If $\lambda\in \mathcal O_n$ (resp. $\lambda\notin \mathcal O_n$), 
then there are
signs $\delta$ and $\eta$ such that 
$xg$ is $\tAlt_n$-conjugate to $\tau_{\lambda}^{\delta\eta}$ (resp. to
$\tau^{\delta}_{\lambda}$). It follows that
$t_{\rho}z^{\frac{q^2-1}{8}}g$ and
$\tau_{\lambda}^{\delta\eta}=t_{\rho}\tau_{\mu}^{\delta\eta}$ (resp. 
$\tau_{\lambda}^{\delta}=t_{\rho}\tau_{\mu}^{\delta}$) are $\tAlt_n$-conjugate.
By Lemma~\ref{lem:conjugaisondistinct}, $z^{\frac{q^2-1}{8}}g$ is
$\tAlt_{n-q}$-conjugate to $\tau_{\mu}^{\delta\mu}$ (resp. to
$\tau_{\mu}^{\delta}$), that is, $g\in
D_{\mu}^{(\frac{q^2-1}{8}\delta)\eta}$ (resp. $g\in
D_{\mu}^{\frac{q^2-1}{8}\delta}$).
%

Now, using the values and properties we gave for the difference characters, we obtain that, for $xg \in D_{\la}^+$ (or similarly for $xg \in D_{\la}^{++}$), we have
\begin{equation}
\label{eq:int4112}
\begin{array}{rcl} \Dla(xg)= \Dla(D_{\la}^{+}) & = & i^{\frac{n-
\ell(\la)}{2}} \sqrt{z_{\la}} \\ & = & i^{\frac{n-q+q- \ell(\mu)-1}{2}}
\sqrt{q} \sqrt{z_{\mu}} \\ & = &  i^{\frac{q-1}{2}}  \sqrt{q}
i^{\frac{n-q- \ell(\mu)}{2}}\sqrt{z_{\mu}}  \\ & = &  i^{\frac{q-1}{2}}
\sqrt{q}  \Dmu(D_{\mu}^{+})\\& = &(-1)^{\frac{q^2-1}{8}} i^{\frac{q-1}{2}}
\sqrt{q} \Dmu(g).\\
\end{array}
\end{equation}
Using the property $\Dla(D_{\la}^-)=-\Dla(D_{\la}^+)$, and its analogues
for the classes $D_{\la}^{\pm \pm}$, we easily deduce that 
Equation~(\ref{Deltas}) does hold for all $g \in \tAlt_{n-q}$.

%
Now, from Equations~(\ref{eq:int411}) and~(\ref{Deltas}), and
Theorem \ref{MNCabanes},
%
%
we deduce that, for $\mu\in M_q(\lambda)$, 
if $\sa (\mu)=-1$, then $a(\pla^+,\pmu)=\frac{1}{2}
(-1)^{\frac{q^2-1}{8}} \alpha_{\mu}^{\la}$, and, if $\sa (\mu)=1$, then
$a(\pla^+, \pmu^+)= \frac{1}{2} (-1)^{\frac{q^2-1}{8}} (
\alpha_{\mu}^{\la} +  i^{\frac{q-1}{2}}  \sqrt{q} )$ and $a(\pla^+,
\pmu^-)= \frac{1}{2} (-1)^{\frac{q^2-1}{8}} ( \alpha_{\mu}^{\la} -
i^{\frac{q-1}{2}}  \sqrt{q} )$.


\smallskip
Our analysis of the term $\Dla(xg)$ also yields a similar formula for
$\pla^-(xg)$, and using Equation~(\ref{eq:clamoinsmarc}), we deduce the
values of $a(\cla^-,\cmu^{\eta})$ for all $\mu\in M_q(\lambda)$ and
$\eta\in\{-1,1\}$.
\end{proof}

\begin{remark} Let $n$ and $q$ be as above.
Assume $n=q$ or $n=q+1$. Then $\tAlt_{n-q}=\Z_2$ and the only spin
character is the non-trivial character $\varepsilon$ of $\Z_2$, labeled by
$\mu=\emptyset$ or $\mu=(1)$ whenever $n=q$ or $n=q+1$. Set
$\pi=(q)$ if $n=q$ and $\pi=(q,1)$ if $n=q+1$. Then there are $4$
classes of $\tAlt_n$ labeled by $\pi$  with representatives
$\tau_{\pi}^{\pm\pm}$. Write $o_{\pi}^{\pm}=\tau_{\pi}^{+\pm}$.
Let $k\in\{0,1\}$ and $\lambda\in\mathcal D_n$ be
with $\overline{q}$-core $\mu$. If $\lambda\in\mathcal D_n^-$, then
$\zeta_{\lambda}(o_{\pi}^{\pm}z^k)=(-1)^{\frac{q^2-1}{8}}\alpha_{\mu}^{\lambda}
\varepsilon(z^k)$.
If $\lambda\in\mathcal D_n^+$ and $\lambda\neq \pi$, then 
$\zeta_{\lambda}(o_{\pi}^{\pm}z^k)=\frac{1}{2}
(-1)^{\frac{q^2-1}{8}}\alpha_{\mu}^{\lambda}\varepsilon(z^k)$.
Finally, for $\delta,\,\eta\in\{-1,1\}$, one has
\begin{equation}
\label{eq:valdiff}
\zeta_{\pi}^{\eta}(o_{\pi}^{\delta}z^k)=
\zeta_{\pi}^{\eta}(o_{\pi}^{\delta})\varepsilon(z^k)=
\frac{(-1)^{\frac{q^2-1}{8}}}{2}\left(\alpha_{\mu}^{\pi}+\eta\delta
i^{\frac{q-1}{2}}\sqrt{q}\right)\varepsilon(z^k).
\end{equation}
\label{rk:degeneree}
\end{remark}


We define $C_{\Sym_n}$ to be the set of elements of
$\Sym_n$ none of whose cycles has length an odd multiple of $p$. We then
let
\begin{equation}
\label{eq:defCtilde}
C_{\tSym_n}= \theta^{-1}(C_{\Sym_n})\quad \mbox{and}  \quad C_{\tAlt_n} = C_{\tSym_n} \cap \tAlt_n.
\end{equation}
%
Finally, we let ${\cal C}_{\tSym_n}$ and ${\cal C}_{\tAlt_n}$ be the sets of (respectively $\tSym_n$- and $\tAlt_n$-) conjugacy classes in $C_{\tSym_n}$ and $C_{\tAlt_n}$ respectively.

\smallskip

From now on, if $G$ is a finite group and $\mathcal C$ a union
of conjugacy classes of $G$, then $\mathcal C$-blocks is meant in the sense
of KOR-blocks (see Proposition~\ref{prop:choixbase}).

We start by showing that the spin $p$-blocks of $\tSym_n$ (respectively
$\tAlt_n$) are also ${\cal C}_{\tSym_n}$-blocks (respectively ${\cal
C}_{\tAlt_n}$-blocks). Recall that the $p$-blocks of $\tSym_n$ are just
the $\tSym_{n,p'}$-blocks, where $\tSym_{n,p'}$ is the set of
$p$-regular elements of $\tSym_n$. Similarly, the $p$-blocks of
$\tAlt_n$ are its $\tAlt_{n,p'}$-blocks. Note that, by definition, we
have $\tSym_{n,p'} \subset C_{\tSym_n}$ and $\tAlt_{n,p'} \subset
C_{\tAlt_n}$.

\begin{lemma}\label{same-blocks}
The $p$-blocks and $\mathcal C_{\tSym_n}$-blocks of spin characters of
$\tSym_n$ coincide, and the $p$-blocks and $\mathcal C_{\tAlt_n}$-blocks of spin characters of
$\tAlt_n$ coincide.
\end{lemma}

\begin{proof}
Let $\chi$ and $\xi$ be a non-spin and a spin character of $\tSym_n$,
respectively. Since $\chi$ is constant on the split classes, we deduce 
that $\cyc{\chi,\xi}_{C_{\tSym_n}}=0$. Thus, spin and non-spin
characters are never in the same $\mathcal C_{\tSym_n}$-block.

Now, take any two spin characters $\xi$ and $\xi'$ of $\tSym_n$, such
that $\xi' \not \in  \{\xi, \e \xi \}$. Then the only elements of
$C_{\tSym_n} \setminus \tSym_{n,p'} $, if any, on which $\xi$ doesn't
vanish belong to split conjugacy classes labeled by the partition
labeling $\xi$ (this is because any split conjugacy class of
$C_{\tSym_n} $ labeled by a partition of ${\cal O}_n$, and thus without
even cycles, must also belong to $\tSym_{n,p'} $). And since $\xi' \not
\in  \{\xi, \e \xi \}$, we see that $\xi'$ vanishes on these elements.
In this case, we therefore have 
\begin{equation}
\label{eq:scal}
\langle \xi , \xi'
\rangle_{C_{\tSym_n}} = \langle \xi , \xi' \rangle_{\tSym_{n,p'}} =
\langle \e \xi , \xi' \rangle_{\tSym_{n,p'}}=\langle \varepsilon \xi , \xi'
\rangle_{C_{\tSym_n}}.
\end{equation}
Assume that $\xi'\notin \{\xi,\e \xi\}$ lies in the same $p$-block as
$\xi$. Then there are distinct spin irreducible characters
$\xi'=\xi_1,\ldots,\xi_s=\xi$ such that $\xi_i\neq \xi$ for $1\leq i\leq
s-1$ and $\langle
\xi_i,\xi_{i+1}\rangle_{\tSym_{n,p'}}\neq 0$. We can assume that $\xi_i\neq
\varepsilon\xi$ for all $1\leq i\leq s$. 
Indeed, let $2\leq i\leq s$ be such that
$\xi_i=\varepsilon\xi$. Since $\langle
\xi_{i-1},\xi_i\rangle_{\tSym_{n,p'}}= \langle
\xi_{i-1},\varepsilon\xi \rangle_{\tSym_{n,p'}}= \langle
\xi_{i-1},\xi\rangle_{\tSym_{n,p'}}$ by Equation~(\ref{eq:scal}), we
can take $s=i$. 
We also can assume that $\xi_{i+1}\neq\varepsilon \xi_{i}$ for all $1\leq
i\leq s-1$. Otherwise, if there is $1\leq i\leq s-1$ such that
$\xi_{i+1}=\varepsilon\xi_i$, then $i<s-1$, and since
$\xi_{i+2}\notin\{\xi_{i},\xi_{i+1}\}=\{\xi_i,\varepsilon\xi_i\}$, we
deduce from Equation~(\ref{eq:scal}) that $\langle
\xi_{i+1},\xi_{i+2}\rangle_{\tSym_{n,p'}}=\langle
\xi_{i},\xi_{i+2}\rangle_{\tSym_{n,p'}}$ and we can remove $\xi_{i+1}$
from the chain.

Hence, Equation~(\ref{eq:scal}) gives $\langle
\xi_i,\xi_{i+1}\rangle_{\tSym_{n,p'}}=\langle
\xi_i,\xi_{i+1}\rangle_{C_{\tSym_n}}$ for all $1\leq i\leq s-1$, and the
characters $\xi'$ and $\xi$ lie in the same $C_{\tSym_n}$-block.

By a similar argument, 
Equation~(\ref{eq:scal}) implies that if $\xi'\notin \{\xi,\e \xi\}$
lies in the same $C_{\tSym_n}$-block as $\xi$, then they are in the same
$p$-block.

Furthermore, if $\xi \neq \e \xi$, then either $\xi$ and $\e \xi$ 
belong to the same $p$-block, or
each is alone in its respective $p$-blocks.
In the first case, the $p$-block which contains $\xi$ and $\e \xi$ also 
contains some spin character $\xi'$ such that $\xi'=\e\xi'$ (this follows from
\cite[(2.1)]{Olsson-blocks}). 
In particular, in the $p$-block of $\xi$, there is some irreducible
$\xi''\notin\{\xi,\e \xi\}$ such that $\langle
\xi'',\xi\rangle_{\tSym_{n,p'}}\neq 0$. Thus, by Equation~(\ref{eq:scal}), we
obtain $\langle
\xi'',\xi\rangle_{C_{\tSym_n}}=\langle
\xi'',\xi\rangle_{\tSym_{n,p'}}=\langle
\xi'',\varepsilon\xi\rangle_{C_{\tSym_n}}$,
and $\xi$ and $\e \xi$ belong to the same ${\cal C}_{\tSym_n}$-block.

In the second case,
$\xi$ and $\e \xi$ have $p$-defect zero. In particular, $\xi$ and $\e
\xi$ are labeled by a $\overline{p}$-core $\lambda\in D_n^-$, and
$\lambda$ is $p$-regular. Hence they both vanish
identically on $p$-singular elements, so also on $\cal C_{\tSym_n}
\setminus \tSym_{n,p'}$, and they are each alone in their
respective ${\cal C}_{\tSym_n}$-blocks as well. 
The result follows.

For the case of $\tAlt_n$, the argument is similar.
\end{proof}

We can now define on $\tSym_n$ and $\tAlt_n$ an MN-structure with
respect to the set of spin $p$-blocks of $\tSym_n$ and $\tAlt_n$, 
and the sets $C_{\tSym_n}$ and $C_{\tAlt_n}$ defined in
Equation~(\ref{eq:defCtilde}), respectively. 
For this, we define $S_{\Sym_n}$
to be the set of elements $\sigma \in \Sym_n$ all of whose cycles have
length 1 or an odd multiple of $p$. By Lemma~\ref{lem:ordre}, we denote
by $o_{\sigma}$ the element of $\tSym_n$ of odd order such that
$\theta(o_{\sigma})=\sigma$, and
we let 
\begin{equation}
\label{eq:defStilde}
S_{\tSym_n}= \{ o_{\sigma} \, | \, \sigma \in
S_{\Sym_n} \} \quad\textrm{and}\quad 
S_{\tAlt_n} = S_{\tSym_n} \cap \tAlt_n.
\end{equation}
Note that, since $p$ is odd, and since we
only consider odd multiples of $p$, we have $S_{\tAlt_n} = S_{\tSym_n}
\cap \tAlt_n= S_{\tSym_n} $. 
%

\begin{proposition}\label{MNTilde}
Let $n>0$ be any integer, and $p$ be an odd prime. Let $Sp(\tSym_n)$ and
$Sp(\tAlt_n)$ be the sets of spin $p$-blocks of $\tSym_n$ and $\tAlt_n$
respectively. Then $\tSym_n$ has an MN-structure (as defined in
Definition \ref{defMN}) with respect to ${\cal C}_{\tSym_n}$ and
$Sp(\tSym_n)$, and $\tAlt_n$ has an MN-structure with respect to ${\cal
C}_{\tAlt_n}$ and $Sp(\tAlt_n)$.

\end{proposition}

\begin{proof}
First note that, by Lemma \ref{same-blocks}, $Sp(\tSym_n)$ and
$Sp(\tAlt_n)$ are indeed unions of ${\cal C}_{\tSym_n}$-blocks and
${\cal C}_{\tAlt_n}$-blocks respectively.

To stick with the notation of Definition \ref{defMN}, we take $G \in \{
\tSym_n, \, \tAlt_n \}$, $B=Sp(G)$, ${\cal C}={\cal C}_G$ and $S=S_G$
(as defined above). Properties 1 and 2 of Definition \ref{defMN} are
immediate consequences of the definition of $S$ and ${\cal C}$. For $x_S
\in S$ and $x_C \in C$, we have $(x_S, \, x_C) \in A$ if and only if the
non-trivial cycles of $\theta(x_S)$ and $\theta(x_C)$ are disjoint (in
particular, $x_S$ and $x_C$ commute).

Now take any $x_S \in S$. If $x_S=1$, then $G_1=G$, $B_1=B$ and
$r^1=\operatorname{id}$ clearly satisfy Properties 3 and 4. If, on the
other hand, $x_S \neq 1$, then, by definition of $S$, we have
$x_S=o_{\sigma}$ for some $\sigma\in S_{\sym_n}$. Write
 $\sigma=\sigma_1  \cdots  \sigma_k$, where, for each $1 \leq i \leq
k$, $\sigma_i$ is a $q_i$-cycle for some odd multiple $q_i$ of $p$, and
the $\sigma_i$'s are pairwise disjoint. In particular, $\sigma_i\in
S_{\sym_n}$ and, since $\sigma_i\in \Alt_n$,~\cite[III, p.\,172]{schur} 
gives $o_{\sigma}=o_{\sigma_1}
  \cdots  o_{\sigma_k}$, and $C_G(x_S)$ has as a subgroup the
group $H=G_{x_S} \times \langle o_{\sigma_1} \rangle \times \cdots \times
\langle o_{\sigma_k} \rangle$, where $G_{x_S} \cong \tSym_{n -
\sum_{i=1}^k q_i}$ if $G = \tSym_n$, and $G_{x_S} \cong \tAlt_{n -
\sum_{i=1}^k q_i}$ if $G = \tAlt_n $ (and with the convention that
$\tSym_0=\tAlt_0=\langle z \rangle$). 

Property 3 now follows from the definition of $A$ we gave above.
Clearly, if $x_C \in G_{x_S} \cap C$, then the non-trivial cycles of
$\theta(x_C)$ and $\theta(x_S)$ are disjoint, so that $(x_S, \, x_C) \in
A$. Conversely, if $(x_S, \, x_C) \in A$, then one must have $x_C \in
C_G(x_S)$. If $x_C \in C_G(x_S) \setminus H$, then $\theta(x_C)$ must
permute (non-trivially) the $(p)$-cycles of $\theta(x_S)$; in
particular, the non-trivial cycles of $\theta(x_C)$ and $\theta(x_S)$
cannot be disjoint, so that $(x_S, \, x_C) \not \in A$. Hence, if $(x_S,
\, x_C) \in A$, then necessarily $x_C \in H$. Now, in order for $x_C$ to
be disjoint from $x_S$, we see that one must have $x_C \in G_{x_S}$.
This proves that $G_{x_S} \cap C= \{ x_C \in C \, | \, (x_S, \, x_C) \in
A \}$.

Finally, we obtain Property 4 by iterating Theorem
\ref{MNCabanes} and Theorem~\ref{MNcoefAntilde}. 
By considering (and removing) the ``cycles'' $o_{\sigma_i}$ ($1
\leq i \leq k$) one at a time, and in increasing order of size, one sees that we can define
$r^{x_S}(\chi)$ for any spin character $\chi \in B$. By construction,
$r^{x_S}(\chi)$ does satisfy $r^{x_S}(\chi)(x_C)=\chi(x_S \cdot x_C)$ for all
$(x_S, \, x_C) \in A$. Taking $B_{x_S}$ to be the set of spin characters
of $G_{x_S}$, and extending $r^{x_S}$ by linearity to $\C \Irr(B)$, we
obtain the result.
\end{proof}

\subsection{Brou\'e perfect isometries}
Throughout this section, we denote by $p$ an odd prime number.

Let $n$, $q$ and $\la$ be as in Theorem \ref{MNCabanes}. Suppose
furthermore that $q$ is an odd multiple of $p$, and that
$w_{\p}(\la)>0$. Next consider any spin $p$-block $B'$, of $\tSym_m$
say, of the same weight and sign as the $p$-block $B$ of $\cla^+$, and
the bijection $\Psi$ described in Lemma \ref{bijections}. In particular,
$\Psi$ preserves the parity of bar partitions. Now, since $q$ is a
multiple of $p$, the removal of a $q$-bar can be obtained by removing a
sequence of $p$-bars, and one sees from Theorem \ref{olsson} that
$M_q(\Psi(\la)) = \Psi(M_q(\la))$. This is a slight abuse of notation,
as $\Psi$ should only act on partitions of the same weight as $\la$,
while the elements of $M_q(\la)$ have a smaller weight. But we see that
$\Psi$ is compatible with the bijections $g_{\la}$ and $g_{\Psi(\la)}$
given by Theorem \ref{olsson}, since everything goes through the
(common) $\p$-quotient of $\la$ and $\Psi(\la)$. Also, thanks to
Equation~(\ref{eq:sign}) one has $\sa(\Psi(\mu))=\sa(\mu)$ for any 
$\mu \in M_q(\la)$. We then have the following:

\begin{proposition}\label{MNcoef}
Let the notation be as above. For any $\mu \in M_q(\la)$, and for any
$\epsilon, \, \eta\in \{ 1, \, -1 \}$, we have
$$\da_{\p}(\la) \da_{\p}(\mu) a( \cla^{\epsilon \da_{\p}(\la)},
\cmu^{\eta \da_{\p}(\mu)}) = \da_{\p}(\Psi(\la)) \da_{\p}(\Psi(\mu)) a(
\xi_{\Psi(\la)}^{\epsilon\da_{\p}(\Psi(\la))}, \xi_{\Psi(\mu)}^{\eta \da_{\p}(\Psi(\mu))}).$$

\end{proposition}

\begin{proof}
Let $\mu \in M_q(\la)$ be obtained by removing the $q$-bar $b$ from
$\la$. Then, by definition of $\Psi$, we see, using Theorem
\ref{olsson}, that $\Psi(\mu) \in M_q(\Psi(\la))$ is obtained by
removing the $q$-bar $\Psi(b)$ from $\Psi(\la)$.

We start by comparing $\alpha_{\mu}^{\la}$ and
$\alpha_{\Psi(\mu)}^{\Psi(\la)}$.  By definition, we have

\begin{equation}
\label{eq:alphatilde}
\alpha_{\mu}^{\lambda}= (-1)^{L(b)} 2^{m(b)} \qquad \mbox{and} \qquad
\alpha_{\Psi(\mu)}^{\Psi(\la)}=(-1)^{L(\Psi(b))} 2^{m(\Psi(b))}, 
\end{equation}
where
\begin{equation}
\label{eq:mb}
m(b)= \left\{ \begin{array}{rl} 1 & \mbox{if} \; \sa(\lambda) =1 \;
\mbox{and} \;  \sa(\mu )= -1 , \\ 0 & \mbox{otherwise}. \end{array}
\right.
\end{equation} and 
\begin{equation}
\label{eq:mbpsi}
m(\Psi(b))= \left\{ \begin{array}{rl} 1 & \mbox{if} \;
\sa(\Psi(\lambda)) =1 \; \mbox{and} \;  \sa(\Psi(\mu) )= -1 , \\ 0 &
\mbox{otherwise}. \end{array} \right.  
\end{equation}

%
And, since $\Psi$ preserves the parity of partitions, we see that $m(b)=m(\Psi(b))$.
Now $L(b)$ is related to $L(g_{\la}(b))$, where $g_{\la}$ is the
bijection described in Theorem~\ref{olsson}. Similarly, $L(\Psi(b))$ is
related to $L(g_{\Psi(\la)}(\Psi(b)))$, but, as we remarked above,
$g_{\Psi(\la)}(\Psi(b))=g_{\la}(b)$. We have, by Theorem
\ref{relativesign}(iii), applied to the $(p)$-bar $b$,
\begin{equation}
\label{eq:comparejambe}
(-1)^{L(b)}=(-1)^{L(g_{\la}(b))} \da_{\p}(\la, \, \mu)= (-1)^{L(g_{\la}(b))} \da_{\p}(\la) \da_{\p}( \mu)
\end{equation}
and similarly
\begin{equation}
\label{eq:comparejambepsi}
(-1)^{L(\Psi(b))}=
(-1)^{L(g_{\la}(b))} \da_{\p}(\Psi(\la)) \da_{\p}( \Psi( \mu)),
\end{equation}
whence
\begin{equation}
\label{alphas}
\da_{\p}(\la) \da_{\p}( \mu) (-1)^{L(b)} = \da_{\p}(\Psi(\la))
\da_{\p}( \Psi( \mu)) (-1)^{L(\Psi(b))}.
\end{equation}

%

If $\sa(\mu)=1$, then $\cmu^+=\cmu^-=\cmu$ and $a(\cla^+,
\cmu)=a(\cla^-, \cmu)= (-1)^{\frac{q^2-1}{8}} \alpha^{\la}_{\mu}$. We
also have $\sa(\Psi(\mu))=1$, so $a(\xi_{\Psi(\la)}^+,
\xi_{\psi(\mu)})=a(\xi_{\Psi(\la)}^-,
\xi_{\psi(\mu)})=(-1)^{\frac{q^2-1}{8}} \alpha_{\Psi(\mu)}^{\Psi(\la)}$.
Thus Equation~(\ref{alphas}) immediately gives the result.

Suppose now that $\sa(\mu)=-1$. Then, by Remark \ref{sign-coef}, we have
$a(\cla^{-},\cmu^+)=a(\cla^{+},\cmu^-)$ and
$a(\cla^{-},\cmu^-)=a(\cla^{+},\cmu^+)$. We need to distinguish between
the cases $\mu = \la \setminus \{q\}$ and $\mu \neq \la \setminus
\{q\}$. If $\mu = \la \setminus \{q\}$, this means $b$ is a part of
length $q$ in $\la$. Then, in Theorem \ref{olsson}, $g(b)$ must be a
part of length $q/p$ in the first (bar) partition of $\la^{(\q)}$ (see
\cite[Theorem (4.3)]{olsson}), and $\Psi(b)$ is then a part of length
$q$ in $\Psi(\la)$. We thus have $\mu = \la \setminus \{q\}$ if and only
if $\Psi(\mu) = \Psi(\la) \setminus \{q\}$.

Suppose first that that $\mu \neq \la \setminus \{q\}$, so that
$\Psi(\mu) \neq \Psi( \la) \setminus \{q\}$. Then, by Theorem
\ref{MNCabanes} and Remark \ref{sign-coef}, $$a(\cla^+,
\cmu^+)=a(\cla^+, \cmu^-)=a(\cla^{-},\cmu^+)=a(\cla^{-},\cmu^-)=
\frac{1}{2} (-1)^{\frac{q^2-1}{8}} \alpha^{\la}_{\mu}$$ and
$$a(\xi_{\Psi(\la)}^+, \xi_{\psi(\mu)}^+)=a(\xi_{\Psi(\la)}^+,
\xi_{\psi(\mu)}^-)=a(\xi_{\Psi(\la)}^-,
\xi_{\psi(\mu)}^+)=a(\xi_{\Psi(\la)}^-, \xi_{\psi(\mu)}^-)=\frac{1}{2}
(-1)^{\frac{q^2-1}{8}} \alpha_{\Psi(\mu)}^{\Psi(\la)},$$ so that
Equation~(\ref{alphas}) gives the result.

Suppose, finally, that ($\sa(\mu)=\sa(\Psi(\mu))=-1$ and) $\mu = \la \setminus \{q\}$, so that $\Psi(\mu) = \Psi(\la) \setminus \{q\}$. This is the only case which is not straightforward. 
By Theorem \ref{MNCabanes}, we have
$$a(\cla^+, \cmu^+)= \frac{1}{2} (-1)^{\frac{q^2-1}{8}} (
\alpha^{\la}_{\mu} + i^{\frac{q-1}{2}} \sqrt{q}) \quad \mbox{and} \quad
a(\cla^+,  \cmu^-)=  \frac{1}{2} (-1)^{\frac{q^2-1}{8}} (
\alpha^{\la}_{\mu} - i^{\frac{q-1}{2}} \sqrt{q})$$
(and similar expressions for $a(\xi_{\Psi(\la)}^+, \xi_{\psi(\mu)}^+)$
and $a(\xi_{\Psi(\la)}^+, \xi_{\psi(\mu)}^-)$). Since, by Remark
\ref{sign-coef}, $a(\cla^{-},\cmu^+)=a(\cla^{+},\cmu^-)$ and
$a(\cla^{-},\cmu^-)=a(\cla^{+},\cmu^+)$, we deduce that, for any
$\epsilon \in \{1, \, -1\}$,
$$a(\cla^{\epsilon \da_{\p}(\la)}, \cmu^{\da_{\p}(\mu)})= \frac{1}{2}
(-1)^{\frac{q^2-1}{8}} ( \alpha^{\la}_{\mu}  + \epsilon \da_{\p}(\la)
\da_{\p}(\mu) i^{\frac{q-1}{2}} \sqrt{q})$$
(and a similar expression for $a(\xi_{\Psi(\la)}^{\epsilon
\da_{\p}(\Psi(\la))}, \xi_{\psi(\mu)}^{\da_{\p}(\Psi(\mu))})$).
Multiplying by $ \da_{\p}(\la) \da_{\p}(\mu) $, we obtain, using
Equation (\ref{alphas}),
$$\begin{array}{rl}
 \da_{\p}(\la) \da_{\p}(\mu)  a(\cla^{\epsilon \da_{\p}(\la)},
\cmu^{\da_{\p}(\mu)}) & =  \frac{1}{2} (-1)^{\frac{q^2-1}{8}}\left(
\da_{\p}(\la) \da_{\p}(\mu) \alpha^{\la}_{\mu}  + \epsilon
i^{\frac{q-1}{2}} \sqrt{q}\right) \\ & =  \frac{1}{2}
(-1)^{\frac{q^2-1}{8}}
\left( \da_{\p}(\Psi(\la)) \da_{\p}(\Psi(\mu)) \alpha^{\Psi(\la)}_{\Psi(\mu)}
+ \epsilon  i^{\frac{q-1}{2}} \sqrt{q}\right)   \end{array}$$
 $$ =   \da_{\p}(\Psi(\la)) \da_{\p}(\Psi(\mu))  \frac{1}{2}
(-1)^{\frac{q^2-1}{8}} \left(\alpha^{\Psi(\la)}_{\Psi(\mu)}  + \epsilon
\da_{\p}(\Psi(\la)) \da_{\p}(\Psi(\mu))   i^{\frac{q-1}{2}}
\sqrt{q}\right),$$
 whence
$$  \da_{\p}(\la) \da_{\p}(\mu)  a\left(\cla^{\epsilon \da_{\p}(\la)},
\cmu^{\da_{\p}(\mu)}\right) =   \da_{\p}(\Psi(\la)) \da_{\p}(\Psi(\mu))
a\left(\xi_{\Psi(\la)}^{\epsilon \da_{\p}(\Psi(\la))},
\xi_{\psi(\mu)}^{\da_{\p}(\Psi(\mu))}\right).$$
Using Remark \ref{sign-coef}, this implies the last equality we have to prove:
$$ \begin{array}{rcl} \da_{\p}(\la) \da_{\p}(\mu)  a\left(\cla^{\epsilon
\da_{\p}(\la)}, \cmu^{-\da_{\p}(\mu)}\right) & = &  \da_{\p}(\la)
\da_{\p}(\mu)  a\left(\cla^{- \epsilon \da_{\p}(\la)},
\cmu^{\da_{\p}(\mu)}\right)
\\  & = & \da_{\p}(\Psi(\la)) \da_{\p}(\Psi(\mu))
a\left(\xi_{\Psi(\la)}^{-\epsilon \da_{\p}(\Psi(\la))},
\xi_{\psi(\mu)}^{\da_{\p}(\Psi(\mu))}\right) \\ & = & \da_{\p}(\Psi(\la))
\da_{\p}(\Psi(\mu)) a\left(\xi_{\Psi(\la)}^{\epsilon
\da_{\p}(\Psi(\la))}, \xi_{\psi(\mu)}^{-\da_{\p}(\Psi(\mu))}\right). 
\end{array}$$
\end{proof}

Now we consider $\p$-cores $\ga$ and $\ga'$, and a
positive integer $w$. Let $B$ and $B'$ be the spin blocks of $\tSym_n$
and $\tSym_m$ of weight $w$ and $\p$-core $\ga$ and $\ga'$ respectively,
and let $B^*$ and $B'^*$ be the corresponding spin blocks of $\tAlt_n$
and $\tAlt_m$. Suppose furthermore that $\sa(\ga)=- \sa(\ga')$, so that,
with the notation of Lemma \ref{bijections}(ii), $\Psi$ is a sign inversing
bijection, and $\widetilde{\Psi}$ gives a bijection between $B$ and
$B'^*$.

\begin{proposition}\label{MNcoefcroise}
Let the notation be as above, and assume $m-q\geq 2$. 
For any $\lambda\in\mathcal D_n$ 
with $\p$-core $\gamma$
and $\mu \in M_q(\la)$, and for any
$\eta, \, \epsilon \in \{ 1, \, -1 \}$, we have
$$\da_{\p}(\la) \da_{\p}(\mu) a( \cla^{\eta \da_{\p}(\la)},
\cmu^{\epsilon \da_{\p}(\mu)}) = \da_{\p}(\Psi(\la)) \da_{\p}(\Psi(\mu))
a( \zeta_{\Psi(\la)}^{\eta\da_{\p}(\Psi(\la))}, \zeta_{\Psi(\mu)}^{\epsilon \da_{\p}(\Psi(\mu))}).$$
\end{proposition}

\begin{proof}
First, assume that $\lambda\in\mathcal
D_n^+$. Then by Lemma~\ref{bijections}(ii), $\Psi(\lambda)\in\mathcal D_m^-$.
Furthermore, 
by Theorem \ref{MNCabanes}, for any $\mu\in M_q(\lambda)$, we have
$a(\xi_{\lambda},\xi_{\mu}) = (-1)^{\frac{q^2-1}{8}}
\alpha_{\mu}^{\lambda}$
whenever $\sa(\mu)=1$, and $a(\xi_{\lambda},\xi_{\mu}^+)=
a(\xi_{\lambda},\xi_{\mu}^-) =  \frac{1}{2} (-1)^{\frac{q^2-1}{8}}
\alpha_{\mu}^{\lambda}$ whenever $\sa(\mu)=-1$.

As previously, we see that $\Psi$ is compatible with the bijections
$g_{\lambda}$ and $g_{\Psi(\lambda)}$ given by Theorem  \cite[Theorem
(4.3)]{olsson}, whence it gives a sign inversing bijection between
$M_q(\Psi(\la))$ and $M_q(\la)$. If $\mu \in M_q(\lambda)$ is obtained by
removing the $q$-bar $b$ from $\lambda$, then $\Psi(\mu) \in
M_q(\Psi(\la))$ is
obtained by removing the $q$-bar $\Psi(b)$ from $\Psi(\la)$, so that we want
to compare $ \alpha_{\mu}^{\lambda}$ and  $
\alpha_{\Psi(\mu)}^{\Psi(\lambda)}$. 
%
For this, we use Equations~(\ref{eq:alphatilde}),~(\ref{eq:mb})
and~(\ref{eq:mbpsi}).
Since $\Psi(\la) \in {\cal D}_m^-$,
we see that $m(\Psi(b))$ is
always 0, so that $ \alpha_{\Psi(\mu)}^{\Psi(\la)}=(-1)^{L(\Psi(b))}$.
And, since $\lambda \in {\cal D}_n^+$, we see that $m(b)=1$ and
$\alpha_{\mu}^{\lambda}= (-1)^{L(b)} 2$ whenever $\sa(\mu)=-1$, while
$m(b)=0$ and $\alpha_{\mu}^{\lambda}= (-1)^{L(b)} $ whenever $\sa(\mu)=1$.

As in the proof of Proposition \ref{MNcoef}, $L(b)$ is related to
$L(g_{\lambda}(b))$, and $L(\Psi(b))$ to $L(g_{\Psi(\lambda)}(\Psi(b)))$, and
$g_{\Psi(\lambda)}(\Psi(b))=g_{\lambda}(b)$. 
Thus, using Equations~(\ref{eq:comparejambe})
and~(\ref{eq:comparejambepsi}), we see that Equation~(\ref{alphas}) holds.
If $\sa(\mu)=1$, then $\sa(\Psi(\mu))=-1$, and
we obtain $$\begin{array}{rcl} \da_{\p}(\lambda) \da_{\p}(
\mu)a(\xi_{\lambda},\xi_{\mu})  & = & (-1)^{\frac{q^2-1}{8}}
\da_{\p}(\lambda)
\da_{\p}( \mu) \alpha_{\mu}^{\lambda} \\  & = & (-1)^{\frac{q^2-1}{8}}
\da_{\p}(\lambda) \da_{\p}( \mu) (-1)^{L(b)} \\ & = &
(-1)^{\frac{q^2-1}{8}}\da_{\p}(\Psi(\lambda)) \da_{\p}( \Psi(
\mu))(-1)^{L(\Psi(b))} \\ & = & (-1)^{\frac{q^2-1}{8}}
\da_{\p}(\Psi(\lambda)) \da_{\p}( \Psi( \mu))
\alpha_{\Psi(\mu)}^{\Psi(\lambda)}
\\ & = & \da_{\p}(\Psi(\lambda)) \da_{\p}( \Psi( \mu)) a (
\zeta_{\Psi(\lambda)}, \zeta_{\Psi(\mu)}) .\end{array}$$ 
If, on the other hand, $\sa(\mu)=-1$, then $\sa(\Psi(\mu))=1$, and we obtain
$$\begin{array}{rcl} \da_{\p}(\lambda) \da_{\p}(
\mu)a(\xi_{\lambda},\xi_{\mu}^{\pm})  & = & (-1)^{\frac{q^2-1}{8}}
\da_{\p}(\lambda) \da_{\p}( \mu)  \frac{1}{2} \alpha_{\mu}^{\lambda} \\ 
& = &
(-1)^{\frac{q^2-1}{8}} \da_{\p}(\lambda) \da_{\p}( \mu) \frac{1}{2} 2
(-1)^{L(b)} \\ & = & (-1)^{\frac{q^2-1}{8}}\da_{\p}(\Psi(\lambda)) \da_{\p}(
\Psi( \mu))(-1)^{L(\Psi(b))} \\ & = & (-1)^{\frac{q^2-1}{8}}
\da_{\p}(\Psi(\lambda)) \da_{\p}( \Psi( \mu))
\alpha_{\Psi(\mu)}^{\Psi(\lambda)}
\\ & = & \da_{\p}(\Psi(\lambda)) \da_{\p}( \Psi( \mu) a (
\zeta_{\Psi(\lambda)}, \zeta_{\Psi(\mu)}^{\pm}) .\end{array}$$

Assume now that $\lambda\in\mathcal D_n^-$. Then $\Psi(\la)\in\mathcal D_m^+$.
Note that $\lambda$ has a part of length $q$, if and only if
$\Psi(\la)$ has one. If this is the case, then $\sa(\lambda \setminus \{q\})=-1$
and $\{ \mu \in M_q(\lambda) \, | \, \sa( \mu)=-1 \} = \{ \lambda \setminus
\{q\}\}$. Otherwise, $\{ \mu \in M_q(\lambda) \, | \, \sa( \mu)=-1 \}$ is
empty. By Theorem \ref{MNCabanes}, we have
$a(\cla^+, \cmu)= (-1)^{\frac{q^2-1}{8}} \alpha^{\lambda}_{\mu}$
whenever $\sa (\mu) = 1$, and 
$$a(\cla^+,
\cmu^+)= \frac{1}{2} (-1)^{\frac{q^2-1}{8}} ( \alpha^{\lambda}_{\mu} +
i^{\frac{q-1}{2}} \sqrt{q})\quad\textrm{and}\quad 
a(\cla^+,  \cmu^-)=  \frac{1}{2}
(-1)^{\frac{q^2-1}{8}} ( \alpha^{\lambda}_{\mu} - i^{\frac{q-1}{2}}
\sqrt{q}),$$
whenever $\sa (\mu) = -1$ (and $\mu = \lambda \setminus \{q\}$).

Furthermore, if $\sa(\mu)=1$, then $a(\cla^{-},\cmu)=a(\cla^{+},\cmu)$,
and, if $\sa(\mu)=-1$, then $a(\cla^{-},\cmu^+)=a(\cla^{+},\cmu^-)$ and
$a(\cla^{-},\cmu^-)=a(\cla^{+},\cmu^+)$.

As in the previous case, $\Psi$ gives a sign inversing bijection between
$M_q(\la)$ and $M_q(\Psi(\la))$. If $\mu \in M_q(\la)$ is obtained by removing
the $q$-bar $b$ from $\la$, then $\Psi(\mu) \in M_q(\Psi(\la))$ is
obtained by removing the $q$-bar $\Psi(b)$ from $\Psi(\la)$. 
Note that $\sigma(\lambda)=-1$ and $\sigma(\Psi(\la))=1$, and for
$\mu\in M_q(\lambda)$, we have $\sigma(\mu)=-\sigma(\Psi(\mu))$.
In particular, Equations~(\ref{eq:mb}) and~(\ref{eq:mbpsi})
give $m(b)=0$, and $m(\Psi(b))=1$ whenever $\sigma(\mu)=1$, and
$m(\Psi(b))=0$ otherwise.

If $\sigma(\mu)=1$, then $m(\Psi(b))=1$ and
$\sigma(\Psi(\mu))=-1$. Thus Theorem~\ref{MNCabanes},
Theorem~\ref{MNcoefAntilde}, and Equations~(\ref{eq:alphatilde})
and~(\ref{alphas}) give the result.

If $\sigma(\mu)=-1$, then $m(\Psi(b))=0$ and $\sigma(\Psi(\mu))=1$.
Thus, using Theorem~\ref{MNCabanes},
Theorem~\ref{MNcoefAntilde} and Equation~(\ref{alphas}),
we conclude with a computation similar to that at the end of
the proof of Proposition~\ref{MNcoef}.
%
\end{proof}

%
%
%
%

\begin{remark}
\label{rk:degeneree2}
Note that, with the notation of Remark~\ref{rk:degeneree} (in particular
we have $m\leq q+1$ and $\Psi(\mu)\in\{\emptyset,(1)\}$), if
$\Psi(\lambda)\neq \pi$ and if we set
$\zeta_{\Psi(\mu)}^+=\zeta_{\Psi(\mu)}^-=\varepsilon$, then Proposition~\ref{MNcoefcroise} still holds. 
\end{remark}

We now can state the main result of this section. Let $\gamma$ and
$\gamma'$ be two $p$-cores, and $w$ be a positive integer. Write
$E_{\gamma,w}$, $E_{\gamma',w}$ and $\Psi:E_{\gamma,w}\rightarrow
E_{\gamma',w}$ as in Lemma~\ref{bijections}, and set
$n=|\gamma|+pw$ and $m=|\gamma'|+pw$.
If $\sigma(\gamma)=\sigma(\gamma')$, then $B_{\gamma,w}$ and
$B_{\gamma',w}$
denote the $p$-blocks of $\bar p$-weight $w$ of $G=\tSym_{n}$ and
$G'=\tSym_m$ corresponding to $\gamma$ and $\gamma'$ respectively.
If $\sigma(\gamma)=-\sigma(\gamma')$, then $B_{\gamma,w}$ and
$B_{\gamma',w}$
denote the $p$-blocks of $\bar p$-weight $w$ of $G=\tAlt_{n}$ and
$G'=\tSym_m$ respectively.
We write
$\Irr(B_{\gamma})=\{X_{\lambda}^{\epsilon}\, |\, \lambda\in
E_{\gamma,w},\,\epsilon\in\{-1,1\}\}$ and 
$\Irr(B_{\gamma'})=\{Y_{\lambda}^{\epsilon}\, |\, \lambda\in
E_{\gamma',w},\,\epsilon\in\{-1,1\}\}$, with the convention that,
when $X_{\lambda}$ or $Y_{\lambda}$ are self-associate, we set
$X_{\lambda}^+=X_{\lambda}^-=X_{\lambda}$ and 
$Y_{\lambda}^+=Y_{\lambda}^-=Y_{\lambda}$.

\begin{theorem} Let $p$ be an odd prime.  
We keep the notation as above. Then the isometry
$I:\C\Irr(B_{\gamma,w})\rightarrow
\C\Irr(B_{\gamma',w})$ defined by
\begin{equation}
\label{eq:defisoTilde}
I\left(X_{\lambda}^{\epsilon\delta_{\bar p}(\lambda)}\right)
=\delta_{\bar
p}(\lambda)\delta_{\bar
p}(\Psi(\lambda))Y_{\Psi(\lambda)}^{\epsilon\delta_{\bar
p}(\Psi(\lambda))},
\end{equation}
where $\lambda\in E_{\gamma,w}$ and  $\epsilon\in\{-1,1\}$, is a Brou\'e perfect isometry.
\label{theo:mainTilde}
\end{theorem}

\begin{proof}
Consider  the map $\widehat{I}$ corresponding to $I$ as in
Equation~(\ref{eq:chapeau}). We will prove that $\widehat{I}$ satisfies
Properties (i) and (ii) of a Brou\'e isometry.

First, we use the MN-structures introduced in Proposition~\ref{MNTilde}
for $(C_G,B_{\gamma,w})$ and $(C_{G'},B_{\gamma',w)})$. 
Let $S_G$ and $S_{G'}$ be as in Equation~(\ref{eq:defStilde}).
Write $\Omega$ for the set of partitions $\pi$ of $i\leq n$ such that
$p$ divides each part of $\pi$. Note that $\pi\in\Omega$ parametrizes one or two $G$-classes of
elements of $S_G$ (always one class when $G=\tSym_n$, and two classes 
when $G=\tAlt_n$ and $\pi\in \mathcal O_n\cap \mathcal D_n$). In the
case where $\pi$ labels two classes, we denote the two parameters by
$\pi^{\pm}$. Let $\Lambda$ be the set of parameters obtained in this way.
 Then $\Lambda$ labels the set of $G$-classes of $S_G$.
We will now define a precise set of representatives for these classes.
Let $\pi=(\pi_1,\ldots,\pi_k)\in \Omega$. 
Note that $\pi$ and $(\pi_1),\ldots,(\pi_k)$ can all be viewed as labels
of conjugacy classes of
$S_{\sym_{n}}$ (by
completing the partitions with parts of length $1$). 

In particular, the element 
$s_{\pi}=s_{\pi_1}\cdots
s_{\pi_k}$ defined before Equation~(\ref{eq:suppSn}) is a representative
for the class of $S_{\sym_n}$ labeled by $\pi$, 
and $s_{\pi}\in\Sym_{|\pi|}$.
For all $1\leq i\leq k$, denote by $o_{\pi_i}$ the element of odd order 
such that $\theta(o_{\pi_i})=s_{\pi_i}$ (see Lemma~\ref{lem:ordre}) and
write $o_{\pi}=o_{\pi_1}\cdots o_{\pi_k}$. 
Furthermore, for $\pi\in \mathcal O_n\cap D_n$, we assume that the 
representatives of the two $\Alt_n$-classes labeled by $\pi$ are
$s_{\pi^{\pm}}= s_{\pi_1^{\pm}}\cdots s_{\pi_{k-1}}s_{\pi_k}$
 as in the proof of Theorem~\ref{theo:mainAn}. 
If $o_{\pi_1}^{\pm}$ denotes the elements of odd order satisfying
$\theta(o_{\pi_1}^{\pm})=s_{\pi_1^{\pm}}$ (see Lemma~\ref{lem:ordre}), 
then we set $o_{\pi^{\pm}}:=o_{\pi_1}^{\pm}\cdots o_{\pi_{k-1}}o_{\pi_k}$. 
Therefore, if $G=\tSym_n$, then the set of $o_{\pi}$ for
$\pi\in\Omega$ is a set of representatives of the $G$-classes 
of $S_{G}$. If $G=\tAlt_n$, then the elements $o_{\pi}$ (for $\pi\in
\Omega$ and
$\pi\notin \mathcal O_n\cap \mathcal D_n$) and $o_{\pi^{\pm}}$ (for $\pi\in
\Omega\cap\mathcal O_n\cap \mathcal D_n$) form a system of
representatives
of the $\tAlt_n$-classes of $S_{\tAlt_n}$. 
Moreover, for any $\widehat{\pi}\in\Lambda$ (with
$\widehat{\pi}\in\{\pi^+,\pi^-\}$ if $\pi\in \mathcal O_n\cap \mathcal
D_n$ and $G=\tAlt_n$, and $\widehat{\pi}=\pi$ otherwise), we write
$G_{\widehat{\pi}}=G_{o_{\widehat{\pi}}}$ and
$r^{\widehat{\pi}}=r^{o_{\widehat{\pi}}}$, where $G_{o_{\widehat{\pi}}}$
and
$r^{o_{\widehat{\pi}}}:\C\Irr(B_{\gamma,w})\rightarrow\C\Irr(B(G_{\pi}))$
are defined in Proposition~\ref{MNTilde} (here $B(G_{\pi})$ is one $p$-block or two $p$-blocks of $G_{\pi}$). 

Now, we define $\Omega_0=\{\pi\in \Omega\,|\,\sum \pi_i\leq
pw\}$ and $\Lambda_0$ the set of parameters $\widehat{\pi}\in\Lambda$
 such that
$\pi\in\Omega_0$. Then $\Omega_0=\Lambda_0$ whenever $G=\tSym_n$ or
$G=\tAlt_n$ with $n\notin\{ pw, pw+1\}$.

Similarly, we define $\Omega'$, $\Lambda'$, $\Omega_0'$ and $\Lambda_0'$
for $G'$ and $S_{G'}$. Since $G'=\tSym_m$, we have $\Lambda'=\Omega'$ and
$\Lambda_0'=\Omega_0'=\Omega_0$. We write $o'_{\pi}$ for the
representatives of the $G'$-classes of $S_{G'}$ (as described above for
$G$) and, for $\pi\in\Omega'$, we define $G'_{\pi}$ and
${r'}^{\pi}:\C\Irr(B_{\gamma',w})\rightarrow\C\Irr(B(G'_{\pi}))$ as
above.

Using Theorems~\ref{MNCabanes} and~\ref{MNcoefAntilde}, we show that for
any $\pi\in\Omega\backslash\Omega_0$ or
$\pi\in\Omega'\backslash\Omega_0'$, one has $r^{\widehat \pi}=0$ and
${r'}^{\widehat \pi}=0$.

Now we suppose that $\Lambda_0=\Omega_0$. Let $\pi\in\Omega_0$. If 
$|\pi|<pw$, then $B(G_{\pi})$ and $B(G'_{\pi})$ are just one $p$-block
of $G_{\pi}$ and $G'_{\pi}$, respectively. If $|\pi|=pw$, then $B(G_{\pi})$ and
$B(G'_{\pi})$ are one $p$-block with defect zero whenever $G=\tSym_n$ and
$\sigma(\gamma)=1$ or $G=\tAlt_n$ and $\sigma(\gamma)=-1$, or 
are the union of two $p$-blocks with defect zero otherwise. We
define 
(and denote by the same symbol to simplify the notation)
$I:\Irr(B(G_{\pi}))\rightarrow\Irr(B(G'_{\pi}))$ by 
Equation~(\ref{eq:defisoTilde}). 

We assume that Convention~\ref{conv} holds and that moreover, if
$\pi=(\pi_1,\ldots,\pi_k)\in\mathcal P_n$ has an odd part divisible by $p$, 
then there is some $r\leq r'\leq k $ such that $\pi_j$ is divisible by $p$
for all $r'\leq j$, and every odd $\pi_j$ with $j<r$ is prime to $p$. 
So, we can use
Theorem~\ref{MNCabanes} and~\ref{MNcoefAntilde} iteratively (see also
Remark~\ref{rk:choix}).
Therefore, using
Propositions~\ref{MNcoef} and~\ref{MNcoefcroise}, we show as in the
proof of Theorem~\ref{theo:mainAn} (see
Equations~(\ref{eq:coeffAitere}), (\ref{eq:acommute})
and~(\ref{eq:commuteAn})), 
that 
\begin{equation}
\label{eq:commuteTilde}
I\circ r^{\pi}={r'}^{\pi}\circ I.
\end{equation}
Thus, Theorem~\ref{th:iso} holds (see Remark~\ref{rk:point3}).

Suppose, on the other hand, that $\Lambda_0\neq\Omega_0$. 
Then $G=\tAlt_n$,
$n\in\{pw,pw+1\}$, and $G'=\tSym_m$. In particular, $\sigma(\gamma)=1$.  Let
$\pi=(\pi_1,\ldots,\pi_l)\in \Omega_0$. If
$\pi\notin\mathcal O_n\cap \mathcal D_n$, then we are in the same
situation as above, and Equation~(\ref{eq:commuteTilde}) holds. Suppose instead
that $\pi\in \mathcal O_n\cap\mathcal D_n$. Then $\pi$ labels two
classes with representatives $o_{\pi^+}$ and $o_{\pi^-}$ of $S_{G}$, and
$G_{\pi^+}$ and $G_{\pi^-}$ are two copies of $\Z_2$, whose only spin
$p$-block has defect zero, and consists of the (only) non-trivial character.
Denote by $\{\epsilon_+\}$ and $\{\epsilon_-\}$ the spin $p$-blocks of
$G_{\pi^+}$ and $G_{\pi^-}$, respectively.

Now, even though $\sigma(\gamma)=1$, $\gamma$ labels just one $p$-block of
$G_{\pi^+}$ (and of $G_{\pi^-}$). Since
$\sigma(\gamma')=-\sigma(\gamma)=-1$, it follows that $\gamma'$
labels two $p$-blocks with defect zero of $G'_{\pi}$. In particular,
$\Irr(B(G'_{\pi}))$ is the union of the $p$-blocks $\{\xi_{\gamma'}^+\}$ and
$\{\xi_{\gamma'}^-\}$. 

We then define $I_{\pi}:\C\{\epsilon_+\}\oplus \C\{\epsilon_-\}\longrightarrow
\Irr(B(G'_{\pi}))$ by setting $I_{\pi}(\epsilon_+)=\xi_{\gamma'}^+$ and 
$I_{\pi}(\epsilon_-)=\xi_{\gamma'}^-$.
Let $\eta,\,\delta\in\{-1,1\}$ and $\lambda\in E_{\gamma,w}$. 
We have
$r^{\pi^{\delta}}(\zeta_{\lambda}^{\eta})=a(\zeta_{\lambda}^{\eta},\epsilon_{\delta})\epsilon_{\delta}$,
and iterating Theorem~\ref{MNcoefAntilde}, we obtain
$$a(
\zeta_{\lambda}^{\eta},\epsilon_{\delta})=\sum_{i_1,\ldots,i_l}a(\zeta_{\lambda_{i_0}}^{\beta_{i_0}},\zeta_{\lambda_{i_1}}^{\beta_{i_1}})a(\zeta_{\lambda_{i_1}}^{\beta_{i_1}},\zeta_{\lambda_{i_2}}^{\beta_{i_2}})\cdots
a(\zeta_{\lambda_{i_{l-1}}}^{\beta_{i_{l-1}}},\zeta_{\lambda_{i_{l}}}^{\beta_{i_{l}}}),$$
where $\lambda_{i_0}=\lambda$, $\beta_{i_0}=\eta$,
$\zeta_{\lambda_{i_{l}}}^{\beta_{i_{l}}}=\epsilon_{\delta}$,
$\beta_{i_j}\in\{-1,1\}$ for all $1\leq j\leq l-1$, and $\zeta_{\lambda_{i_j}}^{\beta_j}$ is
obtained from  $\zeta_{\lambda_{i_{j-1}}}^{\beta_{j-1}}$ by removing a
$\pi_j$-bar from $\lambda_{i_{j-1}}$ for all $1\leq j\leq l$.
Similarly, 
we have
${r'}^{\pi}(\xi_{\Psi(\lambda)}^{\eta'})=a(\xi_{\Psi(\lambda)}^{\eta'},
\xi_{\gamma'}^+)\xi_{\gamma'}^+
+a(\xi_{\Psi(\lambda)}^{\eta'},\xi_{\gamma'}^-)\xi_{\gamma'}^-$ with
$$a(
\xi_{\Psi(\lambda)}^{\eta'},\xi_{\gamma'}^{\delta})=\sum_{i_1,\ldots,i_l}a(\xi_{\Psi(\lambda_{i_0})}^{\beta_{i_0}'},\xi_{\Psi(\lambda_{i_1})}^{\beta_{i_1}'})a(\xi_{\Psi(\lambda_{i_1})}^{\beta_{i_1}'},\xi_{\Psi(\lambda_{i_2})}^{\beta_{i_2}'})\cdots
a(\xi_{\Psi(\lambda_{i_{l-1}})}^{\beta_{i_{l-1}}'},\xi_{\Psi(\lambda_{i_{l}})}^{\beta_{i_{l}}'}),$$
where $\beta_{i_j}'=\delta_{\overline p}(\lambda_{i_j})\delta_{\overline
p}(\Psi(\lambda_{i_j}))\beta_{i_j}$ for all $0\leq j\leq l-1$,
$\eta'=\beta'_{i_0}$, and $\beta'_{i_l}=\delta$.

Write $\mathfrak f=(\pi_1)$ if $n=pw$ or $\mathfrak f=(\pi_1,1)$ if $n=pw+1$. 
Note that $\lambda_{i_{l-1}}$ is a partition of $|\mathfrak f|$, and if 
$\lambda_{i_{l-1}}\neq \mathfrak f$, then for
$1\leq j\leq l$,
Theorem~\ref{MNcoefcroise} and Remark~\ref{rk:degeneree2} give that
$$
a\left(\xi_{\Psi(\lambda_{i_{j-1})}}^{\beta_{i_{j-1}}'},
\xi_{\Psi(\lambda_{i_{j}})}^{\beta_{i_j}'}\right)=
\delta_{\overline p}(\lambda_{i_{j-1}})\delta_{\overline
p}(\Psi(\lambda_{i_{j-1}}))\delta_{\overline
p}(\lambda_{i_j})\delta_{\overline p}(\Psi(\lambda_{i_j}))
a\left(\zeta_{\lambda_{i_{j-1}}}^{\beta_{i_{j-1}}},
\zeta_{\lambda_{i_{j}}}^{\beta_{i_j}}\right),
$$
and it follows that
%
%
\begin{equation}
\label{eq:preuve1}
a(\zeta_{\lambda}^{\eta},\epsilon_{\delta})=\delta_{\overline{p}}(\lambda)
\delta_{\overline{p}}(\Psi(\lambda))a(\xi_{\Psi(\lambda)}^{\eta\delta_{\overline{p}}(\lambda))\delta_{\overline{p}}(\Psi(\lambda))},\xi_{\gamma'}^{\delta}).
\end{equation}
In particular, one has
\begin{equation}
\label{preuve2}
\begin{split}
I_{\pi}\left((r^{\pi^+}+r^{\pi^-})(\zeta_{\lambda}^{\eta})\right)&=a(\zeta_{\lambda}^{\eta},\epsilon_{+})\xi_{\gamma'}^++a(\zeta_{\lambda}^{\eta},\epsilon_{-})
\xi_{\gamma'}^-\\
&=\delta_{\overline{p}}(\lambda)
\delta_{\overline{p}}(\Psi(\lambda)){r'}^{\pi}\left(\xi_{\Psi(\lambda)}^{\eta\delta_{\overline{p}}(\lambda)
\delta_{\overline{p}}(\Psi(\lambda))}\right)\\&=
{r'}^{\pi}\left(I(\zeta_{\lambda}^{\delta})\right).
\end{split}
\end{equation}
Furthermore, Equation~(\ref{eq:valdiff}), Theorem~\ref{MNCabanes} and a
computation similar to that in the proof of
Proposition~\ref{MNcoefcroise}
give
$$a(\zeta_{\mathfrak
f}^{\pm},\epsilon_{\delta})=\delta_{\overline{p}}(\mathfrak f)
\delta_{\overline{p}}(\Psi(\mathfrak f))a\left(\xi_{\Psi(\mathfrak
f)}^{\pm
\delta_{\overline{p}}(\mathfrak f)
\delta_{\overline{p}}(\Psi(\mathfrak f))},\xi_{\gamma'}^{\delta}\right).$$
So, if $\lambda_{i_{l-1}}=\mathfrak f$, then Equation~(\ref{eq:preuve1})
and thus Equation~(\ref{preuve2}) also hold. 
In summary, we have proved that
%
$$I_\pi\circ\left(r^{\pi^+}+r^{\pi^-}\right)={r'}^\pi\circ I.$$
Finally, by the argument of the proof of Theorem~\ref{theo:mainAn},
we obtain for $\widehat{I}$ a decomposition as in
Equation~(\ref{eq:guguAn}).\\

We now prove that $\widehat I$ satisfies property (ii) of a
Brou\'e isometry.
Assume that $x\in G$ is $p$-singular and $x'\in G'$ is 
$p$-regular.
If $x\notin C_{G}$, then $\widehat{I}(x,x')=0$ (see the proof of
Corollary~\ref{cor:isoparfaitegene}). Otherwise, $x\in C_{G}$, and
without loss of generality, we can assume that $x=z^{k}t_{\beta}$ for
some $\beta\in\mathcal P_n$ and $k\in\{0,1\}$. Note that $z^kt_{\beta}\in
C_{G}$ means that $\beta$ has at least one part of length divisible by
$2p$. In particular, $\beta\notin \mathcal O_n$. If $\beta\notin D_n^-$ (when
$G=\tSym_n$) or $\beta\notin D_n^+$ (when $G=\tAlt_n$), then
Propositions~\ref{prop:classschur} and~\ref{prop:splitAntilde} imply
that $X_{\lambda}^{\pm}(z^kt_{\beta})=0$ for all $\lambda\in E_{\gamma,w}$, and
$I(x,x')=0$ by Equation~(\ref{eq:Ichapeau}). Hence, we can suppose that
$\beta\in \mathcal D_n^-$ if $G=\tSym_n$, or that $\beta\in\mathcal D_n^+$ if
$G=\tAlt_n$. Therefore, $X_{\lambda}^{\pm}(z^k
t_{\beta})\neq 0$ if and only if $\lambda=\beta$.
Furthermore, if we write
$\beta^{(\overline{p})}=(\beta^{\,0},\ldots,\beta^{\,(p-1)/2})$ for the
$\overline{p}$-quotient of $\beta$, then the parts of $\beta$ divisible
by $p$ are the parts of
$p\cdot \beta^{\,0}$ (see~\cite[p. 27]{olsson}).
Hence, the definition of $\Psi$ gives
$\beta^{(\overline{p})}=\Psi(\beta)^{(\overline{p})}$,
and $\Psi(\beta)$ has non-trivial parts divisible by $p$. It follows that 
$Y_{\Psi(\lambda)}^+(x')=Y_{\Psi(\lambda)}^-(x')$, because $x'$ is
$p$-regular.
Using Equation~(\ref{eq:Ichapeau}), we obtain
$$\widehat{I}(x,x')=\left (
X_{\beta}^+(z^kt_{\beta})+X_{\beta}^-(z^kt_{\beta})\right)
Y_{\Psi(\beta)}^+(x)=0$$
by Equations~(\ref{eq:valdmoinsSn}) and~(\ref{eq:valDnplus}).
Note that we derive from Remark~\ref{rk:adjoint} and a similar
computation that, if $x$ is $p$-regular and $x'$ is $p$-singular, then
$\widehat{I}(x,x')=0$.

Finally, we show that $\widehat{I}$ satisfies property (i) of a Brou\'e
isometry. Note that the $\mathcal E$ have size $1$ or $2$, and all the assumptions of Theorem~\ref{th:broue} are
satisfied.  

First, we consider the case $G=\tSym_n$. Take $\Phi\in \Z
\mathfrak b_{\gamma,w}^{\vee}$, where $\mathfrak b_{\gamma,w}$ is a
$\Z$-basis of
$\Z\Irr(B_{\gamma,w})^{C_{\tSym_n}}$ as in Remark~\ref{rk:baseadaptee}. 
By Corollary~\ref{blocs} and
Proposition~\ref{prop:classschur}, for
$x\in\tAlt_n$, we have $\Phi(x)\neq 0$ only if $x$ is $p$-regular. Thus,
again by Corollary~\ref{blocs} (applied to $\tAlt_n$ with respect to
the set of $p$-regular elements),
$\Res_{\tAlt_n}^{\tSym_n}(\Phi)$ is a projective character of $\tAlt_n$.
Let $x$ be a $p$-regular element of $\tAlt_n$. In particular,
$x=x_{C_{\tSym_n}}$.
Since $\Phi(x)=\Res_{\tAlt_n}^{\tSym_n}(\Phi)(x)$, it follows that
$\Phi(x)$ is the value of some projective character of $\tAlt_n$.

Let $\pi\in\Omega_0$ and $\phi\in\Z\mathfrak b_{\pi}$, 
where $\mathfrak b_{\pi}$ is a $\Z$-basis of
$\Z\Irr(B_{\pi})^{C_{\tSym_n}\cap G_{\pi}}$ as in
Remark~\ref{rk:baseadaptee}.
Now, we apply the previous computations to $G_{\pi}$,
$G'_{\pi}$ and $I_{\pi}$. We conclude that the condition (2) of
Theorem~\ref{th:iso} holds for $I_{\pi}$. Hence, Remark~\ref{rk:point3} 
gives the condition (3) of
Theorem~\ref{th:iso} for $I_{\pi}$ and we deduce as in the proof of
Theorem~\ref{th:broue} that $J_{\pi}(\Z \mathfrak b_{\pi}^{\vee})=\Z{
\mathfrak{b}}_{\pi}'^{\vee}$. Hence, $J^{*-1}_{\pi}(\Z \mathfrak
b_{\pi})\subseteq\Z \mathfrak{b}_{\pi}'$. Since $\mathcal E$ have size
$1$ or $2$ and $p$ is odd, we have
$l'_{\pi}(J_{\pi}^{*-1}(\phi))(x')\in\mathcal R$ for all $x'\in\tSym_m$.
We conclude with the argument of the proof of Theorem~\ref{th:broue} that
$\widehat{I}(x,x')/|\Cen_G(x)|\in\mathcal R$. Similarly, because of
Remark~\ref{rk:adjoint}, if $x'\in\tAlt_m$, then
$\widehat{I}(x,x')/|\Cen_{G'}(x')|\in\mathcal R$.\medskip
%
%

Assume now that $x\notin \tAlt_n$. By Equation~(\ref{eq:Ichapeau}) and
Proposition~\ref{prop:classschur}, $\widehat{I}(x,x')\neq 0$ only if
$x=z^ut_{\beta}$ and $x'=z^vt'_{\Psi(\beta)}$ with $\beta\in\mathcal D_n^-$
and $u,\,v\in\{0,1\}$. 
In this case,
Equation~(\ref{eq:valdmoins}) gives
\begin{equation}
\label{eq:vali}
\widehat{I}(z^ut_{\beta},z^{v}t'_{\beta'})=\pm  
i^{\frac{n+m-\ell(\beta)-\ell(\beta')-2}{2}}\sqrt{\beta_1\cdots\beta_{k}\beta'_1\cdots\beta'_{k'}},
\end{equation}
where $\beta=(\beta_1,\ldots,\beta_k)$ and
$\Psi(\beta)=(\beta'_1,\ldots,\beta'_{k'})$.
However,
we derive from the proof of \cite[Theorem 4.3]{olsson} that
$\nu_{p}(\beta_1\cdots\beta_k))=p^{|\beta^0|}
\nu_p(\operatorname{prod}(\beta^0))$,
where $\nu_p$ is the $p$-valuation,
$\beta^{(\overline{p})}=(\beta^0,\ldots,\beta^{(p-1)/2})$ is the
$\overline{p}$-quotient of $\beta$, and $\operatorname{prod}(\beta^0)$
is the product of the lengths of the parts of $\beta^0$.
Hence,
\begin{equation}
\label{eq:valuation}
\nu_p(\beta_1\cdots\beta_k))=\nu_p(\beta'_1\cdots\beta'_{k'}),
\end{equation}
Furthermore, 
$$|\Cen_{\widetilde{\Sym}_n}(z^u t_{\beta})|=2\prod_{i=1}^k
\beta_i\quad\textrm{and}\quad |\Cen_{\widetilde{\Sym}_m}(z^{v}
t'_{\beta'})|=2\prod_{i=1}^{k'}
\beta'_i,$$
because $\beta,\,\beta'\in\mathcal D^-$. By Equation~(\ref{eq:valuation}),
there are integers $a$ and $b$ prime to $p$ such that
$\prod \beta_i=\nu_p(\beta_1\cdots\beta_k)a$ and $\prod
\beta'_i=\nu_p(\beta_1\cdots\beta_k)b$.
Therefore, Equation~(\ref{eq:vali}) implies that
$$\frac{\widehat{I}(z^ut_{\beta},z^{v}t'_{\beta'})}{|\Cen_{\widetilde{\Sym}_n}(t_{z^i\beta})|}=\pm
i^{\frac{n+m-\ell(\beta)-\ell(\beta')-2}{2}}\frac{\sqrt{ab}}{2a}.$$
Since $\pm
i^{\frac{n+m-\ell(\beta)-\ell(\beta')-2}{2}}\sqrt{ab}\in\mathcal R$ and $2a$ is
prime to $p$, we deduce that
$$\frac{\widehat{I}(z^ut_{\beta},z^{v}t'_{\beta'})}{|\Cen_{\widetilde{\Sym}_n}(z^ut_{\beta})|}\in\mathcal
R.$$ Similarly, we have
$\frac{\widehat{I}(z^ut_{\beta},z^{v}t'_{\beta'})}{|\Cen_{\widetilde{\Sym}_{m}}(z^{v}t_{\beta'})|}\in\mathcal
R$.

Assume now that $G=\tAlt_n$. Take $\Phi\in \Z \mathfrak b_{\gamma,w}^{\vee}$,
where $\mathfrak b_{\gamma,w}$ is a $\Z$-basis of
$\Z\Irr(B_{\gamma,w})^{C_{\tAlt_n}}$ as in Remark~\ref{rk:baseadaptee}.
By Corollary~\ref{blocs}, there are integers $a_{\lambda}$ (for
$\lambda\in E_{\gamma,w}$ with $\sigma(\lambda)=-1$) and
$a_{\lambda}^{\pm}$ (for
$\lambda\in E_{\gamma,w}$ with $\sigma(\lambda)=1$) such that
\begin{equation}
\label{eq:phicliff}
\Phi=\sum_{\sigma(\lambda)=-1}a_{\lambda}\zeta_{\lambda}+\sum_{\sigma(\lambda)=1}\left(a_{\lambda}^+\zeta_{\lambda}^++a_{\lambda}^-\zeta_{\lambda}^-\right),
\end{equation}
and Clifford theory gives
\begin{equation}
\label{eq:cliffPhi}
\Ind_{\tAlt_n}^{\tSym_n}(\Phi)
=\sum_{\sigma(\lambda)=-1}a_{\lambda}\left(\xi_{\lambda}^++\xi_{\lambda}^-\right)
+\sum_{\sigma(\lambda)=1}(a_{\lambda}^++a_{\lambda}^-)\xi_{\lambda}.
\end{equation}

Let $x$ be a $p$-regular element of $\tAlt_n$. Assume $x=z^kt_\beta$ with
$\beta\in \mathcal O_n$ and $\beta\notin
\mathcal D_n^+$. In particular, one has $x=x_{C_{\tAlt_n}}$, and for $\lambda\in \mathcal D_n^+$, we have
$$\xi_{\lambda}(x)=\zeta_{\lambda}^+(x)+\zeta_{\lambda}^-(x)=2\zeta_{\lambda}^+(x)=2\zeta_\lambda^-(x),$$
and it follows that
\begin{eqnarray*}
\Ind_{\tAlt_n}^{\tSym_n}(\Phi)(x)
&=&\sum_{\sigma(\lambda)=-1}a_{\lambda}\left(\xi_{\lambda}^+(x)+\xi_{\lambda}^-(x)\right)
+\sum_{\sigma(\lambda)=1}(a_{\lambda}^++a_{\lambda}^-)\xi_{\lambda}(x)\\
&=&2\left(
\sum_{\sigma(\lambda)=-1}a_{\lambda}\zeta_{\lambda}(x)+\sum_{\sigma(\lambda)=1}\left(a_{\lambda}^+\zeta_{\lambda}^+(x)+a_{\lambda}^-\zeta_{\lambda}^-(x)\right)\right)\\
&=&2\Phi(x).
\end{eqnarray*}
By Equation~(\ref{eq:cliffPhi}), Proposition~\ref{prop:classschur} and Corollary~\ref{blocs}, 
$\Ind_{\tAlt_n}^{\tSym_n}(\Phi)$ is a projective character of $\tSym_n$
and hence $\tAlt_n$.
Thus, $2\Phi(x)$ is the value of a projective character of $\tAlt_n$,
and we conclude as above, because $2$ is not divisible by $p$. 

Suppose now that $x=z^ut_\beta$ with $\beta\in\mathcal O_n\cap \mathcal
D_n^+$. By Lemma~\ref{lem:ordre}, we can assume that $x=z^{u'}o_{\beta}$
for some non-negative integer $u'$. Write
$H$ for the centralizer of $o_{\beta^{\pm}}$ in $\tAlt_n$.
Then $H=\cyc{z}\times \cyc{o_{\beta_1^{\pm}}}
\times\cdots \times \cyc{o_{\beta_k}}$
contains no elements whose cycle structure has even parts. In
particular, $\Res_{H}^{\tAlt_n}(\Phi)$ is a projective character of $H$.
Since $x\in H$, it follows that $\Phi(x)$ is the value of a projective
character of $H$, and we again conclude with the same argument as above.

Finally, it remains to show the property for $\beta\in\mathcal D_n^+$ and
$\beta\notin \mathcal O_n$. However, $\widehat{I}(z^{u'}o_{\beta},x')\neq 0$ if and only
if $x'=z^vo'_{\Psi(\beta)}$ for some non-negative integer $v$. In
particular, if $\beta':=\Psi(\beta)\in
\mathcal D_m^-$, then 
$$\widehat{I}(z^{u'}o_{\beta},z^vo'_{\beta'})=\pm
\sqrt{2}i^{\frac{n+m-k-k'-1}{2}}\sqrt{\beta_1\cdots\beta_k\,\beta'_1\cdots\beta'_{k'}}\
,$$
where $\beta=(\beta_1,\ldots,\beta_k)$ and
$\beta'=(\beta_1',\ldots,\beta_{k'}')$.
We conclude as above using Equation~(\ref{eq:valuation}).
\end{proof}

\begin{corollary}
\label{cor:brouetilde}
If $p$ is an odd prime, if $B_{\gamma,w}$ and $B_{\gamma',w}$ are $p$-blocks of $\tAlt_n$ and $\tAlt_m$ respectively, and if
$\sigma(\gamma)=\sigma(\gamma')$, then the isometry $I$ defined by
Equation~(\ref{eq:defisoTilde}) is a Brou\'e perfect isometry.
\end{corollary}

\begin{proof}
Let $\widetilde{\gamma}$ be any $\bar{p}$-core
such that $\sigma(\widetilde{\gamma})=-\sigma(\gamma)$.
Denote by $B_{\widetilde{\gamma},w}$ 
the $p$-block of $\tSym_{|\widetilde{\gamma}|+pw}$ corresponding to
$\widetilde{\gamma}$. Since
$\sigma(\gamma')=-\sigma(\widetilde{\gamma})$, by
 Theorem~\ref{theo:mainTilde}, there are Brou\'e perfect
isometries
$I_1:\Irr(B_{\gamma,w})\rightarrow\Irr(B_{\widetilde{\gamma},w})$ and
$I_2:\Irr(B_{\gamma',w})\rightarrow\Irr(B_{\widetilde{\gamma},w})$,
defined by Equation~(\ref{eq:defisoTilde}). Furthermore, we have
$$I=I_2^{-1}\circ I_1,$$
which proves the result.
\end{proof}

\section{Some other examples}
\subsection{Notation}\label{subsec:not}
For any positive  integers $k$ and $l$, we denote by $\mathcal{MP}_{k,l}$
the set of $k$-tuples of partitions $(\mu_1,\ldots,\mu_k)$ such that
$\sum |\mu_i|=l$.

Let $H$ be a finite group and $w$ be a positive integer. We consider
the wreath product $G=H\wr\sym_w$, that is, the semidirect product
$G=H^w\rtimes\sym_w$ where
$\sym_w$ acts on $H^w$ by permutation. Write
$N=|\Irr(H)|$ and $\Irr(H)=\{\psi_i\,|\,1\leq i\leq N\}$, and denote by
$g_i$ ($1\leq i\leq N$) a system of representatives for the conjugacy
classes of $H$.

The irreducible characters of $G$ are parametrized by
$\mathcal{MP}_{N,w}$ as follows. For
$\mm=(\mu_1,\ldots,\mu_N)\in\mathcal{MP}_{N,w}$, consider the
irreducible character $\phi_{\mm}$ of $\Irr(H^w)$ given by
\begin{equation}
\label{eqphimu}
\phi_{\mm}=\prod_{i=1}^{N}\
\underbrace{\psi_i\otimes\ldots\otimes
\psi_i}_{|\mu_i|\textrm{ times}},
\end{equation}
which, by~\cite[p.154]{James-Kerber}, can be extended to an irreducible character
$\widehat{\phi}_{\mm}=\prod_{i=1}^N\widehat{\psi_i^{|\mu_i|}}$ of
its inertia subgroup $I_G(\phi_{\mm})=\prod_{i=1}^N H\wr\sym_{|\mu_i|}$.
The irreducible character of $G$ corresponding to $\mu$ is then given by
$$\carwr_{\mm}=\Ind_{I_G(\phi_{\mm})}^G\left(\prod_{i=1}^N\widehat{\psi_i^{|\mu_i|}}\otimes\carsym_{\mu_i}\right),$$
where $\carsym_{\mu_i}$ denotes the irreducible character of
$\sym_{|\mu_i|}$ corresponding to the partition $\mu_i$ of $|\mu_i|$.

Let $(h_1,\ldots,h_w;\sigma)\in G$ with
$h_1,\ldots,h_w\in H$ and $\sigma\in\sym_w$. For any $k$-cycle
$\kappa=(j,\kappa j,\ldots,\kappa^{k-1}j)$ in $\sigma$, we define the
cycle product $$g((h_1,\ldots,h_w;\sigma); \, \kappa)=h_jh_{\kappa^{-1}j}\cdots
h_{\kappa^{-(k-1)}j}.$$ If $\sigma$ has cycle structure $\pi$, then we
form the $N$-tuple of partitions $(\pi_1,\ldots,\pi_N)$ from $\pi$,
where any cycle $ \kappa$ in $\pi$ gives a cycle of the same length in
$\pi_i$ if $g((h_1,\ldots,h_w;\sigma); \,  \kappa)$ is conjugate to $g_i$ in $H$. 
The $N$-tuple 
\begin{equation}
\label{eq:strucyclcouronne}
\mathfrak{s}(h_1,\ldots,h_w;\sigma)=(\pi_1,\ldots,\pi_N)\in\mathcal{MP}_{N,w}
\end{equation}
describes the cycle structure of $(h_1,\ldots,h_w;\sigma)$, and two
elements of $G$ are conjugate if and only if they have the same cycle
structure (see~\cite[4.2.8]{James-Kerber}). In particular, the conjugacy classes of $G$ are labeled by
$\mathcal{MP}_{N,w}$.

\subsection{Isometries between symmetric groups and natural subgroups}
\label{subsec:sym}
Let $n$ be a positive integer and $p$ be a prime. We denote by
$\mathcal{P}_n$ the set of partitions of $n$. Write $\carsym_{\lambda}$
for the irreducible character of the symmetric group $\sym_n$
corresponding to the partition $\lambda\in\mathcal P_n$. Recall that to
every $\lambda\in\mathcal P_n$, we can associate its $p$-core
$\lambda_{(p)}$ and its $p$-quotient
$\lambda^{(p)}=(\lambda_1,\ldots,\lambda_p)$ 
(see for example~\cite[p.\,17]{olsson}). Moreover, two irreducible
characters $\carsym_{\lambda}$ and $\carsym_{\mu}$ lie in the same $p$-block if
and only if $\lambda$ and $\mu$ have the same $p$-core. For $B_{\gamma}$ the $p$-block of
$\sym_n$ corresponding to a fixed $p$-core $\gamma$, we define the
$p$-weight $w$ of $B_{\gamma}$ by setting $w=(n-|\gamma|)/p$. Then
$\Irr(B_{\gamma})$ is parametrized by $\mathcal{MP}_{p,w}$.
Now, we set $G_{p,w}=(\Z_p\rtimes\Z_{p-1})\wr \sym_w$. We recall that
$\Irr(\Z_p\rtimes\Z_{p-1})=\{\psi_1,\ldots,\psi_p\}$ with the following
convention. If $p$ is odd (respectively $p=2$), 
then put $p^*=(p+1)/2$ (respectively $p^*=2$). Then we can choose
the labeling such that $\psi_i(1)=1$ for $i\neq p^*$ and
$\psi_{p^*}(1)=p-1$. Fix now $\eta$ and $\omega$ generators 
of $\Z_{p-1}$ and $\Z_p$ respectively. Write $g_i=\eta^i$ for $1\leq i\leq
p-1$ and $g_p=\omega$. Then the elements $g_i\in\Z_p\rtimes\Z_{p-1}$
form a system of representatives for the conjugacy classes of $\Z_p\rtimes\Z_{p-1}$. As
explained in~\S\ref{subsec:not}, the irreducible characters and
conjugacy classes of $G_{p,w}$ are labeled by $\mathcal{MP}_{p,w}$. As
above, for $\mm\in\mathcal{MP}_{p,w}$, we
write $\carwr_{\mm}$ for the corresponding irreducible character of
$G_{p,w}$.

\begin{theorem}We keep the notation as above, and define
the linear map $I:\C\Irr(B_{\gamma})\rightarrow\C\Irr(G_{p,w})$ by
$$I(\chi_{\lambda})=(-1)^{|\lambda_{p^*}|}\delta_p(\lambda)\theta_{\widetilde{\lambda}^{(p)}},$$
where
$\widetilde{\lambda}^{(p)}$ is obtained from the $p$-quotient
$\lambda^{(p)}$ of $\lambda$ replacing $\lambda_{p^*}$ by its
conjugate, and $\delta_p(\lambda)$ is the $p$-sign of $\lambda$.
Then $I$ is a generalized
perfect isometry with respect to the $p$-regular elements of $\sym_n$
and the set $C'$ of elements of $G_{p,w}$ with cycle structure
$\pi=(\pi_1,\ldots,\pi_p)$ satisfying $\pi_p=\emptyset$.
\label{theo:BrGr}
\end{theorem}
\begin{proof}
Let $S$ be the set of elements
of $\sym_n$ with cycle decomposition
$\sigma_1\cdots\sigma_r$ (where we omit trivial cycles), such that $\sigma_i$ is a $q_ip$-cycle for
some positive integer $q_i$, and
let $C$ be the set of $p$-regular elements of $\sym_n$. The sets $S$ and $C$
are unions of $\sym_n$-conjugacy classes, and $1\in S$. Moreover, 
$\tau_1\cdots\tau_k$ is the cycle decomposition of $\tau\in C$ if and
only if $\tau_i$ has $p'$-length. Hence the cycle decomposition
with disjoint support in $\sym_n$ proves that (1), (2) and (3) 
of Definition~\ref{defMN} hold with $G_{\sigma_S}=\sym_{\overline{J}}$
whenever $\sigma=\sigma_S\sigma_C$ with $\sigma_S\in S$ and $\sigma_C\in
C$, and $J$ is the support of $\sigma_S$.
Denote by $\Lambda$
the set of classes consisting of elements of $S$ and define
$$\Gamma_0=\bigcup_{b\leq w}\mathcal
P_b.$$
Write $\Lambda_0$ for the classes of $S$ parametrized by
$p\cdot\Gamma_0$. 
For each $\beta\in
\Gamma_0$, we choose a representative $s_{\beta}\in S$ 
in the class of $\Lambda_0$ labeled by $p\cdot
\beta$ with support in $\{n-p|\beta|+1,\ldots,n\}$. Then 
$G_{s_{\beta}}=\sym_{n-p|\beta|}\subseteq \Cen_{\sym_n}(s_{\beta})$. Denote by 
$\Irr(B_{\gamma}(\sym_{n-p|\beta|}))$ the set of irreducible characters
of $\sym_{n-p|\beta|}$ labeled by partitions with $p$-core $\gamma$,
and define
$r^{\beta}:\C\Irr(B_{\gamma})\rightarrow\C\Irr(B_{\gamma}(\sym_{n-p|\beta|}))$ 
by applying~\cite[2.4.7]{James-Kerber} to the cycles of $p\cdot\beta$.
Then $\sym_n$ has an MN-structure with respect to
$C$ and $B_{\gamma}$ in the sense of Definition~\ref{defMN}.

Now, write $S'$ (respectively $C'$) for the set of elements of $G_{p,w}$ with cycle
structure $(\pi_1,\ldots,\pi_p)\in\mathcal{MP}_{p,b}$ for some
$b\leq w$ (respectively$(\pi_1,\ldots,\pi_p)\in\mathcal{MP}_{p,w}$),  such that
$\pi_1=\cdots=\pi_{p-1}=\emptyset$ (respectively $\pi_p=\emptyset$). In
particular, the classes of $S'$ are also parametrized by $\Gamma_0$. Let
$s'_{\beta}\in S'$ be with cycle structure $(\emptyset,\ldots,\emptyset,\beta)$
for $\beta\in\Gamma_0$. Assume that the support of $s'_{\beta}$ is
$\{w-|\beta|+1,\ldots,w\}$. Then $G_{p,w-|\beta|}$ lies in
$\Cen_{G_{p,w}}(s)$, and we define
$r'^{\beta}:\C\Irr(G_{p,w})\rightarrow\C\Irr(G_{p,w-|\beta|})$ 
by applying~\cite[Theorem 4.4]{Pfeiffer} to the cycles of $\beta$. 
Then $G_{p,w}$ has an MN-structure
with respect to $C'$ and $\Irr(G_{p,w})$. 

Let $q=pa$. Define the set $M_a(\lambda^{(p)})$ of $p$-multipartitions
of $w-a$ obtained from $\lambda^{(p)}$ by removing an $a$-hook. Recall that
the canonical bijection $f$ (defined in~\cite[Proposition 3.1]{morrisolsson})
induces a bijection $M_q(\lambda)\rightarrow M_a(\lambda^{(p)}),\
\mu\mapsto\mu^{(p)}$. Write
$$\widetilde{\theta}_{\lambda^{(p)}}=(-1)^{|\lambda_{p^*}|}
\theta_{\widetilde{\lambda}^{(p)}},$$
and assume $\beta=(\beta_1)$. 
Then 
\begin{equation}
\label{eq:MNwr}
r^{\beta}\left(\widetilde{\theta}_{\lambda^{(p)}}\right)=\sum_{\mu\in
M_{p|\beta_1|}(\lambda)}{\alpha'}_{\mu}^{\lambda}\,\widetilde{\theta}_{\mu^{(p)}},
\end{equation}
where ${\alpha'}_{\mu}^{\lambda}=(-1)^{L(f(c_{\mu}^{\lambda}))}$. See the
proof of~\cite[Proposition 3.8]{JB} for more details.
For multiples $q_1,\ldots,q_k$ of $p$, define inductively
the set $M_{q_1,\ldots,q_k}(\lambda)$ of partitions $\mu$ of $n-\sum q_i$ such
that $\mu\in M_{q_k}(\nu)$ for some $\nu\in
M_{q_1,\ldots,q_{k-1}}(\lambda)$. 
Let $\beta=(\beta_1\geq\cdots\geq\beta_k)\in\Gamma_0$.
Applying recursively formula~(\ref{eq:MNwr}) to the
cycles of $\beta$, we obtain
\begin{equation}
r^{\beta}\left(\widetilde{\theta}_{\lambda^{(p)}}\right)=\sum_{\mu\in
M_{p|\beta_1|,\ldots,p|\beta_k|}(\lambda)}a'(\lambda,\mu)\,\widetilde{\theta}_{\mu^{(p)}}.
\label{eq:MNwritere}
\end{equation}
Similarly, the Murnaghan-Nakayama rule in $\sym_n$ gives
\begin{equation}
r^{\beta}\left(\carsym_{\lambda}\right)=\sum_{\mu\in
M_{p|\beta_1|,\ldots,p|\beta_k|}(\lambda)}a(\lambda,\mu)\,\carsym_{\mu}
\label{eq:MNSnitere}.
\end{equation}
Now, with the above notation, Equation~(\ref{eq:lienjambef}) gives
$\alpha_{\mu}^{\lambda}=\delta_p(\lambda)\delta_p(\mu){\alpha'}_\mu^\lambda$,
and by the same
argument as in the proof of Theorem~\ref{theo:mainAn}, we obtain
$$a(\lambda,\mu)=\delta_p(\lambda)\delta_p(\mu)a'(\lambda,\mu).$$
It follows that 
\begin{eqnarray*}
r^{\beta}\left(I(\carsym_{\lambda})\right)&=&\delta_p(\lambda)\sum_{\mu}a'(\lambda,\mu)\widetilde{\theta}_{\mu^{(p)}},\\
&=&\sum_{\mu}\delta_p(\lambda)\delta_p(\mu)a'(\lambda,\mu)\delta_p(\mu)\widetilde{\theta}_{\mu^{(p)}},\\
&=&\sum_{\mu}a(\lambda,\mu)I(\chi_{\mu}),\\
&=&I\left(r^{\beta}(\chi_{\lambda})\right).
\end{eqnarray*}
The result now follows from Corollary~\ref{cor:isoparfaitegene}.
\end{proof}

\begin{corollary}\label{theo:BroueSn}
Assume furthermore that $p>w$. Then the isometry defined
in Theorem~\ref{theo:BrGr} is a Brou\'e isometry. In particular,
Brou\'e's perfect isometry Conjecture holds for symmetric groups.
\end{corollary}

\begin{proof}
We apply Theorem~\ref{th:broue}. 
\end{proof}

\subsection{Osima's perfect isometry} 
\label{subsec:osima}
Using Theorem~\ref{th:iso}, we also can prove the following well-known
result (see \cite[Proposition 5.11]{KOR}).
\begin{theorem}
Let $n$ be an integer and $p\leq n$ be a prime. Let $B$ be a $p$-block
of $\sym_n$ labeled by the $p$-core $\gamma$. Assume that $B$ has
weight $w$.
Then the map defined by
$$I(\carsym_{\lambda})=\delta_p(\lambda)\theta_{\lambda^{(p)}}$$ 
between $B$ and $\Irr(\Z_p\wr\sym_w)$ 
induces a generalized perfect isometry
with respect to the $p$-regular elements of $\sym_n$ and the set of
elements $x\in\Z_p\wr\sym_w$ with cycle structure $g(x)$ satisfying
$g(x)_1=\emptyset$ (here, the first coordinate of $g(x)$ correspond to
the trivial class).
\label{theo:osima}
\end{theorem}

\begin{proof}
The proof is analogue to that of Theorem~\ref{theo:BrGr}.
\end{proof}

\subsection{Isometries between blocks of wreath products}
\label{subsec:wreath}
In this section, we fix a positive integer $l$ and a prime number $p$
such that $p$ does not divide $l$, and we consider the groups
$G_n=\Z_l\wr\sym_n$, where $n$ is any positive integer. Write 
$\Z_{l}=\{\zeta_1,\zeta_2,\ldots,\zeta_{l}\}$ and
$\Irr(\Z_{l})=\{\psi_1,\ldots,\psi_{l}\}$.

Following~\cite[Theorem 1]{osimaII}, we recall that two irreducible
characters $\theta_{\mm}$ and $\theta_{\mm'}$ corresponding to
$\mm=(\mu_1,\ldots,\mu_{l})$ and $\mm'=(\mu'_1,\ldots,\mu'_l)$
of $G_n$ lie in the same
$p$-block $B$ if and only if, for every $1\leq i\leq l$, the partitions
$\mu_i$ and $\mu'_i$ have the same $p$-core $\gamma_i$
and same $p$-weight $b_i$. The
tuple $b=(b_1,\ldots,b_{l})$ (respectively
$\gamma=(\gamma_1,\ldots,\gamma_{l})$) is called the $p$-weight of
$B$ (respectively the $p$-core of $B$). We denote by $\mathcal E_{\gamma,b}$
the set of $l$-multipartitions $\mm=(\mu_1,\ldots,\mu_l)$ such that
$(\mu_i)_{(p)}=\gamma_i$ and the $p$-weight of $\mu_i$ is $b_i$.

\begin{theorem}
Let $n$ and $m$ be any two positive integer. As above, we write 
$\Irr(G_n)=\{ \theta_{\mm}; \, \mm \Vdash n\} $ and $\Irr(G_m)=\{\theta_{\mm'}; \, \mm \Vdash m\}$ for the sets of
irreducible characters of $G_n$ and $G_{m}$.
Let $B$ and $B'$ be two $p$-blocks of $G_n$ and $G_m$, with $p$-cores
$\gamma=(\gamma_1,\ldots,\gamma_{l})$ and
$\gamma'=(\gamma'_1,\ldots,\gamma'_l)$ respectively.
Assume that $B$ and $B'$ have the same $p$-weight $b=(b_1,\ldots,b_l)$. 
Define
$$I(\theta_{\mm})=\left(\prod_{i=1}^l
\delta_p(\mu_i)\delta_p(\Psi(\mu_i))\right)\theta_{\psi(\mm)},$$
where
$\Psi$ is the map defined before Lemma \ref{Psi},
 $\psi(\mm)=(\Psi(\mu_1),\ldots,\Psi(\mu_l))$, and $\delta_p(\mu_i)$
is the $p$-sign of $\mu_i$. Then $I$ induces a Brou\'e perfect 
isometry between $B$ and $B'$.
\label{theo:couronne}
\end{theorem}

\begin{proof}
First, we notice that $\psi(\mathcal E_{\gamma,b})=\mathcal
E_{\gamma',b}$. Let $g=(g_1,\ldots,g_n;\sigma)\in G_n$. Write
$\sigma=\sigma_S\sigma_C$, where all the cycles of $\sigma_S$ have
length divisible by $p$, and $\sigma_C$ is a $p$-regular element. 
Define $g_S=(g_{S,1},\ldots,g_{S,n};\sigma_S)$ (resp.
$g_C=(g_{C,1},\ldots,g_{C,n};\sigma_C)$) by setting $g_{S,i}=g_i$
(respectively
$g_{C,i}=g_i$) if $i$ lies in the support of $\sigma_S$ (respectively of
$\sigma_C$) and $g_{S,i}=1$ (respectively
$g_{C,i}=1$) otherwise. Since $\sigma_S$ and $\sigma_C$ have disjoint
supports, we have the unique decomposition
$$g=g_Sg_C=g_Cg_S.$$
Denote by $S$ (respectively $C$) the set of elements $g=(t;\sigma)$ 
such that all the cycles of $\sigma$ have length divisible by $p$
(respectively prime to $p$). 
Let $\Lambda$ be the set of $l$-multipartitions
$(\pi_1,\ldots,\pi_{l})$ such that $\pi_i\in p\mathcal P$ and
$\sum|\pi_i|\leq n$. Let
$\pp=(\pi_1,\ldots,\pi_{l})\in \Lambda$. Denote by $I$ the set of
integers $1\leq i\leq l$ with $\pi_i\neq\emptyset$.
For $i\in I$, write
$u_{ij}=\sum_{k<i}|\pi_k|+\sum_{r<j}\pi_{i,r}$, where
$\pi_i=(\pi_{i,1},\ldots,\pi_{i,\ell(\pi_i)})$. Consider now
the $\pi_{i,j}$-cycle 
$$\sigma_{ij}=(n-u+u_{ij}+1, \, \ldots , \, n-u+u_{ij}+\pi_{i,j}),$$
where $u=\sum |\pi_i|$. For $1\leq k\leq n$, set $t_{ij,k}=1$ for $k\neq
n-u+u_{ij}+1$ and $t_{ij,n-u+u_{ij}+1}=\zeta_i$. Write
$t_{ij}=(t_{ij,1},\cdots,t_{ij,n};\sigma_{ij})$ and define
\begin{equation}
\label{eq:prodcourel}
t_{\pp}=\prod_{i\in I}\prod_{j=1}^{\ell(\pi_i)}t_{ij}.
\end{equation}
Then by \S\ref{subsec:not} and
Equation~(\ref{eq:strucyclcouronne}), the elements $t_{\pp}$ with
$\pi\in\Lambda$ form a set of representatives of the $G_n$-classes of $S$.
Write $G_{\pp}=G_{n-\sum|\pi_i|}$.
Note that the support of $\pp$ is $\{n-\sum|\pi_i|+1,\ldots,n\}$, and
$G_{\pp}\subseteq \Cen_{G_n}(t_{\pp})$. 
\begin{example}
For example, assume that $l=3$, $n=6$, and $p=2$. 
Write $\zeta_1=1$ and consider $\pp=(\emptyset,
(2),(2))\in\Lambda$. Then one has $u=4$, $I=\{2,3\}$, $u_{21}=0$,
$u_{31}=2$. So $\sigma_{21}=(3\ 4)$, $\sigma_{31}=(5\ 6)$ and
$$t_{21}=\left(1,1,\zeta_2,1,1,1;(3\ 4)\right)\quad\textrm{and}\quad
t_{31}=\left(1,1,1,1,\zeta_3,1;(5\ 6)\right).
$$
Finally, $t_{\pp}=t_{21}t_{31}=\left(1,1,\zeta_2,1,\zeta_3,1;(3\ 4)(5\
6)\right)$ is a representative for the class of $\Z_3\wr \sym_6$ labeled
by $\pp$. 
\end{example}

Assume that $\pp=(k)\in\Lambda$ (so that, in particular, $k$ is divisible by
$p$).
Then, for all $x\in G_{\pp}$
and $\mm\in\mathcal E_{\gamma,b}$,~\cite[Theorem 4.4]{Pfeiffer} gives 
\begin{equation}
\label{eq:MNcour}
\theta_{\mm}(t_{\pp}x)=\sum_{s=1}^{l}\psi_s(\zeta_{l})\sum_{\nu\in
M_{k}(\mu_s)}(-1)^{L\left(c_{\nu}^{\mu_s}\right)}
\theta_{\mm_s}(x),
\end{equation}
where the partitions in $\mm_s$ are the same as those in $\mm$, except the $s$-th one which is equal to $\nu$.
Applying iteratively this process to the cycles of $\pp$, we define
a linear map $r^{\pp}:\C\Irr(B)\rightarrow\C\Irr(B_{n-\sum|\pi_i|})$,
where $B_{n-\sum|\pi_i|}$ denotes the union of $p$-blocks of $G_{n-\sum|\pi_i|}$
with $p$-core $\gamma$ and $p$-weight $(a_1,\ldots,a_l)$ such that
$0\leq a_i\leq b_i$ and $\sum (b_i-a_i)=\sum |\pi_i|$.
In particular, we have $r^{\pp}(\theta_{\mm})(x)=\theta_{\mm}(t_{\pp}x)$
for all $x\in G_{n-\sum|\pi_i|}$. This defines an MN-structure for $G_n$ with
respect to $C$ and $B$.

Similarly, we define an MN-structure for $G_m$ with respect to $B'$ and
the set of $p$-regular elements of $G_m$.
Now, write $w=\sum b_k$, and denote by $\Lambda_0$ the set of
$\pp\in\Lambda$ such that $\sum|\pi_k|\leq w$. By~\cite[Theorem
4.4]{Pfeiffer}, we have
$r^{\pp}(\theta_{\mm})=r^{\pp}(\theta_{\Psi(\mm)})=0$ for every 
$\mm\in \mathcal E_{\gamma,b}$ and $\pp\in\Lambda\backslash\Lambda_0$.

Let $\pi\in \Lambda_0$ and $c$ be a part of $\pi_t$ of length $k$.
Then, by~\cite[Theorem 4.4]{Pfeiffer} (see also
Equation~(\ref{eq:MNcour})), we have
\begin{equation}
\begin{split}
r^{c}(I(\theta_{\mu}))&=
\prod_{i=1}^l\delta_p(\mu_i)\delta_p(\Psi(\mu_i))
\sum_{s=1}^{l}\psi_s(\zeta_{t})\sum_{\nu\in
M_{k}(\mu_s)}(-1)^{L\left(c_{\Psi(\nu)}^{\Psi(\mu_s)}\right)}
\theta_{\psi(\mm_s)}\\
&=\sum_{s=1}^l\psi_s(\zeta_t)\sum_{\nu\in
M_{k}(\mu_s)}(-1)^{L\left(c_{\Psi(\nu)}^{\Psi(\mu_s)}\right)}\delta_p(\mu_s)\delta_p(\Psi(\mu_s))\\
&\qquad\cdot\delta_p(\nu)\delta_p(\Psi(\nu))I(\theta_{\mm_s})\\
&=\sum_{s=1}^l\psi_s(\zeta_t)\sum_{\nu\in
M_{k}(\mu_s)}(-1)^{L\left(c_{\nu}^{\mu_s}\right)}I(\theta_{\mm_s})\\
&=I\left(r^{c}(\theta_{\mm_s})\right).
\end{split}
\label{eq:calcour}
\end{equation}
Using the argument of the proof of Theorem~\ref{theo:mainAn}
(see Equations~(\ref{eq:coeffAitere}),~(\ref{eq:acommute})
and~(\ref{eq:commuteAn})), we conclude that
$r^{\pp}(I(\theta_{\mm}))=I(r^{\pp}(\theta_{\mm}))$ for all
$\pp\in\Lambda_0$ and $\mm\in\mathcal E_{\gamma,b}$.
%

Hence, the hypotheses of
Theorem~\ref{th:broue} are satisfied, and the result holds.
\end{proof}

\begin{corollary}
Let $W_1$ and $W_2$ be Coxeter groups of type $B$. Assume that $p$ is
odd. Then two $p$-blocks of $W_1$ and $W_2$ with the same $p$-weight are
perfectly isometric (in the sense of Brou\'e).
\label{cor:broueBn}
\end{corollary}

\begin{proof}
This is a direct consequence of Theorem~\ref{theo:couronne}, noting
that a Coxeter group of type $B_n$ is isomorphic to $\Z_2\wr \sym_n$.
\end{proof}

\subsection{Isometries between blocks of Weyl groups of type $D$}
\label{subsec:typeD}

Let $n$ be a positive integer and let $W$ be a Weyl group of type $B_n$.
We keep the notation of~\S\ref{subsec:wreath}. Let $p$ be an odd prime
number. We consider the linear
character $\alpha=\theta_{(\emptyset,(n))}\in\Irr(W)$, and denote by $W'$
its kernel. Then $W'$ is a Weyl group of type $D_n$, and one has that $g\in W$ belongs to $ W'$ if
and only if its cycle structure $\mathfrak{s}(g)=(\pi_1,\pi_2)$ is such
that $\ell(\pi_2)$ is even.  Furthermore, the $W$-class of such an
element splits into two $W'$-classes if and only if $\pi_2=\emptyset$
and $\pi_1$ has only parts of even length (i.e. if $\pi_1=2\cdot\pi$ for
some partition $\pi$ of $n/2$); see~\cite[Proposition
25]{carterConjugaison}.  
We fix representatives
$t_{(2\cdot\pi,\emptyset)}^{\pm}$ for the $W'$-classes whose elements
have cycle structure $(2\cdot\pi,\emptyset)$ as follows. 
If $\pi=(\pi^1,\ldots,\pi^r)$ then write $u_i=\sum_{j<i}\pi^j$,
$\sigma_i=(u_i+1\cdots u_i+2\pi^i)$, and $t_i=( (1,\ldots,1);\sigma_i)$.
In particular, $t_1\in B_{2\pi^1}$.
Let $\rho\in B_{2\pi^1}\backslash D_{2\pi^1}$. Set 
$t_1^+=t_1$ and $t_1^-=\rho t_1\rho^{-1}$. Then $t_1^+$ and $t_1^-$ are
representatives for the two split classes of $D_{2\pi^1}$ labeled by $(
(2\pi^1),\emptyset)$.
Now, define
\begin{equation}
\label{eq:repchoixdn}
t_{(2\cdot \pi,\emptyset)}^{\pm}=t_1^{\pm}t_2\cdots t_r.
\end{equation}
Since $\rho\in B_n\backslash D_n$, and $\rho$ commutes with $t_2,\ldots,\,t_r$
(because for $2\leq i\leq r$, the supports of $\rho$ and of $t_i$  are
disjoint), we deduce that
$t_{(2\cdot\pi,\emptyset)}^-=\rho t_{(2\cdot\pi,\emptyset)}^-\rho^{-1}$. Hence,
$t_{(2\cdot\pi,\emptyset)}^{\pm}$ are representatives of the two split
classes of $D_n$ labeled by $(2\cdot\pi,\emptyset)$.
%

For every
$2$-multipartition $(\mu_1,\mu_2)$ of $n$, one has
$\alpha\otimes\theta_{(\mu_1,\mu_2)}=\theta_{(\mu_2,\mu_1)}$. By
Clifford theory, if $\mu_1\neq\mu_2$, then
$\chi_{\mu_1,\mu_2}=\Res^W_{W'}(\theta_{\mu_1,\mu_2})=
\Res^W_{W'}(\theta_{\mu_2,\mu_1})$ is irreducible. If $\mu=\mu_1=\mu_2$,
then $\Res_{W'}^W(\theta_{\mu,\mu})$ splits into two irreducible
characters $\chi_{\mu,\mu}^+$ and $\chi_{\mu,\mu}^-$ of $W'$, which
we can label so that (see~\cite[Theorem
5.1]{Pfeiffer})
\begin{equation}
\chi_{\mu,\mu}^{\epsilon}\left(
t_{(2\cdot\pi,\emptyset)}^{\delta}\right)=
\frac{1}{2}\left(\theta_{(\mu,\mu)}(t^{\delta}_{(2\cdot\pi,\emptyset)})+\epsilon\delta
2^{\ell(\pi)}\carsym_{\mu}(\pi)\right),
\label{eq:formulesplitDn}
\end{equation}
where $\delta,\epsilon\in\{-1,1\}$ and $\carsym_{\mu}$ is the character
of the symmetric group $\sym_{n/2}$ corresponding to $\mu$.

The $p$-blocks of $W'$ can be described as
follows. Let $B_{\gamma_1,\gamma_2}^{(b_1,b_2)}$ 
be a $p$-block of $W$ labeled by
the $p$-cores $\gamma_1$ and $\gamma_2$ and with $p$-weight
$(b_1,b_2)$; 
see~\S\ref{subsec:wreath}.  If $(b_1,b_2)\neq (0,0)$ or
$\gamma_1\neq\gamma_2$, then $B_{\gamma_1,\gamma_2}^{(b_1,b_2)}$ 
contains characters
that are not self-conjugate. By~\cite[Theorem 9.2]{Navarro},
$B_{\gamma_1,\gamma_2}^{(b_1,b_2)}$ covers a unique $p$-block
$b_{\gamma_1,\gamma_2}^{(b_1,b_2)}$ of $W'$. 
Furthermore, when $\gamma_1\neq\gamma_2$ or $b_1\neq b_2$,
$B_{\gamma_1,\gamma_2}^{(b_1,b_2)}$ and
$B_{\gamma_2,\gamma_1}^{(b_2,b_1)}$ contain no
self-conjugate character, and
$b_{\gamma_1,\gamma_2}^{(b_1,b_2)}=b_{\gamma_2,\gamma_1}^{(b_2,b_1)}$ 
consists of the
restrictions to $W'$ of the irreducible characters lying in
$B_{\gamma_1,\gamma_2}^{(b_1,b_2)}$ and
$B_{\gamma_2,\gamma_1}^{(b_2,b_1)}$.  If $(b_1,b_2)=(0,0)$, then
$B_{\gamma_1,\gamma_2}^{(0,0)}=\{\theta_{(\gamma_1,\gamma_2)}\}$ has defect
zero. If $\gamma:=\gamma_1=\gamma_2$, then
$b_{\gamma}^+=\{\chi_{\gamma,\gamma}^+\}$ and
$b_{\gamma}^-=\{\chi_{\gamma,\gamma}^-\}$ are two distinct $p$-blocks of
$W'$ with defect zero, except when $n=0$. In this last case, $W=W'=\{1\}$
and $\theta_{(\emptyset,\emptyset)}=\chi_{(\emptyset,\emptyset)}^+=
\chi_{(\emptyset,\emptyset)}^-=1_{\{1\}}$.

\begin{theorem}Assume $p$ is odd. Let $W_1'$ and $W_2'$ be Coxeter groups of
type $D$.
Let $b_{\gamma,\gamma}^{(b,b)}$ and $b_{\gamma',\gamma'}^{(b,b)}$ 
be $p$-blocks of $W_1'$
and $W_2'$ with the same $p$-weight $(b,b)$. Then the isometry defined by
$$I(\chi_{\mu_1,\mu_2})=\left(\prod_{i=1}^2\delta_p(\mu_i)\delta_p(\Psi(\mu_i))\right)\chi_{\Psi(\mu_1),\Psi(\mu_2)}\quad
\textrm{and}\quad
I(\chi_{\mu,\mu}^{\epsilon})=\chi_{\Psi(\mu),\Psi(\mu)}^{\epsilon\delta_p(\mu)\delta_p(\Psi(\mu))},$$
where the notation is as above, is a Brou\'e perfect isometry between
$b_{\gamma,\gamma}^{(b,b)}$ and $b_{\gamma',\gamma'}^{(b,b)}$.
\label{theo:broueDnconj}
\end{theorem}

\begin{proof}
Assume that $W_1'$ and $W_2'$ are of type $D_n$ and $D_m$, respectively.
We denote by $S$ and $C$ the intersections of $W_1'$ with the sets $S$
and $C$ defined in the proof of Theorem~\ref{theo:couronne}, and we
write $\Omega$ (respectively $\Omega_0$) for the set of bipartitions
$\pp=(\pi_1,\pi_2)$ with $\pi_1,\,\pi_2\in p\mathcal P $ and
$\ell(\pi_2)$ even, such that $|\pi_1|+|\pi_2|\leq n$ (respectively
$|\pi_1|+|\pi_2|\leq 2pb$). 
Denote by $\Lambda$ the $W'_1$-classes of elements of $S$. Note that
$\Omega$ is the set of cycle type of the classes in $\Lambda$.
Furthermore, we write $\Lambda_0$ for the set of classes in $\Lambda$
whose cycle type belong to $\Omega_0$.
When $n\neq 2pb$, the set $\Omega_0$ labels $\Lambda_0$. 
Otherwise, there are in $\Omega_0$ elements $\pp$ that parametrize two
$W_1'$-classes denoted by $\pp^+$ and $\pp^-$. 
In this case, 
$\pp=(2\cdot\pi,\emptyset)\in\Omega_0$ for some partition $\pi$ of $n/2$,
and we denote by $t_{\pp}^+$ and $t_{\pp}^-$ representatives
for the split classes as in Equation~(\ref{eq:repchoixdn}). The two
corresponding classes are denoted by $\pp^+$ and $\pp^-$.
So, when $n=2pb$, the elements of $\Lambda_0$ are denoted by
$\widehat{\pp}$ with $\widehat{\pp}=\pp$ when $\pp\in\Omega_0$ labels one
class, and $\widehat{\pp}\in\{\pp^{+},\pp^-\}$ otherwise.
We also will write $t_{\pi^+}=t_{\pi}^+$ and $t_{\pi^-}=t_{\pi}^-$.
Finally, for
$\pp\in\Omega_0$, we define $G_{t_{\widehat{\pp}}}=D_{n-|\pi_1|-|\pi_2|}$.

We then take $t_{\pp}$ as in Equation~(\ref{eq:prodcourel}) for a
representative of the class of $S$ labeled by $\pp\in\Omega_0$ 
whenever $\widehat\pp=\pp$.

Assume that $n$ is even. 
For any partition $\mu$ of $n/2$, we write
$\Delta_{\mu}=\chi_{\mu,\mu}^+-\chi_{\mu,\mu}^-$.
Let $1\leq k< n$, and $t=( (1,\ldots,1);\sigma)\in D_n$, where $\sigma=(n-k+1\cdots n)$.
We will prove that
\begin{equation}
\label{eq:deltaMn}
\Delta_{\mu}(tx)=2\sum_{\nu\in
M_{k}(\mu)}(-1)^{L(c_{\nu}^{\mu})}\Delta_{\nu}(x),
\end{equation}
for all $x\in D_{n-k}$. Note that $tx$ lies
in a split class of $D_n$ if and only if $x$ lies in a split class of
$D_{n-k}$. So, to prove Equation~(\ref{eq:deltaMn}), we
can assume that $x$ lies in a split class of $D_{n-k}$.
Suppose that $tx$ is $D_n$-conjugate by $g\in D_n$
to $ty$
with $y\in D_{n-k}$ (in particular, $y$ lies in a split
class of $D_{n-k}$). Then $x$ and $y$ have the same cycle
type, so they are $B_{n-k}$-conjugate, say by $g_0\in
B_{n-k}$. Furthermore, $g_0$ and $t$ commute
(because their have disjoint supports). It follows that
${}^{g_0}(t x)=ty$, and the set of elements
that conjugate $tx$ and $ty$ is
$g_0\Cen_{B_n}(tx)$.  
Furthermore, since $tx$ lies in a split class, one has
$\Cen_{B_n}(t x)=\Cen_{D_n}(t x)$. 
So, there is $h\in \Cen_{D_n}(t x)$ such that $g=g_0h$.
This proves that $g_0\in D_{n-k}$.
Hence $x$ and $y$ are $D_{n-k}$-conjugate, and
Equation~(\ref{eq:deltaMn}) now follows from
Equation~(\ref{eq:formulesplitDn}).

Furthermore, assume that $k=n$. Then Equation~(\ref{eq:formulesplitDn})
gives 
$\Delta_{\mu}(t_{(k)}^{\delta})=\delta2\chi_{\mu}((k))$ for
$\delta\in\{+,-\}$. 
If $\mu$ is not a hook, then
$M_k(\mu)=\emptyset$. Otherwise, $M_k(\mu)=\{\emptyset\}$. 
Setting $\Delta_{\emptyset}=1$, and using the
Murnaghan-Nakayama rule for the symmetric group, we obtain
 \begin{equation}
 \label{eq:deltaMnempty}
 \Delta_{\mu}(t_{\sigma}^{\delta})=2\sum_{\nu\in
 M_k(\mu)}(-1)^{L(c_{\nu}^{\mu})}\Delta_{\nu}(1).
 \end{equation}

For $\pp\in\Omega_0$, we define
$r^{\widehat{\pp}}(\Delta_{\mu})(x)=\Delta_{\mu}(t_{\widehat{\pp}}x)$ for all
$x\in G_{t_{\widehat{\pp}}}$. Applying iteratively
Equation~(\ref{eq:deltaMn}) and Equation~(\ref{eq:deltaMnempty})
to the parts of $t_{\widehat\pp}$, we obtain 
\begin{equation}
\label{eq:rpiDelta}
r^{\widehat{\pp}}(\Delta_{\mu})=2^{\ell(\pi)}\sum_{\nu}a(\mu,\nu)\Delta_{\nu},
\end{equation} where the coefficients are those appearing in
Equation~(\ref{eq:MNSnitere}).

Now, for $\mm=(\mu_1,\mu_2)\in\mathcal E_{(\gamma,\gamma),(b,b)}$ with
$\mu_1\neq\mu_2$, we define $r^{\widehat{\pp}}(\chi_{\mu_1,\mu_2})$ to
be the restriction to $W_1'$ of $r^{\pp}(\theta_{(\mu_1,\mu_2)})$, where
$r^{\pp}$ is the map defined in the proof of
Theorem~\ref{theo:couronne}.  For $\mm=(\mu,\mu)\in\mathcal
E_{(\gamma,\gamma),(b,b)}$, define
\begin{equation}
r^{\widehat{\pp}}(\chi_{\mu,\mu}^{\epsilon})=\frac{1}{2}\left(\Res_{W_1'}^{W_1}(r^{\pp}(\theta_{(\mu,\mu)})+\epsilon
r^{\widehat{\pp}}(\Delta_{\mu})\right).
\label{eq:defepsilonDn}
\end{equation}
It is then straightforward to show that, if
$b_{\gamma,\gamma}(n-|\pi_1|-|\pi_2|)$ denotes the union of the $p$-blocks of
$G_{n-|\pi_1|-|\pi_2|}$ with $p$-core $(\gamma,\gamma)$ and $p$-weights
$(b_1,b_2)$ such that $0\leq b_i\leq b$ and $b_1+b_2=|\pi_1|+|\pi_2|$, 
then the map
$r^{\widehat{\pp}}:\C\Irr(b_{\gamma,\gamma}^{(b,b)})\rightarrow
\C\Irr(b_{\gamma,\gamma}(n-|\pi_1|-|\pi_2|))$ defines an MN-structure for $W_1'$ with
respect to the set of $p$-regular elements and
$b_{\gamma,\gamma}^{(b,b)}$.
Similarly, we define an MN-structure for $W_2'$ with respect to the set
of $p$-regular elements of $W_2'$ and
$b_{\gamma',\gamma'}^{(b,b)}$.
As we showed in the proof of Theorem~\ref{theo:couronne}, if
$\mu_1\neq\mu_2$ and $I$ is defined on
$\Irr(b_{\gamma,\gamma}(n-|\pi_1|-|\pi_2|))$ by the same formula, then we have 
\begin{equation}
I\left(r^{\widehat\pp}(\chi_{\mu_1,\mu_2})\right)=
r^{\widehat{\pp}}\left(I(\chi_{\mu_1,\mu_2})\right).
\label{eq:MNDndifferent}
\end{equation}
For any $\mu\neq\emptyset$ with
$p$-core $\gamma$, one has
$$\Delta_{\mu}=\delta_p(\mu)\delta_p(\Psi(\mu))\left(\chi_{\mu,\mu}^{\delta_p(\mu)\delta_p(\Psi(\mu))}-\chi_{\mu,\mu}^{-\delta_p(\mu)\delta_p(\Psi(\mu))}\right).$$
In particular, 
\begin{equation}
I(\Delta_{\mu})=\delta_p(\mu)\delta_p(\Psi(\mu))\Delta_{\Psi(\mu)}.
\label{eq:IdeltaDn}
\end{equation}
Therefore, we deduce from the fact that
$I(\theta_{(\mu,\mu)})=\theta_{\Psi(\mu),\Psi(\mu)}$ and
Equations~(\ref{eq:defepsilonDn}), (\ref{eq:rpiDelta}), (\ref{eq:fcrochet})
and~(\ref{eq:lienjambef}) that
\begin{equation}
I\left(r^{\widehat{\pp}}(\chi_{\mu,\mu}^{\epsilon})\right)
=r^{\widehat{\pp}}\left(I(\chi_{\mu,\mu}^{\epsilon})\right).
\label{eq:MNDnmeme}
\end{equation}

Assume first that $|\Lambda_0|=|\Lambda'_0|$. Then
Equations~(\ref{eq:MNDndifferent}) and~(\ref{eq:MNDnmeme}) hold and we
derive from Theorem~\ref{th:broue} (see also the note in the proof of
Theorem~\ref{theo:mainAn}) that $I$ is a Brou\'e perfect
isometry.

Assume, on the other hand, that $|\Lambda_0|>|\Lambda_0'|$. In particular, $n$ is divisible
by $2p$, $\gamma=\emptyset$, and $\Lambda'_0=\Omega_0$. Let $\pp\in\Omega_0$ be such that 
$\widehat{\pp}=\pp$. If we define $I_{\pp}$ on $G_{t_{\widehat{\pp}}}$ 
in the same way as $I$, then by Equations~(\ref{eq:MNDndifferent})
and~(\ref{eq:MNDnmeme}), we have $I_{\widehat{\pp}}\circ
r^{\widehat{\pp}}=r^{\widehat{\pp}}\circ I$.
Let now $\pp=(2\cdot\pi,\emptyset)$ be such that $2|\pi|=n$. Then
$\widehat{\pp}\in\{\pp^+,\pp^-\}$, and $G_{t_{\pp}^+}$ and
$G_{t_{\pp}^-}$ are two copies of the trivial group. We set
$\Irr(G_{t_{\pp}^+})=\{1_{\pp^+}\}$ and
$\Irr(G_{t_{\pp}^-})=\{1_{\pp^-}\}$.
Furthermore,
$\Irr(b_{\gamma',\gamma'}(m-n))=\{\chi_{\gamma',\gamma'}^+,
\chi_{\gamma',\gamma'}^-\}$.
We define
$I_{\pp}:\C\Irr(G_{t_{\pp}^+})\oplus\C\Irr(G_{t_{\pp}^-})\rightarrow \C
\Irr(b_{\gamma',\gamma'}(m-n))$ by setting
$I_{\pp}(1_{\pp^\delta})=\chi_{\gamma',\gamma'}^\delta$. 
Note that
\begin{equation}
r^{\pp^{\delta}}(\chi_{\mu,\mu}^{\epsilon})=\chi_{\mu,\mu}^{\epsilon}(t_{\pp}^{\delta})1_{\pp^{\delta}}=\frac{1}{2}\left(
\theta_{(\mu,\mu)}(t_{\pp}^{\delta})+\epsilon\delta
2^{\ell(\pi)}a(\mu,\gamma)\right)1_{\pp^{\delta}}.
\label{eq:valDnpourri}
\end{equation}
Moreover, by Equation~(\ref{eq:rpiDelta}), one has
\begin{eqnarray*}
r^{\pp}\left(\delta_p(\mu)\delta_p(\Psi(\mu))\Delta_{\Psi(\mu)}\right)&=&2^{\ell(\pi)}\delta_p(\mu)\delta_p(\Psi(\mu))a(\Psi(\mu),\gamma')\Delta_{\gamma'}\\
&=&2^{\ell(\pi)}a(\mu,\gamma)\Delta_{\gamma'},
\end{eqnarray*}
because $\delta_p(\gamma)=\delta_p(\gamma')=1$. Write
$\epsilon_{\mu}=\delta_p(\mu)\delta_p(\Psi(\mu))$, and note that
$r^{\pp}(\theta_{\mu,\mu})=\theta_{\mu,\mu}(\pp)1_{\{1\}}$,
$I(r^{\pp}(\theta_{\mu,\mu}))=r^{\pp}(\theta_{\Psi(\mu),\Psi(\mu)})$,
and $r^{\pp}\circ\Res_{W_2'}^{W_2}=\Res_{W_2'}^{W_2}\circ r^{\pp}$. So we
obtain
\begin{eqnarray*}
r^{\pp}(I(\chi_{\mu,\mu}^{\epsilon}))&=&r^{\pp}\left(\chi_{\Psi(\mu),\Psi(\mu)}^{\epsilon_{\mu}\epsilon}\right)\\
&=&\frac{1}{2}\left(\Res_{W_2'}^{W_2}(r^{\pp}(\theta_{\Psi(\mu),\Psi(\mu)}))+\epsilon_{\mu}\epsilon
r^{\pp}(\Delta_{\Psi(\mu)})\right)\\
&=&\frac{1}{2}
\left(\theta_{(\mu,\mu)}(t_{\pp})+
\epsilon 2^{\ell(\pi)}a(\mu,\gamma)\right)\chi_{\gamma',\gamma'}^+
\\
&&\quad
+\frac{1}{2}\left(\theta_{(\mu,\mu)}(t_{\pp})-
\epsilon 2^{\ell(\pi)}a(\mu,\gamma)\right)\chi_{\gamma',\gamma'}^-\\
 &=&\chi_{\mu,\mu}^{\epsilon}(t_{\pp}^+)\chi_{\gamma',\gamma'}^+
+\chi_{\mu,\mu}^{\epsilon}(t_{\pp}^-)\chi_{\gamma',\gamma'}^-\\
&=&I_{\pp}\left(r^{\pp^+}(\chi_{\mu,\mu}^{\epsilon})+r^{\pp^-}(\chi_{\mu,\mu}^{\epsilon})\right).
\label{eq:valDmpourri}
\end{eqnarray*}
Now, assume that $\mu_1\neq\mu_2$. Note that
$r^{\pp^\pm}(\chi_{\mu_1,\mu_2})=\theta_{\mu_1,\mu_2}(\pp)1_{\pp^\pm}$.
Thus,
\begin{eqnarray*}
I_{\pp}(r^{\pp^+}(\chi_{\mu_1,\mu_2})+r^{\pp^-}(\chi_{\mu_1,\mu_2}))&=&\theta_{\mu_1,\mu_2}(\pp)(\chi_{\gamma',\gamma'}^++\chi_{\gamma',\gamma'}^-)\\
&=&\theta_{\mu_1,\mu_2}(\pp)\Res_{W_2'}^{W_2}(\theta_{\gamma',\gamma'})\\
&=&\Res_{W_2'}^{W_2}(I(\theta_{\mu_1,\mu_2}(\pp)1_{\{1\}}))\\
&=&\Res_{W_2'}^{W_2}(I(r^{\pp}(\theta_{\mu_1,\mu_2})))\\
&=&\Res_{W_2'}^{W_2}(r^{\pp}(I(\theta_{\mu_1,\mu_2})))\\
&=&r^{\pp}(I(\chi_{\mu_1,\mu_2}))).
\end{eqnarray*}
Hence, we have $$r^{\pp}\circ
I=I_{\pp}\circ(r^{\pp^+}+r^{\pp^-}),$$
and we conclude as in the proof of Theorem~\ref{theo:mainAn}.
\end{proof}

\begin{theorem}
Assume $p$ is odd. Let $W_1'$ and $W_2'$ be Coxeter groups of type $D$.
Assume that $\gamma_1\neq\gamma_2$ and $\gamma'_1\neq\gamma'_2$, or
$(\gamma_1,\gamma_2)=(\gamma_1',\gamma_2')$ and $b_1\neq b_2$.  If the
$p$-blocks $b_{\gamma_1,\gamma_2}^{(b_1,b_2)}$ and
$b_{\gamma'_1,\gamma'_2}^{(b_1,b_2)}$ have the same $p$-weight
$(b_1,b_2)$, then they are perfectly isometric in the sense of Brou\'e.
\label{theo:broueDnpasconj}
\end{theorem}

\begin{proof}
The isometry is the restriction to
$\Irr(b_{\gamma_1,\gamma_2}^{(b_1,b_2)})$ of that
of Corollary~\ref{cor:broueBn}.
\end{proof}

\subsection{Isometries between alternating groups and natural subgroups}
\label{subsec:fh}

It would be interesting to give an analogue of Osima's perfect
isometry between $p$-blocks of the alternating groups and the
``alternating'' subgroup of $\Z_p\wr\sym_w$.
But such perfect isometries do not exist, as we can show
in the following example. 

\begin{example}
\label{ex:osima}
Consider the principal $3$-block $b$ of $\Alt_6$. It contains
$6$ irreducible characters. Note that $b$ is covered by the principal
$3$-block $B$ of $\sym_6$ (which has $3$-weight $2$ and contains $9$
irreducible characters).
Let $G=\Z_3\wr \sym_2$. Then $G$ has $9$ irreducible characters and by  
Theorem~\ref{theo:osima}, $B$ and $G$ are perfectly isometric. Now,
viewing $G$ as a subgroup of $\sym_6$, we can restrict the sign character
$\varepsilon:\sym_6\rightarrow \{-1,1\}$ to a linear character (also
denoted by $\varepsilon$) of $G$, whose kernel is the base
group $H=\Z_3^2$ of $G$. 
Define the regular elements of $H$ to be the elements
with cycle structure $(\pi_1,\pi_2,\pi_3)$ and $\pi_1=\emptyset$. 
These elements are the products of $2$ disjoint $3$-cycles contained in $H$,
when $H$ is viewed as a subgroup of $\sym_6$, and there are $4$ such
elements. Now, a straightforward computation gives that $\langle
\operatorname{res}_{\operatorname{reg}}(\chi),\operatorname{res}_{\operatorname{reg}}(1_H)\rangle\in\{\-2/9,1/9,4/9\}$
for any $\chi\in\Irr(H)$. So, we conclude by Remark~\ref{rk:kor} that
$\Irr(H)$
forms a reg-block, and since $\Irr(H)$ has $9$ elements, $b$ and $\Irr(H)$ are 
not perfectly isometric.
\end{example}

However, when we replace $\Z_p\wr\sym_w$ by $G_{p,w}$ (see
\S\ref{subsec:sym} for the notation), we can show that the $p$-blocks
of $\Alt_n$ are perfectly isometric with the ``alternating'' subgroup of
$G_{p,w}$.
In a way, we prove in this section an analogue of 
Osima's isometries for the alternating groups.

Throughout, we keep the notation of \S\ref{subsec:sym}, and view
$G_{p,w}$ as a subgroup of $\sym_{pw}$. Moreover, we assume that 
$p$ is odd, so that, in particular, $p^*=(p+1)/2$.
Furthermore, we view $H=\Z_p\rtimes\Z_{p-1}$ as the normalizer of some
Sylow $p$-subgroup of $\sym_p$, 
and denote by
$\varepsilon_H$ the restriction of the sign character
$\varepsilon_{\sym_p}$ to $H$. 
Note that only the irreducible character of degree $p-1$ of $H$
is $\varepsilon_H$-stable. So we choose the labeling of 
$\Irr(H)=\{\psi_1,\ldots,\psi_p\}$
so that $\psi_1=\varepsilon_H$, $\psi_p=1_H$, and 
$\psi_i=\psi_{p+1-i}\otimes\varepsilon_H$ for
any $1\leq i\leq p$ (in particular, $\psi_{p^*}(1)=p-1$). Recall that
$\Irr(G_{p,w})$ is labeled by $\mathcal{MP}_{p,w}$, and, with the above
choices, for every $\mm=(\mu_1,\ldots,\mu_p)\in\mathcal{MP}_{p,w}$,
 we have (see~\cite[Proposition 4D]{fongharris}) 
\begin{equation}
\label{eq:etoile}
\varepsilon\theta_{\mm}=\theta_{\mm^*},
\end{equation}
where $\varepsilon$ again denotes the restriction of the sign character of
$\sym_{pw}$ to $G_{w,p}$, and $\mm^*=(\mu_p^*,\ldots,\mu_1^*)$ is as in
Equation~(\ref{eq:conjquotient}).
Define the ``alternating'' subgroup of $G_{p,w}$ by setting
$$H_{p,w}=\ker(\varepsilon:G_{p,w}\rightarrow\{-1,1\}).$$

Consider  the set of partitions
$\mathcal E$ (respectively $\mathcal O\mathcal D$) all of whose parts have even length
(respectively whose parts are distinct and of odd length).
We recall that (see
for example~\cite[Lemma 4E]{fongharris}) the set
\begin{equation}
\label{eq:defT}
\mathcal T=\{(\pi_1,\ldots,\pi_p)\in\mathcal
MP_{p,w}\,|\,\pi_{2i}=\emptyset,\,\pi_{2i+1}\in \mathcal E,
\pi_p\in\mathcal O\mathcal D\}
\end{equation}
labels the set of splitting
classes of $G_{w,p}$ with respect to $H_{w,p}$.
We will now give representatives for these classes. Let
$\pp=(\pi_1,\ldots,\pi_p)\in\mathcal
T$. For $1\leq i\leq p$, write
$\pi_i=(\pi_{i,1},\ldots,\pi_{i,\ell(\pi_i)})$, and assume that 
there is some integer $1\leq
r_i\leq \ell(\pi_i)$ such that $\pi_{i,j}$ is prime to $p$ for all
$j<r_i$ and $\pi_{i,j}$ is divisible by $p$ for $j\geq r_i$.
Let $t_{\pp}$ be the element of $G_{p,w}$ obtained in the same way as in
Equation~(\ref{eq:prodcourel}). Let $1\leq i\leq p$ be such that
$\pi_i\neq \emptyset$. Then with the notation of
Equation~(\ref{eq:prodcourel}), $t_{i1}\in G_{p,K}$, where 
$K$ is the support of $t_{i1}$ and $G_{p,K}=(\Z_p\rtimes\Z_{p-1})\wr
\sym_K$. In particular, viewed as an element of $G_{p,K}$,
the cyclic structure of $t_{i1}$ is
$(\emptyset,\ldots,\emptyset,(\pi_{i,1}),\emptyset,\ldots,\emptyset)$.
Since $\pp\in\mathcal T$, either $i<p$ is odd and $\pi_i\in\mathcal E$,
and so also $(\pi_{i,1})$, or $i=p$ and $(\pi_{i,1})\in
\mathcal{OD}_{\pi_{i,1}}$. Hence, $t_{i1}$ lies in a split class of
$G_{p,K}$. Let $\rho_K\in G_{p,K}\backslash H_{p,K}$. Set
$t_{i1}^+=t_{i1}$ and $t_{i1}^-=\rho_K t_{i1}^+\rho_{K}^{-1}$. 
Write $m$ for the minimum integer such that $\pi_m\neq \emptyset$. 
Using the notation of
Equation~(\ref{eq:prodcourel}), we define
\begin{equation}
\label{eq:labelclassHn}
r_{\pp}=\left(\prod_{i\neq m
}\prod_{j=1}^{\ell(\pi)}t_{ij}\right)\prod_{j\neq
1}t_{mj}\quad\textrm{and}\quad
t_{\pp}^{\pm}=r_{\pp}t_{m1}^{\pm}.
\end{equation}
Since $\rho_K\notin H_{p,w}$ and the supports of $\rho_K$ and $r_{\pp}$
are disjoint, the elements $t_{\pp}^{+}$ and $t_{\pp}^-$ are representatives for the
two split classes of $H_{p,w}$ labeled by $\pp$.

Write $\mathcal S$ for the set of $\mm\in\mathcal{MP}_{p,w}$ such that
$\mm^*=\mm$.
Now, following~\cite{fongharris}, we define an explicit bijection 
$\mathbf{a}:\mathcal S\rightarrow \mathcal T$ as follows.
Let
$\mm=(\mu_1,\ldots,\mu_p)\in \mathcal
S$. Then $\mu_{p+1-i}^*=\mu_{i}$ for all $1\leq i\leq p$. In particular,
$\mu_{p^*}=\mu_{p^*}^*$. Write
$\mu_i=\prod_jj^{p_{ij}}$ for $1\leq i <p^*$, and
recall the definition of $\mathbf{a}(\mm):=(\pi_1,\ldots,\pi_p)\in\mathcal{T}$
from~\cite{fongharris}
by setting $\pi_p=a(\mu_{p^*})$, 
$\pi_{2i-1}=\prod_{j}(2j)^{p_{ij}}$, and $\pi_{2i}=\emptyset$, where $a$ is
the map defined in Equation~(\ref{eq:defa}).
Then $\mathbf{a}$ is
a bijection. 
Indeed, if for $(\pi_1,\ldots,\pi_p)\in\mathcal T$, we define
$\mm=(\mu_1,\ldots,\mu_p)$ by setting
$\mu_{p^*}=a^{-1}(\pi_p)$, $\mu_i=\prod_j j^{p_{ij}}$ for  $1\leq i< p^*$, 
where $\pi_{2i-1}=\prod_j(2j)^{p_{ij}}$, and
$\mu_i=\mu_{p+1-i}^*$ for $p^*<i\leq p$, then the map
$(\pi_1,\ldots,\pi_p)\mapsto(\mu_1,\ldots,\mu_p)$ is the inverse map of
$\mathbf a$. 

\begin{lemma}
The conjugacy class of $G_{p,w}$ labeled by
$(\emptyset,\ldots,\emptyset,1^{w-k},\beta)\in\mathcal{MP}_{p,w}$ (for
$k\leq w$) lies in
$H_{p,w}$ if and only if $\beta$ has an even number of even parts.
\label{lem:classHnp}
\end{lemma}

\begin{proof}
Because of~\cite[Equation (4.1)]{fongharris}, for every
$( h_1,\ldots,h_w;\sigma)\in G_{p,w}$ with $h_i\in H$ and
$\sigma\in\sym_w$, we have
\begin{equation}
\label{eq:signature}
\varepsilon(
h_1,\ldots,h_w;\sigma)=\varepsilon(\sigma)\prod_{i=1}^w\varepsilon_{H}(h_i).
\end{equation}
Write $\beta=(\beta_1,\ldots,\beta_r)$, and set $
\sigma_{\beta}=\sigma_1\cdots\sigma_k$, where $\sigma_i$ is a
cycle of length $|\beta_i|$. 
Let $\{j_1,\ldots,j_{|\beta_i|}\}$ be the
support of $\sigma_i$. Define $h_{j_1}=\omega$ (the element $\omega$ is
as in \S\ref{subsec:sym}, \emph{i.e.} a generator of the Sylow $p$-subgroup 
of the base group of the wreath product $G_{p,w}$) and $h_{j_l}=1$ for
$2\leq l\leq |\beta_i|$. If $l$ doesn't belong to the support of any $\beta_i$
then put $h_l=1$. Thus, the element
$x_{\beta}=(h_1,\ldots,h_w;\sigma_{\beta})$ is a representative for the
class of $G_{p,w}$ labeled by $(\emptyset,\ldots,\emptyset,\beta)$. By
Equation~(\ref{eq:signature}),
$\varepsilon(h_1,\ldots,h_w;\sigma_{\beta})=1$ if and only if
$\varepsilon(\sigma_{\beta})=1$ (because $\varepsilon_H(\omega)=1$), as
required.
\end{proof}

By Equation~(\ref{eq:etoile}) and Clifford Theory,  
if $\mm\notin \mathcal S$, then the restriction
$\vartheta_{\mm}=\Res_{H_{p,w}}^{G_{p,w}}(\theta_{\mm})=\Res_{H_{p,w}}^{G_{p,w}}(\theta_{\mm^*})$
is irreducible. Otherwise, the restriction of 
$\theta_{\mm}$ splits into a sum of two irreducible characters of
$H_{p,w}$, denoted $\vartheta_{\mm}^{+}$ and $\vartheta_{\mm}^{-}$. In
the last case, such a $\theta_{\mm}$ is called a split irreducible
character of $G_{p,w}$.

Let $\mm=(\mu_1,\ldots,\mu_p)\in \mathcal S$.  In order to distinguish
$\vartheta_{\mm}^{+}$ and $\vartheta_{\mm}^{-}$, we need to introduce
some notation.  We associate to $\mm$ two multipartitions
$\mm'\in\mathcal{MP}_{p,w-|\mu_{p^*}|}$ and
$\mm''\in\mathcal{MP}_{p,|\mu_{p^*}|}$ by setting
$$\mm'=(\mu_1,\ldots,\mu_{(p-1)/2},\emptyset,\mu_{(p+3)/2},\ldots,\mu_p)\
\textrm{
and }
\mm''=(\emptyset,\ldots,\emptyset,\mu_{p^*},\emptyset,\ldots,\emptyset).$$

Moreover, to $\mm'$ and $\mm''$, we associate subgroups as follows. 
Write $E_{\mm'}=\{1,\ldots,n-|\mu_{p^*}|\}$ and
$E_{\mm''}=\{n-|\mu_{p^*}|+1,\ldots,n\}$, and define $G_{\mm'}=H\wr\sym(E_{\mm'})$ and 
$G_{\mm''}=H\wr\sym(E_{\mm''})$. Note that $\mm'$ and $\mm''$ are
self-conjugate.

In particular, by \S\ref{subsec:not}, $\mm'$ and
$\mm''$ label split irreducible characters $\theta_{\mm'}$ and
$\theta_{\mm''}$ of $G_{\mm'}$ and $G_{\mm''}$ respectively. 

Since $\mm''$ is self-conjugate, $\mathbf a(\mm'')$ is a splitting class
of $G_{\mm''}$, and thus labels two classes $\mathbf a(\mm'')^{\pm}$ of
$H_{\mm''}=\ker(\varepsilon_{G_{\mm''}})$.
Now, we make the same choices for 
the labeling for the irreducible characters $\vartheta_{\mm''}^{\pm}$ 
and for the classes $\mathbf{a}(\mm'')^{\pm}$ of $H_{\mm''}$
as in~\cite[Proposition 4F]{fongharris}, so that yields
\begin{equation}
\left(\vartheta_{\mm''}^+-\vartheta_{\mm''}^-\right)(g)=
\left\{
\begin{array}{l}
\epsilon
(\sqrt{\epsilon_pp})^d\sqrt{\epsilon_{\mu_{p^*}}\operatorname{ph}(\mu_{p^*})}\quad\textrm{if
}g\in \mathbf{a}(\mm'')^{\epsilon},\\
0\quad \textrm{otherwise},
\end{array}
\right.
\label{eq:fong4F}
\end{equation}
where $\epsilon\in\{\pm 1\}$, $\epsilon_p=(-1)^{(p-1)/2}$,  $d$ is the
number of parts of $a(\mu_{p^*})$,
$\epsilon_{\mu_{p^*}}=(-1)^{(|\mu_{p^*}|-d)/2}$, and
$\operatorname{ph}(\mu_{p^*})$ denotes the product of the lengths of the parts
of $a(\mu_{p^*})$. 

Furthermore, fix any labeling for the irreducible characters
$\vartheta_{\mm'}^{\pm}$ of $H_{\mm'}=\ker(\varepsilon_{G_{\mm'}})$.
Labelings for $\mm'$ and $\mm''$
being fixed as above, we can assume that the characters
$\vartheta_{\mm}^{\pm}$ are parametrized
as in~\cite[Proposition 4H(ii)]{fongharris}, and we always make this choice
in the following. We can now show the following crucial result.

\begin{lemma}Let $c$ be a cycle of odd length $k\leq w$.
Let $x=(t;\sigma)\in G_{p,w}$ have cycle structure
$(\emptyset,\ldots,\emptyset,1^{w-k},(k))$, and be such that $\sigma=(w-k+1,\ldots,w)$. 
Let $\mm=(\mu_1,\ldots,\mu_p)\in\mathcal{MP}_{p,w}$ be such that $\mm=\mm^*$.
If $c$ is a cycle of $a(\mu_{p^*})$, then
for any $g\in H_{p,w-k}$, we have
$$\left(\vartheta_{\mm}^+-\vartheta_{\mm}^-\right)(xg)=\sqrt{(-1)^{(pk-1)/2}pk}\
\left(\vartheta_{\mm_c}^+-\vartheta_{\mm_c}^-\right)(g),$$
where $(\mm_c)_i = \mu_i$ if $i\neq p^*$, and $(\mm_c)_{p^*}$ is
obtained from $\mu_{p^*}$ by removing the diagonal hook of length $k$.
\label{mn-courAn}
\end{lemma}

\begin{proof}
By Lemma~\ref{lem:classHnp}, one has $x\in H_{p,w}$. Furthermore,
$\vartheta_{\mm_c}^{\pm}$ are
irreducible characters of $H_{p,w-k}$.
Write $\mm'$ and $\mm''$ for the multipartitions associated to $\mm$ as
above. By construction, we have $\mm'_c=\mm'$, and $\mm''_c$ is obtained
from $\mm''$ by removing the diagonal hook of length $k$ (this is
possible because $c$ is a cycle of $a(\mu_{p^*})$) at the
$p^*$-coordinate. 

Let $g\in H_{p,w-k}$. Then by \cite[(i) of Proposition 4H]{fongharris},
either $\left(\vartheta_{\mm}^+-\vartheta_{\mm}^-\right)(xg)=0=
\left(\vartheta_{\mm_c}^+-\vartheta_{\mm_c}^-\right)(g)$ (and the claim
is true), or there are $y\in H_{\mm'_c}$ and $z\in
H_{\mm''_c}$ such that $g=yz=zy$ and $\mathfrak s(y)_p=\emptyset$. Since
$H_{\mm''_c}\subseteq H_{\mm''}$, the elements $x$ and $z$ lie in 
$ H_{\mm''}$. On the other hand, $x$ commutes with $z$ and with $y$ (because
these elements have disjoint supports).
Hence, \cite[Proposition
4H]{fongharris} implies that 
\begin{equation}
\label{eq:fong4H}
\left(\vartheta_{\mm}^+-\vartheta_{\mm}^-\right)(xg)=(\vartheta_{\mm'}^+-\vartheta_{\mm'}^-)(y)\,(\vartheta_{\mm''}^+-\vartheta_{\mm''}^-)(xz).
\end{equation}
First, suppose that $xz\in\mathbf{a}(\mm'')^\epsilon$. Without loss of
generality, in the writing of $t_{\mathbf{a}(\mm'')}$ as in
Equation~(\ref{eq:labelclassHn}), we can assume that $x=t_{p\ell(\pi_p)}$.
Hence,
$t_{\mathbf{a}(\mm'')}^{\epsilon}=xt_{\mathbf{a}(\mm''_c)}^{\epsilon}$.
A similar argument to that after Equation~(\ref{eq:deltaMn}) shows
that $xz\in \mathbf{a}(\mm'')^{\epsilon}$ if and only if
$z\in\mathbf{a}(\mm''_c)^{\epsilon}$.
%
Note that  $H_{\mm_c'}=H_{\mm'}$, 
$\vartheta_{\mm_c'}^+=\vartheta_{\mm'}^+$ and
$\vartheta_{\mm_c'}^-=\vartheta_{\mm'}^-$ (because $\mm_c'=\mm$).
Let $d$ be the number of parts of $a(\mu_{p^*})$. Then
$a( (\mm_c)_{p^*})$ has $(d-1)$ parts. Moreover,
one has
$\epsilon_{\mu_{p^*}}=(-1)^{(k-1)/2}\epsilon_{(\mm_c)_{p^*}}$ and
$\operatorname{ph}(\mu_{p^*})=k\operatorname{ph}( (\mm_c)_{p^*} )$, so
that
Equations~(\ref{eq:fong4F}) and~(\ref{eq:fong4H}) 
give
\begin{eqnarray*}
\left(\vartheta_{\mm}^+-\vartheta_{\mm}^-\right)(gx)&=&
\left(\vartheta_{\mm'}^+-\vartheta_{\mm'}^-\right)(y)\,\epsilon(\sqrt{\epsilon_pp})^d\sqrt{\epsilon_{\mu_{p^*}}\operatorname{ph}(\mu_{p^*})}\\
&=&
\left(\vartheta_{\mm_c'}^+-\vartheta_{\mm_c'}^-\right)(y)\,\sqrt{\epsilon_pp
k(-1)^{(k-1)/2}
}\epsilon(\sqrt{\epsilon_pp})^{d-1}\\
&&\cdot\sqrt{\epsilon_{(\mm_c)_{p^*}}\operatorname{ph}(
(\mm_c)_{p^*})}\\
&=&\sqrt{\epsilon_pp k(-1)^{(k-1)/2}
}
\left(\vartheta_{\mm_c'}^+-\vartheta_{\mm_c'}^-\right)(y)\,\left(\vartheta_{\mm_c''}^+-\vartheta_{\mm_c''}^-\right)(z)\\
&=&\sqrt{\epsilon_p pk(-1)^{(k-1)/2}
}
\left(\vartheta_{\mm_c}^+-\vartheta_{\mm_c}^-\right)(g).
\end{eqnarray*}
Furthermore, 
$$
(-1)^{(pk-1)/2}=\left((-1)^{(p-1)/2}\right)^{k}\,(-1)^{(k-1)/2}=\epsilon_p(-1)^{(k-1)/2},$$
because $k$ is odd. The result follows.

Now, if $xz\notin\mathbf{a}(\mm'')^\pm$, then
$z\notin\mathbf{a}(\mm_c'')^{\pm}$. We then have  $(\vartheta_{\mm''}^+-\vartheta_{\mm''}^-)(xz)=0 =
(\vartheta_{\mm_c''}^+-\vartheta_{\mm_c''}^-)(z)=0$ by Equation~(\ref{eq:fong4F}),
and 
Equation~(\ref{eq:fong4H}) gives
$$(\vartheta_{\mm}^+-\vartheta_{\mm}^-)(gx)=0=(\vartheta_{\mm_c}^+-\vartheta_{\mm_c}^-)(g).$$
This proves the result.
\end{proof}

For $\lambda\neq\lambda^*$ and $\mm\neq\mm^*$, we write
$\caralt_{\lambda}^+=\caralt_{\lambda}^-=\caralt_{\lambda}$ and
$\vartheta_{\mm}^+=\vartheta_{\mm}^-=\vartheta_{\mm}$. 
Furthermore, an element $h\in H_{p,w}$ is said regular if its cycle
structure $\mathfrak s(h)$ satisfies $\mathfrak s(h)_p=\emptyset$. 

\begin{theorem}Let $p$ be an odd prime number.
Let $\gamma$ be a self-conjugate $p$-core of $\sym_n$ of
$p$-weight $w>0$. Denote by $b_{\gamma}$ the corresponding $p$-block of
$\Alt_n$. Then 
the linear map $I:\C\Irr(b_{\gamma})\rightarrow \C\Irr(H_{p,w})$ defined,
for $\epsilon\in\{\pm 1\}$ and $\lambda$ with $p$-core $\gamma$, by
$$I(\caralt_{\lambda}^{\epsilon})=
(-1)^{|\lambda_{p^*}|}\delta_p(\lambda)
\vartheta_{\widetilde{\lambda}^{(p)}}^{\epsilon \delta_p(\lambda)},$$
where the notation is as in Theorem~\ref{theo:BrGr},
is a generalized perfect isometry with respect to the $p$-regular
elements of $\Alt_n$ and the regular elements of $H_{p,w}$ (defined as above).
\label{theo:fh}
\end{theorem}

\begin{proof}
First, we consider the case $n=pw$. In this case, one has
$\gamma=\emptyset$.
Let $S$ and $C$ be the sets that define an MN-structure for the
principal $p$-block of $\Alt_{pw}$ with respect to the set of
$p$-regular elements of $\Alt_{pw}$. We denote by $\Omega_0$ and
$\Lambda_0$ the corresponding sets of partitions (see the proof
of Theorem~\ref{theo:mainAn}). 
Write $S'$ and $C'$ as in the proof of Theorem~\ref{theo:BrGr} (but
for elements of $H_{p,w}$). 
Then $\Lambda_0$ labels the $H_{p,w}$ classes of $S'$ by
$p\cdot\widehat{\beta}\in\Lambda_0\mapsto t_{\widehat{\beta}}$, where by
Equation~(\ref{eq:labelclassHn}), $
t_{\beta^{\pm}}= t_{(\emptyset,\ldots,\emptyset,\beta)}^{\pm}$ and 
$t_{\beta}=t_{(\emptyset,\ldots,1^{w-|\beta|},\beta)}$.
Hence, if we set $H_{\beta}=H_{p,w-|\beta|}$, then
$H_{{\beta}}$ satisfies Definition~\ref{defMN}(3). 

Now, for every partition $\lambda$ of $pw$ with trivial $p$-core, and
any $\epsilon\in\{-1,1\}$, we define
$$\widetilde{\vartheta}_{\lambda^{(p)}}^{\epsilon}=(-1)^{|\lambda_{p^*}|}\vartheta_{\widetilde{\lambda}^{(p)}}^{\epsilon}.$$

Let $p\cdot \beta\in\Omega_0$ be such that $\beta=(\beta_1,\ldots,\beta_k)$.
Assume that $\beta_k$ is odd. 	Then 
$t_{\widehat{\beta}_k}\in H_{p,w}$, and by Lemma~\ref{mn-courAn},
$(\widetilde{\vartheta}_{\lambda^{(p)}}^{+}-\widetilde{\vartheta}_{\lambda^{(p)}}^{-})(t_{\widehat{\beta}_k}g)=0$
for $g\in H_{p,w-|\beta_1|}$,
except when $(\lambda^{(p)})_{p^*}$ contains a diagonal hook $c_k$ of length
$|\beta_k|$. In this case, we have. 
%
\begin{equation}
\begin{split}
\left(\widetilde{\vartheta}_{\lambda^{(p)}}^+-\widetilde{\vartheta}_{\lambda^{(p)}}^-\right)(t_{\widehat{\beta}_k}g)&=(-1)^{|\beta_0|}
\sqrt{(-1)^{(p\beta_k-1)/2}p\beta_k}\left(\vartheta_{\lambda^{(p)}}^+-\vartheta_{\lambda^{(p)}}^-\right)(s'_{\widehat{\beta}_1}g)\\
&=-(-1)^{|\beta_0\backslash\{c_k\}|}\sqrt{(-1)^{(p\beta_k-1)/2}p\beta_k}\left(
\vartheta_{\lambda^{(p)}\backslash\{c_k\}}^+-\vartheta_{\lambda^{(p)}\backslash\{c_k\}}^-\right)(g)\\
&=\sqrt{(-1)^{(p\beta_k-1)/2}p\beta_k}\left(\widetilde{\vartheta}_{\lambda^{(p)}\backslash\{c_k\}}^--
\widetilde{\vartheta}_{\lambda^{(p)}\backslash\{c_k\}}^+\right)(g),
\end{split}
\label{eq:valdiftilde}
\end{equation}
where
$\lambda^{(p)}\backslash\{c_k\}$ is the multipartition with the same
parts as $\lambda^{(p)}$, except the $p^*$-part which is obtained from
$(\lambda^{(p)})_{p^*}$ by removing the diagonal hook of length
$\beta_k$.
Therefore, Equations~(\ref{eq:MNwr}), (\ref{eq:valdiftilde}) and Clifford theory give, for 
$\epsilon\in\{\pm 1\}$ and $g\in H_{n-|\beta_k|}$,

\begin{eqnarray*}
\widetilde{\vartheta}_{\lambda^{(p)}}^{\epsilon}(t_{\widehat{\beta}_k}
g)&=&\sum_{\mu\in M'_{\beta_k}(\lambda)\atop
\mu\neq\mu^*}b(\widetilde{\vartheta}_{\lambda^{(p)}}^{\epsilon},\widetilde{\vartheta}_{\mu^{(p)}})\,\widetilde{\vartheta}_{\mu^{(p)}}(g)
+\sum_{\mu\in M_{\beta_k}(\lambda)\atop \mu=\mu^*}\left(
b(\widetilde{\vartheta}_{\lambda^{(p)}}^{\epsilon},\widetilde{\vartheta}_{\mu^{(p)}}^+)\,\widetilde{\vartheta}_{\mu^{(p)}}^+(g)+\right.\\
&&
\left. b(\widetilde{\vartheta}_{\lambda^{(p)}}^{\epsilon},\widetilde{\vartheta}_{\mu^{(p)}}^-)\,\widetilde{\vartheta}_{\mu^{(p)}}^-(g)\right)
,\end{eqnarray*}
where $M_{\beta_k}(\lambda)$ and $M'_{\beta_k}(\lambda)$ are
defined as in~\S\ref{subsec:BroueAn}, and
the complex numbers 
$b(\widetilde{\vartheta}_{\lambda^{(p)}}^{\epsilon},\widetilde{\vartheta}_{\mu^{(p)}}^{\eta})$
satisfy the following:
\begin{enumerate}
\item[--] If $\mu^*\neq \mu$ and $\mu^*\in M_{\beta_1}(\lambda)$, then
$b(\widetilde{\vartheta}_{\lambda^{(p)}}^{\epsilon},\widetilde{\vartheta}_{\mu^{(p)}})=
\alpha(\lambda)({\alpha'}_{\mu}^{\lambda}+{\alpha'}_{\mu^*}^{\lambda})$
(see Equation~(\ref{eq:MNwr}) for the definition of
${\alpha'}_{\mu}^{\lambda}$).
\item[--] If $\mu^*\neq \mu$ and $\mu^*\notin M_{\beta_1}(\lambda)$, then
$b(\widetilde{\vartheta}_{\lambda^{(p)}}^{\epsilon},\widetilde{\vartheta}_{\mu^{(p)}})=
\alpha(\lambda){\alpha'}_{\mu}^{\lambda}$.
\item[--] If $\mu^*=\mu$ and $\mu\neq\mu_{\lambda}$, then
$b(\widetilde{\vartheta}_{\lambda^{(p)}}^{\epsilon},\widetilde{\vartheta}_{\mu^{(p)}}^{\eta})=
\alpha(\lambda){\alpha'}_{\mu}^{\lambda}$.
\item[--] If $\mu^*=\mu$ and $\mu=\mu_{\lambda}$, 
then
$b(\widetilde{\vartheta}_{\lambda^{(p)}}^{\epsilon},\widetilde{\vartheta}_{\lambda^{(p)}\backslash\{c_k\}}^{\eta})=\frac{1}{2}\left(
{\alpha'}_{\mu_{\lambda}}^{\lambda}-\eta\epsilon
\sqrt{(-1)^{(q-1)/2}q}\right)$,
where $q=p\beta_k$.
\end{enumerate}

Note that, as in the proof of Theorem~\ref{theo:BrGr}, we use that $f$
induces a bijection between $M_{p\beta_k}(\lambda)$ and
$M_{\beta_k}(\lambda^{(p)})$.

Assume now that $\beta_k$ and $\beta_{k-1}$ are even. Let $\mu\in
M_{\beta_k,\beta_{k-1}}(\lambda)$. We denote by 
$b(\widetilde{\vartheta}_{\lambda^{(p)}}^{\epsilon},\widetilde{\vartheta}_{\mu^{(p)}}^{\eta})$
the hermitian product of the class function $x\rightarrow
\widetilde{\vartheta}_{\lambda^{(p)}}^{\epsilon}(t_{\beta_k}t_{\beta_{k-1}}x)$
with
$\widetilde{\vartheta}_{\mu^{(p)}}^{\eta}\in\Z\Irr(H_{p,w-\beta_k-\beta_{k-1}})$.
Then, applying
Equation~(\ref{eq:MNwr}) twice and Clifford theory,
we obtain an analogue of Theorem~\ref{theo:MNAn2}. For $\mu\in
M_{\beta_k,\beta_{k-1}}(\lambda)$ or $\mu\in
M'_{\beta_k,\beta_{k-1}}(\lambda)$, 
the coefficient
$b(\widetilde{\vartheta}_{\lambda^{(p)}}^{\epsilon},\widetilde{\vartheta}_{\mu^{(p)}}^{\eta})$
is obtained from $a(\caralt_{\lambda}^{\epsilon},\caralt_{\mu}^{\eta})$ by
replacing $(-1)^{L(c_{\mu}^{\nu})}$ and $(-1)^{L(c_{\nu}^{\lambda})}$ by
$(-1)^{L(f(c_{\mu}^{\nu}))}$ and $(-1)^{L(f(c_{\nu}^{\lambda}))}$ respectively.

Now, as in the proof of Theorem~\ref{theo:mainAn},
if we suppose that $\beta$ is labeled such that there is some integer
$r$ with $\beta_i$ even for $i\leq r$ and $\beta_i$ odd for $i>r$,
%
then, applying
iteratively the above process, as in the proof of Theorem~\ref{theo:mainAn}, and using the fact that the
$\widetilde{\vartheta}_{\lambda^{(p)}}^{\epsilon}$'s give a basis of
$\C\Irr(H_{p,w})$, we can define a linear map 
$r^{\widehat{\beta}}:\C
H_{p,w}\rightarrow \C H_{p,w-|\beta|}$ such that
$r^{\widehat{\beta}}(\chi)(x)=\chi(t_{\widehat{\beta}}x)$ for all
$\chi\in\C\Irr(H_{p,w})$ and $x\in
H_{p,w-|\beta|}$.
In particular, $\Irr(H_{p,w})$ has an MN-structure in the sense of
Definition~\ref{defMN} with respect to $C'$.

Let $p\cdot \beta\in \Omega_0$. We define 
$I_{\widehat{\beta}}:\C \Irr(b_{\gamma}(n-p|\beta|))\rightarrow \C
H_{p,w-|\beta|}$, where $b_{\gamma}(n-p|\beta|)$ is defined in
\S\ref{subsec:pblockAn}, by setting
$$I_{\widehat{\beta}}(\caralt_{\mu}^{\eta})=(-1)^{|\mu_{p^*}|}\delta_p(\mu)
\widetilde{\vartheta}_{\widetilde{\mu}^{(p)}}^{\eta\delta_p(\mu)(-1)^{\ell(\beta)}},$$
where $\eta\in\{\pm 1\}$ and $\mu$ is a partition of $p(w-|\beta|)$ with
$p$-core $\gamma$.
Note that $I_{\{1\}}=I$.

Write
$b(\widetilde{\vartheta}_{\lambda^{(p)}}^{\epsilon},\widetilde{\vartheta}_{\mu^{(p)}}^{\eta})=\langle
r^{\widehat{\beta}}(\widetilde{\vartheta}_{\lambda^{(p)}}^{\epsilon}),\widetilde{\vartheta}_{\mu^{(p)}}^{\eta}\rangle_{H_{p,w-|\beta|}}$.
If either $\widehat{\beta}=\beta$ or
$\widehat{\beta}=\beta^{\pm}$ and $\lambda\neq \kappa$ (where $\kappa$
is the partition defined in the proof of Theorem~\ref{theo:mainAn}), then a
straightforward computation (see the proofs of Theorem~\ref{theo:mainAn}
and of Theorem~\ref{theo:BrGr}) gives 
$$a\left(\caralt_{\lambda}^{\epsilon},\caralt_{\mu}^{\eta}\right)=b\left(
I\left(\widetilde{\vartheta}_{\lambda^{(p)}}^{\epsilon}\right),I_{\beta}
\left(\widetilde{\vartheta}_{\mu^{(p)}}^{\eta}\right)\right).$$ 

Hence, the only case to consider is 
$\beta=(\beta_1,\ldots,\beta_k)\in \mathcal{OD}_w$ and
$\lambda=\kappa$. Write $(h_1,\ldots,h_k)$ for the diagonal
hooks of $\kappa$ and assume that the hook length of $h_i$ is $p\beta_i$.
Furthermore, define $\beta(0)=\{1\}$ and
$\beta(i)=\{\beta_1,\ldots,\beta_i\}$ for $1\leq i\leq k$ (in particular,
$p\cdot\beta(i)\in\Omega_0$). Note that
$\ell(\beta(i))=i$.

Let $1\leq i\leq k$. Write
$\nu=\kappa\backslash\{h_1,\ldots,h_{i-1}\}$ and
$\mu=\kappa\backslash\{h_1,\ldots,h_{i}\}$. Therefore, if we set
$q=ph_i$, then we have
\begin{align*}
\begin{split}
b\left(I_{\widehat{\beta}(i-1)}(\rho_{\nu}^{\epsilon}),
I_{\widehat{\beta}(i)}(\rho_{\mu}^{\eta})\right)&=
\delta_p(\nu)\delta_p(\mu)b\left(\widetilde{\vartheta}_{\nu^{(p)}
}^{\epsilon\delta_p(\nu)(-1)^{i-1}},
\widetilde{\vartheta}_{\mu^{(p)}
}^{\eta\delta_p(\mu)(-1)^{i}}\right),\\
&=\delta_p(\nu)\delta_p(\mu)\left({\alpha'}_{\mu}^{\nu}-\epsilon\eta\delta_p(\nu)\delta_p(\mu)(-1)^{2i-1}\sqrt{(-1)^{(q-1)/2}q}\right),\\
&=\left(\alpha_{\mu}^{\nu}+\epsilon\eta\sqrt{(-1)^{(q-1)/2}q}\right),\\
&=a\left(\caralt_{\nu}^{\epsilon},\caralt_{\mu}^{\eta}\right).
\end{split}
\end{align*}
Thus, using an argument similar to Equations~(\ref{eq:coeffAitere})
and~(\ref{eq:acommute}), we conclude that
$b(I(\caralt_{\kappa}^{\epsilon}),I_{\widehat\beta}(\caralt_{\mu}^{\eta}))=a(\caralt_{\kappa}^{\epsilon},\caralt_{\mu}^{\eta})$.
It follows that
$$r^{\widehat{\beta}}\circ I= I_{\widehat\beta}\circ
r^{\widehat{\beta}}$$ for every $p\cdot\widehat{\beta}\in \Lambda_0$,
and Corollary~\ref{cor:isoparfaitegene} gives the result.

Now we return to the general case, that is, $\gamma$ is any self-conjugate
$p$-core of $n$ with $p$-weight $w$. Let $b'$ be the principal $p$-block
of $\Alt_{pw}$. We consider
$I_n:\C\Irr(b_{\gamma})\rightarrow \C\Irr(b')$ the perfect isometry
obtained in Theorem~\ref{theo:mainAn} and
$I_{pw}:\C\Irr(b')\rightarrow\C\Irr(H_{p,w})$ the perfect isometry
obtained in the first part of the proof.
Then $I_{pw}\circ I_n:\C\Irr(b_{\gamma})\rightarrow\C\Irr(H_{p,w})$
is a perfect isometry. In order to prove the result, it is sufficient to
show that $I=I_{pw}\circ I_n$.
Let $\lambda\in\mathcal P_n$ be such that $\lambda^{(p)}=\gamma$ and
$\epsilon\in\{\pm 1\}$. Then using that the $p$-quotient of
$\Psi(\lambda)$ is $\lambda^{(p)}$, we derive that
\begin{eqnarray*}
I_{pw}\circ I_n(\caralt_{\lambda}^{\epsilon})&=&
I_{pw}\left(\delta_p(\lambda)\delta_p(\Psi(\lambda))\rho_{\Psi(\lambda)}^{\epsilon\delta_p(\lambda)\delta_p(\Psi(\lambda))}\right),\\
&=&\delta_p(\lambda)\delta_p(\Psi(\lambda))\delta_p(\Psi(\lambda))
\widetilde{\vartheta}_{\lambda^{(p)}}^{\varepsilon\delta_p(\lambda)\delta_p(\Psi(\lambda))\delta_p(\Psi(\lambda))},\\
&=&\delta_p(\lambda)\widetilde{\vartheta}_{\lambda^{(p)}}^{\varepsilon\delta_p(\lambda)},\\
&=&I(\rho_{\lambda}^{\epsilon}),
\end{eqnarray*}
as required.
\end{proof}

\begin{corollary}\label{cor:fh}
With the assumptions of Theorem~\ref{theo:fh}, and if furthermore $w<p$, then 
$I$ is a Brou\'e perfect isometry.
\end{corollary}

{\bf Acknowledgements.} The authors wish to thank the referee for a very
careful and precise reading of several earlier versions of this manuscript.
They are grateful for the number, quality and helpfulness of comments
and suggestions received.

 \bibliographystyle{abbrv}
\bibliography{references}

\end{document}